\documentclass[10pt,a4paper]{article}
\usepackage{amssymb}

\usepackage{graphicx,amssymb,amsfonts,epsfig,amsthm,a4,amsmath,url}
\usepackage[latin1]{inputenc}

\usepackage{latexsym}
\usepackage{amsmath,amsthm}
\usepackage{amsfonts}
\usepackage{epsfig}
\usepackage{amssymb}
\usepackage{amscd}
\usepackage[all]{xy}
\usepackage{stmaryrd}
\input xy
\usepackage{hyperref}

\usepackage{enumitem}

\newtheorem{Thm}{Theorem}[section]
\newtheorem{Prop}[Thm]{Proposition}
\newtheorem{Lem}[Thm]{Lemma}
\newtheorem{Cor}[Thm]{Corollary}
\newtheorem{Ex}[Thm]{Example}
\theoremstyle{definition}
\newtheorem{Def}[Thm]{Definition}

\newtheorem{Rem}[Thm]{Remark}
\newtheorem{Conj}{Conjecture}

\newcommand{\Z}{\mathbf{Z}}

\newcommand{\N}{\mathbf{N}}
\newcommand{\R}{\mathbf{R}}
\newcommand{\C}{\mathbf{C}}
\newcommand{\Q}{\mathbf{Q}}

\newcommand{\F}{\mathbf{F}}

\newcommand{\HH}{\mathcal{H}}

\newcommand{\G}{\rtimes^{\rho} G}

\numberwithin{equation}{section}

\newcommand{\Group}{\mathcal{G}} 

\newcommand{\norme}[1]{\left\Vert #1\right\Vert}
\newcommand{\End}{\mathrm{End}}
\newcommand{\A}{\mathcal{A}^{\rho}_r(G)}

\title{The Baum-Connes conjecture: an extended survey}
\author{Maria Paula GOMEZ APARICIO, Pierre JULG, and Alain VALETTE}
\date{\today}

\begin{document}

\baselineskip=16pt

\maketitle
\hfill{\it To Alain Connes, for providing lifelong inspiration}


\tableofcontents

\section{Introduction}

\subsection{Building bridges}

Noncommutative Geometry is a field of Mathematics which builds bridges between many different subjects. Operator algebras, index theory, K-theory, geometry of foliations, group representation theory are, among others, ingredients of the impressive achievements of Alain Connes and of the many mathematicians that he has inspired in the last 40 years. 

At the end of the 1970's the work of Alain Connes on von Neumann theory naturally led him to explore foliations and groups. His generalizations of Atiyah's $L^2$ index theorem were the starting point of his ambitious project of Noncommutative Geometry. A crucial role has been played by the pioneering conference in Kingston in July 1980, where he met the topologist Paul Baum. The picture of what was soon going to be known as the Baum-Connes conjecture quickly emerged. The catalytic effect of IHES should not be underestimated; indeed the paper \cite{BC} was for a long time available only as an IHES 1982 preprint. It is only in 1994 that the general and precise statement was given in the proceedings paper \cite{BCH} with Nigel Higson.

\subsection{In a nutshell - without coefficients...}

The Baum-Connes conjecture also builds a bridge, between commutative geometry and non-commutative geometry. Although it may be interesting to formulate the conjecture for locally compact groupoids\footnote{This is important e.g. for applications to foliations, see Chapter 7.}, we stick to the well-accepted tradition of formulating the conjecture for locally compact, second countable groups.

{\bf For every locally compact group $G$ there is a Baum-Connes conjecture!}. We start by associating to $G$ four abelian groups $K_*^{top}(G)$ and $K_*(C^*_r(G))$ (with $*=0,1$), then we construct a group homomorphism, the {\it assembly map}:
$$\mu_r:K_*^{top}(G)\rightarrow K_*(C^*_r(G))\;\;\;(*=0,1).$$
We say that {\bf the Baum-Connes conjecture holds for $G$ if $\mu_r$ is an isomorphism for $*=0,1$}. Let us give a rough idea of the objects.

\begin{itemize}
\item The RHS of the conjecture, $K_*(C^*_r(G))$, is called the {\it analytical side}: it belongs to noncommutative geometry. Here $C^*_r(G)$, the {\it reduced $C^*$-algebra of $G$}, is the closure in the operator norm of $L^1(G)$ acting by left convolution on $L^2(G)$; and $K_*(C^*_r(G))$ is its {\it topological K-theory}.

Topological K-theory is a homology theory for Banach algebras $A$, enjoying the special feature of Bott periodicity ($K_i(A)$ is naturally isomorphic to $K_{i+2}(A)$), so that there are just two groups to consider: $K_0$ and $K_1$. K-theory conquered $C^*$-algebra theory around 1980, as a powerful invariant to distinguish $C^*$-algebras up to isomorphism. A first success was, in the case of the free group $\mathbf{F}_n$ of rank $n$, the computation of $K_*(C^*_r(\mathbf{F}_n))$ by Pimsner and Voiculescu \cite{PimsnerVoiculescu}: they obtained 
$$K_0(C^*_r(\mathbf{F}_n))=\Z,\;\;K_1(C^*_r(\mathbf{F}_n))=\Z^n,$$
so that $K_1$ distinguishes reduced $C^*$-algebras of free groups of various ranks.

For many connected Lie groups (e.g. semisimple), $C^*_r(G)$ is type I, which points to using {\it d\'evissage} techniques: representation theory allows to define ideals and quotients of $C^*_r(G)$ that are less complicated, so $K_*(C^*_r(G))$ can be computed by means of the 6-terms exact sequence associated with a short exact sequence of Banach algebras. By way of contrast, if $G$ is discrete, $C^*_r(G)$ is very often {\it simple} (see \cite{BKKO} for recent progress on that question); it that case, d\'evissage must be replaced by brain power (see \cite{Pimsner86} for a sample), and the Baum-Connes conjecture at least provides a conjectural description of what $K_*(C^*_r(G))$ should be (see e.g. \cite{SanchezGarcia}). 

\item The LHS of the conjecture, $K_*^{top}(G)$, is called the {\it geometric}, or {\it topological} side. This is actually misleading, as its definition is awfully analytic, involving Kasparov's bivariant theory (see Chapter 3). A better terminology would be the {\it commutative side}, as indeed it involves a space $\underline{EG}$, the {\it classifiying space for proper actions of $G$} (see Chapter 4), and $K_*^{top}(G)$ is the $G$-equivariant K-homology of $\underline{EG}$. 

When $G$ is discrete and torsion-free, then $\underline{EG}=EG=\widetilde{BG}$, the universal cover of the classifying space $BG$. As $G$ acts freely on $EG$, the $G$-equivariant K-homology of $EG$ is $K_*(BG)$, the ordinary K-homology of $BG$, where K-homology for spaces can be defined as the homology theory dual to topological K-theory for spaces.

\item The assembly map $\mu_r$ will be defined in Chapter 4 using Kasparov's equivariant KK-theory. Let us only give here a flavor of the meaning of this map. It was discovered in the late 1970's and early 1980's that the K-theory group $K_*(C^*_r(G))$ is a receptacle for indices, see section \ref{ConMosco}. More precisely, if $M$ is a smooth manifold with a proper action of $G$ and compact quotient, and $D$ an elliptic $G$-invariant differential operator on $M$, then $D$ has an index $ind_G(D)$ living in $K_*(C^*_r(G))$. Therefore, the geometric group $K_*^{top}(G)$ should be thought of as the set of homotopy classes of such pairs $(M,D)$, and the assembly map $\mu_r$ maps the class $[(M,D)]$ to $ind_G(D)\in K_*(C^*_r(G))$.

\end{itemize}

\subsection{... and with coefficients}
There is also a more general conjecture, called {\it the Baum-Connes conjecture with coefficients}, where we allow $G$ to act by *-automorphisms on an auxiliary $C^*$-algebra $A$ (which becomes a $G-C^*$-algebra), and where the aim is to compute the K-theory of the reduced crossed product $C^*_r(G,A)$. One defines then the assembly map
$$\mu_{A,r}:K_*^{top}(G,A)\rightarrow K_*(C^*_r(G,A))\;\;\;(*=0,1)$$
and we say that {\bf the Baum-Connes conjecture with coefficients holds for $G$ if $\mu_{A,r}$ is an isomorphism for $*=0,1$ and every $G-C^*$-algebra $A$}. The advantage of the conjecture with coefficients is that it is inherited by closed subgroups; its disadvantage is that it is false in general, see Chapter 9.

\subsection{Structure of these notes}\label{Summary}

Using the acronym BC for ``{\it Baum-Connes conjecture}'', here is what the reader will find in this piece.

\begin{itemize}
\item Where does BC come from? Chapter 2, on the history of the conjecture.
\item What are the technical tools and techniques? Chapter 3, on Kasparov theory (and the Dirac - dual Dirac method).
\item What is BC, what does it entail, what is the state of the art? Chapter 4.
\item Why is BC difficult? Chapter 5, discussing BC with coefficients for semisimple Lie groups and their closed (e.g. discrete) subgroups.
\item How can we hope to overcome those difficulties? Chapter 6, on Banach algebraic methods.
\item Is BC true or false? For BC without coefficients we don't know, but we know that the natural extension of BC from groups to groupoids is false (see Chapter 7), and we know that BC with coefficients is false (see Chapter 9).
\end{itemize}
We could have stopped there. But it seemed unfortunate not to mention an important avatar of BC, namely the {\it coarse Baum-Connes conjecture} (CBC) due to the late John Roe: roughly speaking, groups are replaced by metric spaces, see Chapter 8. An important link with the usual BC is that for a finitely generated group, which can be viewed as a metric space via some Cayley graph, CBC implies the injectivity part of BC. 

Finally, it was crucial to mention the amount of beautiful mathematics generated by BC, and this is done in Chapter 9. 

\subsection{What do we know in 2019?}\label{known}

In Chapter \ref{indexmaps} we explain the ``{\it Dirac - dual Dirac}'' method used by Kasparov \cite{KaspConsp} to prove the injectivity of $\mu_{A,r}$ for all semisimple Lie group $G$ and all $G$-$C^*$-algebras $A$; this also proves injectivity for closed subgroups of a semisimple Lie group, as this property passes to closed subgroups.  Since then, an abstraction of the Dirac - dual Dirac method, explained in section \ref{viveTu}, has been used by Kasparov and Skandalis \cite{Kasparov-Skandalis03}, to prove the injectivity of the assembly map for a large class of groups denoted by $\mathcal{C}$ in \cite{LaffInv}. This class contains, for example all locally compact groups acting continuously, properly and isometrically on a complete and simply connected Riemannian manifold  of non-positive scalar curvature (see \cite{Kasparov88}), or on a Bruhat-Tits affine building (for example all $p$-adic groups, see \cite{Kasparov-Skandalis91}), all hyperbolic groups (see \cite{Kasparov-Skandalis03}). So the injectivity of the Baum-Connes assembly map has been proven for a huge class of groups.

The conjecture with coefficients has been proven for a large class of groups that includes all groups with the Haagerup property (eg. $SL_2(\R)$, $SO_0(n,1)$, $SU(n,1)$, and all free groups). For those groups the proof is due to Higson and Kasparov (see \cite{HigKas}) and it is also based on the "Dirac-dual Dirac" method. This method cannot however be applied to non-compact groups having property (T), not even for the conjecture without coefficients: see section \ref{KvsH} for more on the tension between the Haagerup and Kazhdan properties.

Nevertheless, as will be explained in section \ref{Laff}, Lafforgue managed to prove the conjecture without coefficients for all semisimple Lie groups and for some of their discrete subgroups, precisely those having property (RD) (as defined in section \ref{(RD)}). For example, the conjecture without coefficients is true for all cocompact lattices in $SL_3(\R)$ but it is still open for $SL_3(\mathbf{Z})$\footnote{In the case of $SL_3(\Z)$, surjectivity of $\mu_r$ is the open problem; the LHS of the Baum-Connes conjecture was computed in \cite{SanchezGarcia}.}.

On the other hand, the conjecture with coefficients has been proven for all hyperbolic groups (see \cite{LaffHyp}), but it still open for higher rank semisimple Lie groups and their closed subgroups: see sections \ref{Trichotomy} and \ref{hyperLafforgue} for more on that.

An example of a group for which, at the time of writing, $\mu_r$ is not known to be either injective or surjective, is the free Burnside group $B(d,n)$, as soon as it is infinite\footnote{Recall that $B(d,n)$ is defined as the quotient of the non-abelian free group $\F_d$ by the normal subgroup generated by all $n$'s powers in $\F_d$.}.

\subsection{A great conjecture?}

What makes a conjecture great? Here we should of course avoid the chicken-and-egg answer ``{\it It's a great conjecture because it is due to great mathematicians}''. We should also be suspicious of the pure maths self-referential answer: ``{\it It's a great conjecture because it implies several previous conjectures}'': that an abstruse conjecture implies even more abstruse ones\footnote{Compare with sections \ref{Novikov} and \ref{BCconsequences}.}, does not necessarily make it great.

We believe that the interest of a conjecture lies in the feeling of unity of mathematics that it entails. We hope that the reader, in particular the young expert, after glancing at the table of contents and the various subjects listed in section \ref{prerequis} below, will not let her/himself be discouraged. Rather (s)he should take this as an incentive to learn new mathematics, and most importantly connections between them.

Judging by the amount of fields that it helps bridging (representation theory, geometric group theory, metric geometry, dynamics,...), we are convinced that yes, the Baum-Connes conjecture is indeed a great conjecture.

\subsection{Which mathematics are needed?}\label{prerequis}
We use freely the following concepts; for each we indicate one standard reference.
\begin{itemize}
\item locally compact groups (Haar measure, unitary representations): see \cite{Dix};
\item semisimple Lie groups and symmetric spaces: see \cite{Helgason};
\item operator algebras (full and reduced group $C^*$-algebras, full and reduced crossed products): see \cite{Pedersen};
\item K-theory for $C^*$-algebras (Bott periodicity, 6-terms exact sequences, Morita equivalence): see \cite{WeggeOlsen};
\item index theory: see \cite{BoossBleecker}
\end{itemize}

\bigskip
{\bf Acknowledgements}: Thanks are due to J.-B. Bost, R. Coulon, N. Higson, V. Lafforgue , P.-Y. Le Gall, H. Oyono-Oyono, N. Ozawa and M. de la Salle for useful conversations and exchanges.

\section{Birth of a conjecture}

\subsection{Elliptic (pseudo-) differential operators}

Let $M$ be a closed manifold, and let $D$ be a(pseudo-) differential operator acting on smooth sections of some vector bundles $E,F$ over $M$, so $D$ maps $C^\infty(E)$ to $C^\infty(F)$. Let $T^*M$ denote the cotangent bundle of $M$. The (principal) symbol is a bundle map $\sigma(D)$ from the pullback of $E$ to the pullback of $F$ on $T^*M$. Recall that $D$ is said to be {\it elliptic} if $\sigma(D)$ is invertible outside of the zero section of $T^*M$. In this case standard elliptic theory guarantees that $ker(D)$ and $coker(D)$ are finite-dimensional, so that the (Fredholm) index of $D$ is defined as 
$$Ind(D)=\dim_\C \mathrm{ker} (D) -\dim_\C \mathrm{coker}(D) \in \Z.$$
The celebrated Atiyah-Singer theorem \cite{AtiSin} then provides a topological formula for $Ind(D)$ in terms of topological invariants associated with $M$ and $\sigma(D)$.

Now let $\tilde{M}\rightarrow M$ be a Galois covering of $M$, with group $\Gamma$, so that $M=\Gamma\backslash\tilde{M}$. Assume that $D$ lifts to a $\Gamma$-invariant operator $\tilde{D}$ on $\tilde{M}$, between smooth sections of $\tilde{E}, \tilde{F}$, the vector bundles pulled back from $E,F$ via the covering map. 

\begin{itemize}
\item Assume first that $\Gamma$ is finite, i.e. our covering has $n=|\Gamma|$ sheets. Then $\tilde{M}$ is a closed manifold, and the index of $\tilde{D}$ satisfies $Ind(\tilde{D})=n\cdot Ind(D)$. Now we may observe that, in this case, there is a more refined analytical index, obtained by observing that $ker(\tilde{D})$ and $coker(\tilde{D})$ are finite-dimensional representation spaces of $\Gamma$, hence their formal difference makes sense in the additive group of the representation ring $R(\Gamma)$: we get an element $\Gamma-Ind(\tilde{D})\in R(\Gamma)$; the character of this virtual representation of $\Gamma$, evaluated at $1\in \Gamma$, gives precisely $Ind(\tilde{D})$.

\item Assume now that $\Gamma$ is infinite. Then the $L^2$-kernel and $L^2$-cokernel of $\tilde{D}$ are closed subspaces of the suitable space of $L^2$-sections, namely $L^2(\tilde{M},\tilde{E})$ and $L^2(\tilde{M},\tilde{F})$, and by $\Gamma$-invariance those spaces are representation spaces of $\Gamma$. The problem with these representations is that their classical dimension is infinite. Atiyah's idea in \cite{AtiyahL2} is to measure the size of these spaces via the dimension theory of von Neumann algebras.

More precisely, the $L^2$-kernel of $\tilde{D}$ is $\Gamma$-invariant, so the orthogonal projection onto that kernel belongs to the algebra $A$ of operators commuting with the natural $\Gamma$-representation on $L^2(\tilde{M}, \tilde{E})$. Choosing a fundamental domain for the $\Gamma$-action on $\tilde{M}$ allows to identify $\Gamma$-equivariantly $L^2(\tilde{M},\tilde{E})$ with $\ell^2(\Gamma)\otimes L^2(M,E)$. So $A$ becomes the von Neumann algebra $L(\Gamma)\otimes \mathcal{B}(L^2(M,E))$, where $L(\Gamma)$, the {\it group von Neumann algebra} of $\Gamma$, is generated by the right regular representation of $\Gamma$ on $\ell^2(\Gamma)$. The canonical trace on $L(\Gamma)$ (defined by $\tau(a)=\langle a(\delta_e), \delta_e\rangle$ for $a\in L(\Gamma)$) provides a dimension function $\dim_\Gamma$ on the projections in $A$. Atiyah's {\it $L^2$-index theorem} \cite{AtiyahL2} states that 
\begin{Thm}\label{L2index} In the situation above:
 $$Ind(D)=Ind_\Gamma(\tilde{D}),$$
where the right-hand side is defined as $$Ind_\Gamma(\tilde{D}):=\dim_\Gamma( \mathrm{ker}\tilde{D}))-\dim_\Gamma(\mathrm{coker}(\tilde{D})).$$
\end{Thm}
\end{itemize}

\subsection{Square-integrable representations}

Recall that, for $G$ a locally compact unimodular group, a unitary irreducible representation $\pi$ of $G$ is said to be {\it square-integrable} if, for every two vectors $\xi,\eta$ in the Hilbert space of the representation $\pi$, the coefficient function
$$g\mapsto \langle \pi(g)\xi, \eta\rangle$$
is square-integrable on $G$. Equivalently, $\pi$ is a sub-representation of the left regular representation $\lambda_G$ of $G$ on $L^2(G)$ (see \cite{Dix}, section 14.1, for the equivalence). The set of square-integrable representations of $G$ is called the {\it discrete series} of $G$.

When $G$ is a semisimple Lie group with finite center, we denote by $\hat{G}_r$ the {\it reduced dual}, or {\it tempered dual} of $G$: this is the set of (equivalence classes of) unitary irreducible representations of $G$ weakly contained in $\lambda_G$; it may also be defined as the support of the Plancherel measure on the full dual $\hat{G}$ of $G$. A cornerstone of 20th century mathematics is Harish-Chandra's explicit description of the Plancherel measure on semisimple Lie groups, and it turns out that the discrete series of $G$ is exactly the set of atoms of the Plancherel measure.

Let us be more specific. Let $K$ be a maximal compact subgroup of $G$, a connected semisimple Lie group with finite center. A first result of Harish-Chandra states that the discrete series of $G$ is non-empty if and only if $G$ and $K$ have equal rank. This exactly means that a maximal torus of $K$ is also a maximal torus of $G$. Let us assume that this holds, and let us fix a maximal torus $T$ in $K$. Let $\mathfrak{g}_\C,\mathfrak{k}_\C,\mathfrak{t}_\C$ be the complexified Lie algebras of $G,K,T$ respectively. Decomposing the adjoint representations of $T$ on $\mathfrak{k}_\C$ and $\mathfrak{g}_\C$ respectively, we get two root systems $\Phi_c$ and $\Phi$, with $\Phi_c\subset \Phi$: we say that $\Phi$ is the set of roots, while $\Phi_c$ is the set of {\it compact} roots. Correspondingly there are two Weyl groups $W(K)\subset W$. We denote by $\Lambda$ the lattice of weights of $T$. An element of $\mathfrak{t}_\C$ is {\it regular} if its stabilizer in $W$ is trivial. We denote by $\rho$ half the sum of positive roots in $\Phi$ (with respect to a fixed set $\Psi$ of positive roots), and by $\rho_c$ half the sum of the positive compact roots. We have then Harish-Chandra's main result on existence and exhaustion of discrete series (see \cite{Lips}, section I.B.2 for a nice summary of Harish-Chandra's theory):

\begin{Thm}\label{Harish} To each regular element $\lambda\in\Lambda+\rho$ is naturally associated a square-integrable irreducible representation $\pi_\lambda$ of $G$ such that $\pi_\lambda|_K$ contains with multiplicity 1 the $K$-type with highest weight $\lambda+\rho-2\rho_c$. Every discrete series representation of $G$ appears in this way. If $\lambda,\mu\in \Lambda+\rho$, the representations $\pi_\lambda, \pi_\mu$ are unitarily equivalent if and only if $\lambda$ and $\mu$ are in the same $W(K)$-orbit.
\end{Thm}

Impressive as it is, Theorem \ref{Harish} left open the question of constructing geometrically the discrete series representations $\pi_\lambda$. That question was solved by Atiyah and Schmid \cite{AtiSch}. Assume that $G$ has discrete series representations, which forces the symmetric space $G/K$ to be even-dimensional. Assume moreover that $G/K$ carries a $G$-invariant spin structure, meaning that the isotropy representation of $K$ on $V :=\mathfrak{g}/\mathfrak{k}$ lifts to the Spin group of $V$; this can be ensured by replacing $G$ by a suitable double cover. Then we have the two irreducible spinor representations $S^+,S^-$ of $Spin(V)$, that we view as $K$-representations\footnote{$G/K$ carries a $G$-invariant spin structure if and only if $\rho-\rho_c\in\Lambda$, see \cite[4.34]{AtiSch}; the distinction between $S^+$ and $S^-$ is made by requiring that $\rho-\rho_c$ is the highest weight for $S^+$, see \cite[3.13]{AtiSch}.}. Fix a regular element $\lambda$ in $\Lambda+\rho$; conjugating $\Psi$ by some element of $W$, we may assume that $\lambda$ is dominating for $\Psi$. Then $\mu :=\lambda-\rho_c\in\Lambda$ is a weight dominating for $\Phi_c\cap\Psi$, and we denote by $E_\mu$ the irreducible representation of $K$ with highest weight $\mu$. Form the G-equivariant induced vector bundles $G\times_K(E_\mu\otimes S^\pm)$ over $G/K$, and let 
$$D_\mu: C^\infty(G\times_K(E_\mu\otimes S^+))\rightarrow C^\infty(G\times_K(E_\mu\otimes S^-))$$ 
be the corresponding Dirac operator with coefficients in $\mu$. The main result of Atiyah and Schmid (see \cite[9.3]{AtiSch}) is then: 

\begin{Thm}\label{AtiyahSchmid} Let $\lambda\in\Lambda+\rho$ be regular, with $\lambda=\mu+\rho_c$ as above. Then $\mathrm{coker}(D_\mu^+)=0$ and the $G$-representation on $\mathrm{ker}(D_\mu^+)$ is the discrete series representation $\pi_\lambda$. If $\lambda$ is not regular, then $\mathrm{ker}(D_\mu^+)=\mathrm{coker}(D_\mu^+)=0$.
\end{Thm}

It is interesting to observe that Atiyah's $L^2$-index theorem plays a role in the proof, as the authors need a torsion-free cocompact lattice $\Gamma$ in $G$ and apply the $L^2$-index theorem to the covering of the compact manifold $\Gamma\backslash G/K$ by $G/K$.

To summarize, {\it Dirac induction} (i.e. realizing $G$-representations by means of Dirac operators with coefficients in $K$-representations) sets up a bijection between a generic set of irreducible representations of $K$, and all square-integrable representations of $G$. Suitably interpreted using K-theory of $C^*$-algebras, this principle paved the way towards the Connes-Kasparov conjecture, which was the first form of the Baum-Connes conjecture.

\subsection{Enters K-theory for group $C^*$-algebras}\label{ConMosco}

The Atiyah-Schmid construction of the discrete series, served as a crucial motivation for Connes and Moscovici \cite{ConMos} in their study of the $G$-index for $G$-equivariant elliptic differential operators $D$ on homogeneous spaces of the form $G/K$, where $G$ is a unimodular Lie group with countably many connected components, and $K$ is a compact subgroup. Their aim is to define the $G$-index of $D$ intrinsically, i.e. without appealing to Atiyah's $L^2$-index theory (so, not needing an auxiliary cocompact lattice in $G$): $D$ will not be Fredholm in the usual sense (unless $G$ is compact), but $\mathrm{ker}(D)$ and $\mathrm{coker}(D)$ will have finite $G$-dimension in the sense of the Plancherel measure on $\hat{G}_r$. The formal difference of these, the $G$-index of $D$, is a real number shown to depend only on the class $[\sigma(D)]$ of the symbol of $D$ in K$_K(V^*)$, where K$_K$ denotes equivariant K-theory with compact supports and $V^*$ is the cotangent space to $G/K$ at the origin. This $G$-index is computed in terms of the symbol of $D$, and this index formula is used to prove that $\ker(D)$ is a finite direct sum of square-integrable representations of $G$.

 Crucial for our story is the final section of \cite{ConMos}. Indeed, there Connes and Moscovici sketch the construction of an index taking values in K$_*(C^*_r(G))$, the topological K-theory of the reduced $C^*$-algebra of $G$. It goes as follows: let $\rho$ be a finite-dimensional unitary representation of $K$ on $H_\rho$, form the induced vector bundle $E_\rho := G\times_K H_\rho$ over $G/K$. Denote by $\Psi^*_G(G/K,E_\rho)$ be the norm closure of the space of 0-th order $G$-invariant pseudo-differential operators on $G/K$ acting on sections of $E_\rho$: since such an operator acts by bounded operators on $L^2(G/K,E_\rho)$, we see that $\Psi^*_G(G/K,E_\rho)$ is a $C^*$-algebra on $L^2(G/K,E_\rho)$. The symbol map induces a $*$-homomorphism $\Psi^*_G(G/K,E_\rho)\rightarrow C_K(S(V^*), \mathcal{B}(H_\rho))$, where  the latter is the algebra of $K$-invariant, $\mathcal{B}(H_\rho)$-valued continuous functions on $S(V^*)$, the unit sphere in $V^*$. It fits into a short exact sequence
\begin{equation}\label{pseudo}
0\rightarrow C^*_G(G/K,E_\rho)\rightarrow \Psi^*_G(G/K,E_\rho)\rightarrow C_K(S(V^*), \mathcal{B}(H_\rho))\rightarrow 0,
\end{equation}
where the kernel $C^*_G(G/K,E_\rho)$ is the norm closure of $G$-invariant regularizing operators on $G/K$. When $\rho$ is the left regular representation of $K$, Connes and Moscovici observe that $C^*_G(G/K,E_\rho)$ is canonically isomorphic to the reduced $C^*$-algebra $C^*_r(G)$ of $G$. If $D\in \Psi^*_G(G/K,E_\rho)$ is elliptic, then its symbol is invertible in $C_K(S(V^*), \mathcal{B}(H_\rho))$, so defines an element $[\sigma(D)]\in K_1(C_K(S(V^*)))$. The short exact sequence (\ref{pseudo}) defines a 6-terms exact sequence in K-theory, and the connecting map $K_1(C_K(S(V^*)))\rightarrow K_0(C^*_r(G))$ allows to define $ind_G(D)\in K_0(C^*_r(G))$. So the K-theory $K_*(C^*_r(G))$ appears as a receptacle for indices of $G$-invariant elliptic pseudo-differential operators on manifolds of the form $G/K$, with $K$ compact.

We quote the final lines of \cite{ConMos}: ``{\it Of course, to obtain a valuable formula for the index map $ind_G$, one first has
to compute $K_0(C^*_r(G))$. When $G$ is simply connected and solvable, it follows from the Thom isomorphism in \cite{ConThom} that $K_i(C^*(G)) \simeq K^{i+j}(point), i,j\in\Z_2$, where $j$ is the dimension $\mod 2$ of $G$. The computation of the K-theory of $C^*(G)$ for an arbitrary Lie group $G$ and the search for an "intrinsic" index formula certainly deserve further study.}'' This served as a research programme for the following years!\footnote{We believe that Connes and Moscovici actually had $C^*_r(G)$, not $C^*(G)$, in mind when writing this.}

Let us end this section by mentioning that, since the framework in \cite{ConMos} is unimodular Lie groups with countably many connected components, it applies in particular to countable discrete groups $\Gamma$. In this case the canonical trace $\tau: C^*_r(\Gamma)\rightarrow\C$ defines a homomorphism $\tau_*:K_0(C^*_r(\Gamma))\rightarrow\R$, and $\tau_*(ind_\Gamma(D))=Ind_\Gamma(D)$, the $\Gamma$-index of $D$ as in (\ref{L2index}).

\subsection{The Connes-Kasparov conjecture}

Disclaimer: the Connes-Kasparov conjecture is not a conjecture anymore since 2003! After proofs of several particular cases, starting with the case of simply connected solvable groups established by Connes \cite{ConThom}, and the cornerstone of semisimple groups being established first by Wassermann \cite{Wassermann87} by representation-theoretic methods then by Lafforgue \cite{LaffInv} by geometric/analytical techniques, the general case was handled by Chabert-Echterhoff-Nest \cite{ChEcNe} building on Lafforgue's method. Nevertheless the Connes-Kasparov conjecture was fundamental for the later formulation of the more general Baum-Connes conjecture.

Let $G$ be a connected Lie group, and let $K$ be a maximal compact subgroup (it follows from structure theory that $K$ is unique up to conjugation). Set $V=\mathfrak{g}/\mathfrak{k}$; assume that $G/K$ carries a $G$-invariant spin structure, i.e. that the adjoint representation of $K$ on $V$ lifts to $Spin(V)$. Let $S^+,S^-$ be the spinor representations of $Spin(V)$ (with the convention $S^+=S^-$ if $j=\dim G/K$ is odd), that we view as $K$-representations. Let $\rho$ be a finite-dimensional representation of $K$, form the induced $G$-vector bundles $E^\pm_\rho= G\times_K (\rho\otimes S^\pm)$. Let $D_\rho:C^\infty(E^+_\rho)\rightarrow C^\infty(E^-_\rho)$ be the corresponding Dirac operator. Let $R(K)$ be the representation ring of $K$. Thanks to the previous section, we may define the {\it Dirac induction}
$$\mu_G:R(K)\rightarrow K_j(C^*_r(G)):\rho\mapsto ind_G(D^+_\rho),$$
a homomorphism of abelian groups. The {\it Connes-Kasparov conjecture} (see \cite{BC}, section 5;  \cite{Kasp}; \cite{KaspConsp}, Conjecture 1) is the following statement:

\begin{Conj}\label{ConKas}(1st version) Let $G$ be a connected Lie group, $K$ a maximal compact subgroup, $j=\dim(G/K)$. Assume that $G/K$ carries a $G$-invariant spin structure.
\begin{enumerate}
\item[1)] The Dirac induction $\mu_G:R(K)\rightarrow K_j(C^*_r(G))$ is an isomorphism;
\item[2)] $K_{j+1}(C^*_r(G))=0$
\end{enumerate}
\end{Conj}

\begin{Rem}\label{CoKadiscrete} If $G$ is semisimple with finite center, and $\pi$ is a square-integrable representation of $G$, then $\pi$ defines an isolated point of $\hat{G}_r$, so there is a splitting $C^*_r(G)=J_\pi \oplus\mathcal{K}$, where $J_\pi$ is the $C^*$-kernel of $\pi$ and $\mathcal{K}$ is the standard algebra of compact operators. Hence $K_0(C^*_r(G))=K_0(J_\pi)\oplus\Z$, i.e. $\pi$ defines a free generator $[\pi]$ of $K_0(C^*_r(G))$. In terms of the Connes-Kasparov conjecture, Theorem \ref{AtiyahSchmid} expresses the fact that the Dirac induction $\mu_G$ induces an isomorphism between an explicit free abelian subgroup of $R(K)$ and the free abelian part of $K_0(C^*_r(G))$ associated with the discrete series.
\end{Rem}

\begin{Ex} Take $G=SL_2(\R)$, so that $K=T=SO(2)$. Then the set $\Lambda$ of weights of $T$ identifies with $\Z$, the set $\Phi$ of roots is $\{-2,0,2\}$ (so that $\rho =1$ if $\Psi=\{2\}$), the set $\Phi_c$ of compact roots is $\{0\}$, and the Dirac induction consists in associating to $n>0$ the holomorphic discrete series representation $\pi_{n+1}$ (with minimal $K$-type $n+1$), and to $n<0$ the anti-holomorphic discrete series representation $\pi_{n-1}$ (with minimal $K$-type $n-1$). For the singular weight $n=0$ (i.e. the trivial character of $K$), it follows from Theorem \ref{AtiyahSchmid} that the corresponding Dirac operator $D_0$ has no kernel or cokernel. However, as prescribed by Conjecture \ref{ConKas}, its image by $\mu_G$ provides the ``missing'' generator of $K_0(C^*_r(G))$. To understand this, let us dig further into the structure of $C^*_r(G)$: apart from discrete series representations, $\hat{G}_r$ comprises two continuous series of representations. To describe those, consider the subgroup $B$ of upper triangular matrices and define two families of unitary characters (where $t\geq 0$):
$$\chi_{0,t}: B\rightarrow\mathbf{T}:\left(\begin{array}{cc}a & b \\0 & a^{-1}\end{array}\right)\mapsto |a|^{it}$$
$$\chi_{1,t}: B\rightarrow\mathbf{T}:\left(\begin{array}{cc}a & b \\0 & a^{-1}\end{array}\right)\mapsto sign(a)\cdot|a|^{it}$$
For $\epsilon=0,1$ and $t\geq 0$, denote by $\sigma_{\epsilon,t}$ the unitarily induced representation:
$$\sigma_{\epsilon,t}=Ind_B^G \chi_{\epsilon,t}.$$ 
The family $\{\sigma_{0,t}: t\geq 0\}$ (resp. $\{\sigma_{1,t}: t\geq 0\}$) is the {\it even principal series} (resp. {\it odd principal series}). For $t>0$ or for $\epsilon=0$, the representation $\sigma_{\epsilon,t}$ is irreducible. But $\sigma_{1,0}$ splits into two irreducible components $\sigma_1^+,\sigma_1^-$ (sometimes called {\it mock discrete} representations), and $\hat{G}_r$ is the union of the discrete series, the even and the odd principal series of representations. The topology on the even principal series is the topology of $[0,+\infty[$, while the topology on the odd principal series is mildly non-Hausdorff: for $t\rightarrow 0$, the representation $\sigma_{1,t}$ converges simultaneously to $\sigma_1^+$ and $\sigma_1^-$. As a consequence, the direct summand of $C^*_r(G)$ corresponding to the even principal series is Morita equivalent to $C_0([0,+\infty[)$, and hence is trivial in K-theory, while the direct summand corresponding to the odd principal series is Morita equivalent to 
$$\{f\in C_0([0,+\infty[,M_2(\C)): f(0)\;\mbox{is diagonal}\},$$
that contributes a copy of $\Z$  to $K_0(C^*_r(G))$, generated by the image of the trivial character of $K$ under Dirac induction. This description of $C^*_r(G)$ also gives $K_1(C^*_r(G))=0$ by direct computation. 
\end{Ex}

Coming back to the general framework ($G$ connected Lie group, $K$ maximal compact subgroup), let us indicate how to modify the conjecture when $G/K$ does {\it not} have a $G$-invariant spin structure. Then we may construct a double cover $\tilde{G}$ of $G$, with maximal compact subgroup $\tilde{K}$, such that $\tilde{G}/\tilde{K}=G/K$ carries a $\tilde{G}$-invariant spin structure. Let $\varepsilon\in Z(\tilde{G})$ be the non-trivial element of the covering map $\tilde{G}\rightarrow G$. Then $R(\tilde{K})$ splits into a direct sum $$R(\tilde{K})=R(\tilde{K})^0\oplus R(\tilde{K})^1,$$
where $R(\tilde{K})^0$ (resp. $R(\tilde{K})^1$) is generated by those irreducible representations $\rho\in\hat{K}$ such that $\rho(\varepsilon)=1$ (resp. $\rho(\varepsilon)=-1$). So $R(\tilde{K})^0$ identifies canonically with $R(K)$. Similarly $C^*_r(\tilde{G})$ splits into the direct sum of two ideals $C^*_r(\tilde{G})=J^0\oplus J^1$ where $J^0$ (resp. $J^1$) corresponds to those representations $\pi\in\widehat{(\tilde{G})}_r$ such that $\pi(\varepsilon)=1$ (resp. $\pi(\varepsilon)=-1$); so $J^0$ identifies canonically with $C^*_r(G)$. Now we observe that the Dirac induction for $\tilde{G}$:
$$\mu_{\tilde{G}}:R(\tilde{K})=R(\tilde{K})^0\oplus R(\tilde{K})^1\rightarrow K_j(C^*_r(\tilde{G}))=K_j(J^0)\oplus K_j(J^1)$$
interchanges the $\Z/2$-gradings: indeed the spin representations $S^\pm$ do not factor through $K$ by assumption, but if $\rho$ is in $R(\tilde{K})^1$, then $S^\pm\otimes\rho$ factors through $K$ (as $\varepsilon$ acts by the identity). Hence the second case of the Connes-Kasparov conjecture:

\begin{Conj}\label{ConKas2} (2nd version) Let $G$ be a connected Lie group, $K$ a maximal compact subgroup, $j=\dim(G/K)$. Assume that $G/K$ does not carry a $G$-invariant spin structure.
\begin{enumerate}
\item[1)] The Dirac induction $\mu_{\tilde{G}}:R(\tilde{K})^1\rightarrow K_j(C^*_r(G))$ is an isomorphism;
\item[2)] $K_{j+1}(C^*_r(G))=0$
\end{enumerate}
\end{Conj}

As we said before, the Connes-Kasparov conjecture was eventually proved for arbitrary connected Lie groups by J. Chabert, S. Echterhoff and R. Nest \cite{ChEcNe}, whose result is even more general as it encompasses almost connected groups, i.e. locally compact groups whose group of connected components is compact.

\begin{Thm}\label{ConKasCEN} The Connes-Kasparov conjecture holds for almost connected groups.
\end{Thm}

In the same paper \cite{ChEcNe}, Chabert-Echterhoff-Nest obtain a purely representation-theoretic consequence of Theorem \ref{ConKasCEN}:

\begin{Cor} Let $G$ be a connected unimodular Lie group. Then all square-integrable factor representations of $G$ are type I. Moreover, $G$ has no square-integrable factor representations if $\dim(G/K)$ is odd.
\end{Cor}

\subsection{The Novikov conjecture}\label{Novikov}

For discrete groups, an important motivation for the Baum-Connes conjecture was provided by the work of A.S. Mishchenko (see e.g. \cite{Mish}) and G.G. Kasparov (see e.g. \cite{KaspConsp}) on the {\it Novikov conjecture}, whose statement we now recall. 

For a discrete group $\Gamma$, denote by $B\Gamma$ ``the'' {\it classifying space} of $\Gamma$, a $CW$-complex characterized, up to homotopy, by the properties that its fundamental group is $\Gamma$ and its universal cover $E\Gamma$ is contractible\footnote{As A. Connes once pointed out: ``{\it $E\Gamma$ is a point on which $\Gamma$ acts freely!}''}. Alternatively, $B\Gamma$ is a $K(\Gamma,1)$-space. As a consequence, group cohomology of $\Gamma$, defined algebraically, is canonically isomorphic to cellular cohomology of $B\Gamma$. 

Let $M$ be a smooth, closed, oriented manifold of dimension $n$, equipped with a map $f:M\rightarrow B\Gamma$. For $x\in H^*(B\Gamma,\Q)$ (cohomology with rational coefficients), consider the {\it higher signature}
$$\sigma_x(M,f)=\langle f^*(x)\cup L(M),[M]\rangle \in\Q,$$
where $L(M)$ is the $L$-class (a polynomial in the Pontryagin classes, depending on the smooth structure of $M$), and $[M]$ is the fundamental class of $M$. The Novikov conjecture states that these numbers are homotopy invariant (and so do not depend on the smooth structure of $M$):

\begin{Conj}\label{Nov} (The Novikov conjecture on homotopy invariance of higher signatures). Let $h:N\rightarrow M$ be a homotopy equivalence; then for any $x\in H^*(B\Gamma,\Q)$:
$$\sigma_x(M,f)=\sigma_x(N,f\circ h).$$
\end{Conj}

We say that {\it the Novikov conjecture holds for $\Gamma$} if Conjecture \ref{Nov} holds for every $x\in H^*(B\Gamma,\Q)$. We refer to the detailed survey paper \cite{FRR} for the history of this conjecture, and an explanation why it is important.

We summarize now Kasparov's approach from section 9 in \cite{KaspConsp}\footnote{Although published only in 1995, the celebrated ``{\it Conspectus}'' was first circulated in 1981.}. Keeping notations as in Conjecture \ref{Nov}, Kasparov considers the homology class $\mathcal{D}(M)=L(M)\cap [M]\in H_*(M,\Q)$ which is Poincar\'e-dual to $L(M)$, and Conjecture \ref{Nov} is equivalent to the homotopy invariance of the class $f_*(\mathcal{D}(M))\in H_*(B\Gamma,\Q)$.

Let $d:\Omega^p(M)\rightarrow \Omega^{p+1}(M)$ be the exterior derivative on differential forms. Up to crossing $M$ with the circle $S^1$, we may assume that $n=\dim M$ is even. Fix an auxiliary Riemannian metric on $M$. This allows to define the adjoint $d^*:\Omega^p(M)\rightarrow \Omega^{p-1}(M)$: it satisfies $d^*=-\star d\star$, where $\star$ is the Hodge operator associated with the Riemannian structure. 

Now consider $d+d^*$ acting on the space of all forms %
$\Omega (M)=\bigoplus _{p=0}^{n}\Omega ^{p}(M)$. One way to consider this as a graded operator is the following: let $\tau$  be an involution on the space of all forms defined by:
$$\tau (\omega )=i^{p(p-1)+\frac{n}{2}}\star \omega \quad ,\quad \omega \in \Omega ^{p}(M).$$
It is verified that $d+d^{*}$ anti-commutes with $\tau$: with this grading on forms, $d+d^*$ is the {\it signature operator} on $M$. As it is an elliptic operator, it defines an element $[d+d^*]$ in the group $K_0(M)$ of K-homology\footnote{K-homology is the homology theory dual to topological K-theory. It was shown by Atiyah \cite{Atiyah70} that an elliptic (pseudo-)differential operator on $M$ defines an element in $K_0(M)$.} of $M$. Note that, by connectedness of the space of Riemannian metrics on $M$, the element $[d+d^*]\in K_0(M)$ does not depend on the choice of a Riemannian metric. Using Hodge theory, it is classical to check that the index of $d+d^*$ is exactly the {\it topological signature} of $M$, i.e. the signature of the quadratic form given by cup-product on the middle-dimensional cohomology $H^{\frac{n}{2}}(M,\C)$. Now consider the {\it index pairing} between K-theory and K-homology of $M$:
$$K^0(M)\times K_0(M)\rightarrow \Z: (\xi,D)\mapsto Ind(D_\xi),$$
the index of the differential operator $D_\xi$, which is $D$ with coefficients in the vector bundle $\xi$ on $M$. In particular $Ind((d+d^*)_\xi)$ is the index of the signature operator with coefficients in $\xi$, i.e. acting on sections of $\Lambda^*(M)\otimes\xi$. It is given by the cohomological version of the Atiyah-Singer index theorem:
\begin{equation}\label{AS}
Ind((d+d^*)_\xi)= \langle Ch^*(\xi)\cup L(M),[M]\rangle,
\end{equation}
where $Ch^*$ denotes the Chern character in cohomology. Recall that, for every finite $CW$-complex $X$, we have Chern characters in cohomology and homology:
$$Ch^*:K^0(X)\rightarrow \bigoplus_{k=0}^\infty H^{2k}(X,\Q);$$
$$Ch_*:K_0(X)\rightarrow \bigoplus_{k=0}^\infty H_{2k}(X,\Q),$$
which are rational isomorphisms, compatible with the index pairing and with the pairing between cohomology and homology. Equation \ref{AS} then implies that 
\begin{equation}\label{Chern}
Ch_*([d+d^*])=L(M)\cap[M]=\mathcal{D}(M).
\end{equation}
Assume for simplicity that $B\Gamma$ is a closed manifold \footnote{When $B\Gamma$ is a general $CW$-complex we must replace $K_0(B\Gamma)$ by 
$RK_0(B\Gamma)=\varinjlim_X K_0(X)$, where $X$ runs along compact subsets of $B\Gamma$.}, which implies that $\Gamma$ is torsion-free. Recall that Conjecture \ref{Nov} is equivalent to homotopy invariance of $f_*(\mathcal{D}(M))$. By equation \ref{Chern} and functoriality of $Ch_*$, we have:
$$f_*(\mathcal{D}(M))=f_*(Ch_*([d+d^*]))=Ch_*(f_*[d+d^*]).$$
By rational injectivity of $Ch_*$, we see that Conjecture \ref{Nov} is equivalent to the homotopy invariance of $f_*[d+d^*]$ in $K_0(B\Gamma)\otimes_\Z \Q$.

 In the final section of \cite{KaspConsp}, Kasparov defines a homomorphim $\beta:K_i(B\Gamma)\rightarrow K_i(C^*_r(\Gamma))$ that later was identified with the assembly map $\mu_r: K_i^\Gamma(\underline{E\Gamma})\rightarrow K_i(C^*_r(\Gamma))$. Kasparov's $\beta$ is defined as follows. Keep the assumption that $B\Gamma$ is a finite complex. Form the induced vector bundle $\mathcal{L}_\Gamma=E\Gamma\times_\Gamma C^*_r(\Gamma)$ (where $\Gamma$ acts on $C^*_r(\Gamma)$ by left translations). This is a vector bundle with fiber $C^*_r(\Gamma)$ over $B\Gamma$, sometimes called the {\it Mishchenko line bundle}. Its space $C(E\Gamma,C^*_r(\Gamma))^\Gamma$ of continuous sections, is a projective finite type module over $C(B\Gamma)\otimes C^*_r(\Gamma)$ (and as such it defines a K-theory element $[\mathcal{L}_\Gamma]\in K_0(C(B\Gamma)\otimes C^*_r(\Gamma))$). For a K-homology element $[D]\in K_0(B\Gamma)$ given by an elliptic (pseudo-)differential operator $D$ over $B\Gamma$ we may form the operator $D_{\mathcal{L}_\Gamma}$ with coefficients in $\mathcal{L}_\Gamma$: its kernel and co-kernel are projective finite type modules over $C^*_r(\Gamma)$, so their formal difference defines an element $\beta[D]\in K_0(C^*_r(\Gamma))$: this defines the desired homomorphism\footnote{In terms of Kasparov theory, to be defined in Chapter \ref{indexmaps} below, this can be expressed using Kasparov product: $\beta[D]=[\mathcal{L}_\Gamma]\otimes_{C(B\Gamma)}[D]$.}
$$\beta:K_0(B\Gamma)\rightarrow K_0(C^*_r(\Gamma)).$$

Coming back to the Novikov conjecture, recall that it is equivalent to the homotopy invariance of $f_*[d+d^*]$ in $K_0(B\Gamma)\otimes_\Z \Q$. Now one of Kasparov's result in \cite{KaspConsp} (Theorem 2 in the final section) is:

\begin{Thm}\label{KasNovikov} If $M$ is an even-dimensional smooth, closed, oriented manifold and $f:M\rightarrow B\Gamma$ is a continuous map, then $\beta(f_*[d+d^*])\in K_0(C^*_r(\Gamma))$ is a homotopy invariant of $M$.
\hfill$\square$
\end{Thm}

As an immediate consequence of Theorem \ref{KasNovikov}, we get the following result:

\begin{Cor}\label{BCimpliesNov} If the map $\beta$ is rationally injective, then the Novikov conjecture (conjecture \ref{Nov}) holds for $\Gamma$.
\end{Cor}

The main result of Kasparov's Conspectus \cite{KaspConsp}, is the following: 
\begin{Thm} If $\Gamma$ is a discrete subgroup of a connected Lie group, then the map $\beta$ is injective.
\end{Thm}
\begin{Cor} The Novikov conjecture holds for any discrete subgroup of a connected Lie group.
\end{Cor}

\section { Index maps in K-theory: the contribution of Kasparov}\label{indexmaps}

\subsection{Kasparov bifunctor}\label{bifunctor}

The powerful tool developed by G. Kasparov in his proof of the Novikov conjecture is the equivariant KK-theory. We refer to \cite{KaspConsp} and \cite{Kasparov88}.

For any locally compact group $G$ and $A$, $B$ two $G-C^*$-algebras  (i.e. $C^*$-algebras equipped with a strongly continuous action by automorphisms of the group $G$), Kasparov defines an abelian group $KK_G(A, B)$. The main tool  in the theory is the cup product $$KK_G(A,B)\times KK_G(B,C)\rightarrow KK_G(A,C):(x,y)\mapsto x\otimes_B y.$$ 
In particular, if ${\bf C}$ is the field of complex numbers equipped with the trivial $G$-action,  $KK_G({\bf C}, {\bf C})$ is a  ring, which turns out to be commutative. Moreover the homomorphisms 
$$\tau_D: KK_G(A,B)\rightarrow KK_G(A\otimes D,B\otimes D)$$
defined by tensoring by a $C^*$-algebra $D$, equip all the $KK_G(A, B)$'s with a structure of $KK_G({\bf C}, {\bf C})$-modules. 

One of the most important ingredients in $G$-equivariant KK-theory is the existence of descent maps : for all $G$-$C^*$algebras $A$ and $B$ there are group homomorphisms

 $$j_{G,r}: KK_G(A,B)\rightarrow KK(C_r^*(G,A),C_r^*(G,B))$$
 $$j_{G,{\rm max}}: KK_G(A,B)\rightarrow KK(C_{\rm max}^*(G,A),C_{\rm max}^*(G,A))$$
where $C_r^*(G,A)$ and $C_{\rm max}^*(G,A)$ denote respectively the reduced and the full crossed product.

The abelian group $KK_G(A, B)$ is defined as follows:

\begin{Def}\label{Fredmodule} An $(A,B)$-Fredholm bimodule is given by:
\begin{enumerate}
\item[(i)]\ a $B$-Hilbert module $E$;
\item[(ii)] a covariant representation ($\pi,\rho (g))$  of $(G, A)$ on the Hilbert module $E$;
\item[(iii)] an operator $T$ on $E$ , $B$-bounded and self-adjoint (i.e. $T=T^*$) and such that: for any $a$ in $A$ and $g$ in $G$, the operators $(1-T^2) \rho (a)$, $T\rho(a)-\rho(a)T$ and $T\pi (g)-\pi(g)T$ are $B$-compact operators; moreover the map $g\mapsto T\pi (g)-\pi(g)T$ is norm continuous.
\end{enumerate}
\end{Def}

Such a $(A,B)$-Fredholm module is also called odd $(A,B)$-Fredholm module. An even $(A,B)$-Fredholm module is given by a $(A,B)$-Fredholm module together wih a ${\bf Z}/2$-grading on the module $E$, such that the covariant representation preserves the grading, and the operator $T$ is odd with respect to the grading.

\bigskip
One defines a homotopy of $(A,B)$-Fredholm modules to be a $(A,B\otimes C[(0,1])$-Fredholm module. 
An element of $KK_G(A, B)$ is defined as a homotopy class of even $(A,B)$-Fredholm modules. Addition is given by direct sum. The zero element is given by the class of degenerate modules, i.e. those where ``compact'' is replaced by ``zero'' in Definition \ref{Fredmodule}. When necessary we use the notation $KK_G^j(A, B)$ with $j=0$ (resp. $1$) for the even (resp. odd case). 

When there is no group acting, we simply write $KK(A,B)$. Ordinary K-theory for $C^*$-algebras is recovered by $K_*(B)=KK^*(\C,B)$, while K-homology corresponds to $K^*(A)=KK^*(A,\C)$.

\subsection{Dirac induction in KK-theory}

In \cite{KaspConsp}, G. Kasparov gives an interpretation of the Dirac induction map   from $K_*(C^*(K,A))$ to $K_*(C^*_r(G,A))$ in the framework of KK-theory,. Here $G$ is a semisimple Lie group with finite center and $K$ a maximal compact subgroup. We assume that the adjoint representation of $K$ on $V=\mathfrak{g}/\mathfrak{k}$ lifts to $Spin(V)$. The symmetric space $X=G/K$ then carries a $G$-invariant spin structure. Let $D$ be the corresponding Dirac operator, a $G$-invariant elliptic operator defined on the sections of the spinor bundle $S$ of $X$.

We define an element $\alpha$ of the group $KK_G^j(C_0(X),{\bf C})$ as the homotopy class of the $(C_0(X),{\bf C})$-Fredholm bimodule defined by:

\begin{enumerate}
\item[(i)] The Hilbert space $L^2(X , S)$ of $L^2$-sections of the spinor bundle $S$. 

\item[(ii)] The covariant action on $L^2(X , S)$ of the $G-C^*$-algebra $C_0(X)$ of continuous functions on $X$ vanishing at infinity.

\item[(iii)] The operator $F=D(1+D^2)^{-1/2}$ obtained by functional calculus from the Dirac operator $D$.
\end{enumerate}

Note that the bundle $S$ is graded for $j$ even, and trivially graded if $j$ is odd. The above Fredholm module therefore defines an element  $\alpha \in KK_G^j(C_0(G/K),{\bf C})$ where $j=\dim G/K \;(\mod 2)$.

Now consider the following composition
$$KK_G(C_0(G/K),{\bf C})\rightarrow KK_G(C_0(G/K)\otimes A, A)\rightarrow KK(C^*(K, A), C^*_r(G,A))$$
where the first map is $\tau_A$  and the second is $j_{G,r}$, taking into account the Morita equivalence of $C^*_r(G,C_0(G/K)\otimes A)$ with $C^*(K, A)$.
The image of $\alpha$ by the above composed map is an element of $KK^j(C^*(K,A), C^*_r(G,A))$ which defines a map $$\tilde{\alpha}_A: K_{*+j}(C^*(K,A))\rightarrow K_*(C^*_r(G,A)).$$

Note the two special cases:
\begin{enumerate}
\item[(i)] When $A={\bf C}$, this is nothing but the Connes-Kasparov map $K_{*+j}(C^*(K))\rightarrow K_*(C^*_r(G))$, see Conjecture \ref{ConKas}.
\item[(ii)] When $\Gamma$ is a torsion-free discrete cocompact subgroup of $G$, and $M=\Gamma\backslash G/K$, this gives the map $\beta:K_{*+j}(C(M))\rightarrow K_*(C^*_r({\Gamma}))$, see section 2.5 where $M=B\Gamma$, the classifying space of $\Gamma$.
\end{enumerate}

\subsection {The dual Dirac method and the $\gamma$-element}\label{gamma}

In order to construct the inverse map, Kasparov defines in  \cite{KaspConsp} the element $$\beta\in KK_G^j({\bf C}, C_0(X))$$ as the homotopy class of the following $({\bf C}, C_0(X))$-Fredholm bimodule:
\begin{enumerate}
\item[(i)] The $C_0(X)$-Hilbert module $C_0(X, S)$ of sections of the spinor bundle $S$;
\item[(ii)] the natural action of $G$ on $C_0(X)$;
\item[(iii)] the operator on $C_0(X)$ which is the  Clifford multiplication by the 
vector field $b$ on $X$ defined as follows: let $x_0$ be the origin in 
$X$ (i.e. the class of the identity in $G/K$), then the value of  $b$ at a point $x\in X$ is the vector tangent to the geodesic from $x$ to $x_0$, and of length $\rho(1+\rho^2)^{-1/2}$ if $\rho$ is the distance between $x$ and $x_0$.
\end{enumerate}

Similarly to what was done for $\alpha$, the element $\beta\in KK_G^j({\bf C}, C_0(G/K))$ gives rise to an element $\tilde\beta$ of $KK^j(C^*_{\rm red}(G,A),C^*(K,A))$ by applying to $\beta$ the following maps:
$$KK_G(({\bf C}, C_0(G/K))\rightarrow KK_G(A, C_0(G/K)\otimes A)\rightarrow KK(C^*_r(G,A), C^*(K, A)),$$
hence a map $K_*(C^*_r(G,A)) \rightarrow K_{*+j}(C^*(K,A))$ which is a candidate to be the inverse of the index map.

In other words, one would hope that the following equalities hold in KK-theory: 
$\alpha\otimes_{\bf C}\beta=1$ in $KK_G(C_0(X),C_0(X))$ and $\beta\otimes_{C_0(X)}\alpha=1$ in $KK_G({\bf C}, {\bf C}) $. However such a dream is not fullfilled. Only the first statement is true in general.

 \begin {Thm}\label{dualdirac} One has $\alpha\otimes_{\bf C}\beta=1$ in $KK_G(C_0(X),C_0(X))$. As a consequence,  $\gamma:=\beta\otimes_{C_0(X)}\alpha$   
 is an idempotent of the ring $KK_G(\C,\C)$ , i.e. $\gamma\otimes_{\bf C}\gamma=\gamma$. 
\end {Thm}

This element $\gamma$ plays a key role in the Baum-Connes conjecture. The main step in the proof of Theorem \ref{dualdirac} is the following {\it rotation lemma}:

\begin {Lem}\label{rotation}  $\alpha\otimes_{\bf C}\beta = \tau_{C_0(X)}(\beta\otimes_{C_0(X)}\alpha)$.
\end{Lem}

On the other hand, Kasparov shows  

\begin {Lem}\label{restriction} ${\rm Rest}^G_K (\gamma)=1$ in $R(K)$.
\end{Lem}

where $${\rm Rest}^G_K: KK_G({\bf C}, {\bf C})\rightarrow R(K)$$ is the natural restriction map. This is a $K$-equivariant version of the Bott periodicity. Namely, from the $K$-equivariant point of view, the space $G/K$ can be replaced by its tangent space $V$ at $x_0$. Then the Euclidean space $V$ is equipped with a representation of $K$  which factors though $Spin (V)$ and the Bott periodicity has an equivariant version, an isomorphism between $K_*(C^*(K,C_0(V)))$ and $R(K)$.

\begin {Cor}\label{induction} ${\tau_{C_0(G/K)}(\gamma})=1$ in $KK_G(C_0(G/K),C_0(G/K))$.
\end{Cor}

This follows from the fact that $\tau_{C_0(G/K)}={\rm Ind}^G_K \circ {\rm Rest}^G_K$,
where the induction ${\rm Ind}^G_K: R(K)\rightarrow KK_G(C_0(G/K),C_0(G/K))$ is defined in \cite {Kasparov88}. Theorem \ref{dualdirac} follows by combining lemma \ref{rotation} with corollary \ref{induction}.

Since $\gamma$ is an idempotent, the ring $KK_G({\bf C}, {\bf C})$ is a direct sum of two subrings
$$KK_G({\bf C}, {\bf C})=\gamma KK_G({\bf C},{\bf C})\oplus (1-\gamma ) KK_G({\bf C},{\bf C})$$
Moreover, by lemma \ref{restriction} the restriction map
 $KK_G({\bf C},{\bf C})\rightarrow KK_K({\bf C},{\bf C})=R(K)$ is an isomorphism from $\gamma KK_G({\bf C},{\bf C})$ to $R(K)$, and vanishes on the complement $(1-\gamma ) KK_G({\bf C},{\bf C})$. More generally for any $A$,$B$ as above, 
 $$KK_G(A,B)=\gamma KK_G(A,B)\oplus (1-\gamma ) KK_G(A,B),$$
 the restriction map is an isomorphism from $\gamma KK_G(A,B)$ to $KK_K(A,B)$ and vanishes on $(1-\gamma ) KK_G(A,B)$.
 
 \bigskip

The element $\gamma$ acts on the K-theory of $C_r^*(G,A)$  by an idempotent map which can be described as follows. Consider the composition of ring homomorphisms
$$KK_G({\bf C}, {\bf C})\rightarrow KK_G(A, A)\rightarrow KK(C_r^*(G, A), C_r^*(G,A))\rightarrow {\rm End}(K_*(C_r^*(G,A)))$$
and take the image of the idempotent $\gamma$ by the above map: 
$$\tilde\gamma_{A}\in{\rm End}(K_*(C_r^*(G,A))).$$
The results of Kasparov  \cite{KaspConsp} \cite{Kasparov88} can then be summarized as follows:

\begin{Thm}\label{StrongNovikov} The map $\tilde\alpha_A$  is injective\footnote{The injectivity of $\tilde\alpha_A$ is responsible for the Novikov conjecture, conjecture \ref{Nov}: see subsection \ref{InjectivityNovikov}.}. Its image in $K_*(C_r^*(G,A))$ is equal to the image of the idempotent map
$\tilde\gamma_A$.
\end {Thm}

\begin{Cor} The Connes-Kasparov conjecture with coefficients in $A$ (i.e. the statement that $\tilde\alpha_A$ is an isomorphism) is equivalent to the equality $\tilde\gamma_A={\rm Id}$.
\end{Cor}

\begin{Cor} If $\gamma=1$ in $KK_G({\bf C},{\bf C})$, then the Connes-Kasparov conjecture with coefficients is true.
\end{Cor}

\subsection {From K-theory to K-homology}\label{Poincare}

All the constructions above rest upon the assumption that the space $X=G/K$ carries a $G$-equivariant structure of a {\it spin manifold}, or equivalently that the representation of $K$ on $V^*=T_{x_0}^*X$ is spinorial.

In the case of a general connected Lie group, this is not necessarily the case, and Kasparov's constructions have to be modified as follows.
Consider the cotangent bundle $T^*X$ which has an almost-complex structure. There is therefore a Dirac operator on $T^*X$, which defines an element $\alpha\in KK_G(C_0(T^*X),{\bf C})$. 
Applying the same procedure as above yields an element of $KK(C^*(K,A\otimes C_0(V^*)), C^*_r(G,A))$ since $C^*(G,A\otimes C_0(T^*X))$ is Morita equivalent to $C^*(K,A\otimes C_0(V^*))$.

Therefore the element $\alpha$ defines a map $$K_*(C^*(K,A\otimes C_0(V^*)))\rightarrow K_*(C^*_r(G,A)).$$
Note that there is no dimension shift but that $A$ is replaced by $A\otimes C_0(V^*)$.
As usual, note the special cases $A={\bf C}$ and $A=C(G/\Gamma )$
\begin{enumerate}
\item[1)] $K_*(C^*(K, C_0(V^*)))\rightarrow K_*(C^*_r(G));$
\item[2)] $K^*(T^*M)\rightarrow K_*(C^*_r(\Gamma))$ where $M=\Gamma\backslash G/K$.
\end{enumerate}
\bigskip
In the same way one can define a dual-Dirac element $\beta\in KK_G({\bf C}, C_0(T^*X))$ and an element $\gamma\in KK_G({\bf C},{\bf C})$. The same results as above do hold.

The role of the cotangent bundle $T^*X$ or equivalently the representation of $K$ on $V^*=T_{x_0}^*X$ is closely related to Poincar\'e duality in K-theory. The latter is conveniently formulated in Kasparov theory as follows. As we shall see, 
the left-hand side of the conjecture should in fact be interpreted, rather than a K-theory group, as a K-homology group. The Dirac induction map appears rather as the composition of the assembly map with the Poincar\'e duality map.

Let us explain that point. In Kasparov theory, the K-homology $K^*(A)$ of a $C^*$-algebra is defined as the group $KK(A, {\bf C})$. There is a duality pairing 
$$K_*(A)\otimes K^*(A)\rightarrow {\bf Z}$$
with the K-theory $K_*(A)=KK({\bf C},A)$, defined by  the cup product $$KK({\bf C},A)\otimes KK(A,{\bf C})\rightarrow KK({\bf C},{\bf C})={\bf Z}$$

For example if $M$ is a compact manifold, the K-homology group $K_*(M)=K^*(C(M))$ can be described, according to Atiyah \cite{Atiyah70} , as the group ${\rm Ell}(M)$ of classes of elliptic operators on the manifold $M$. The pairing $K^*(M)\otimes K_*(M)\rightarrow {\bf Z}$ associates to a vector bundle $E$ and an elliptic operator $D$ the index of the operator $D_E$ with coefficients in $E$.
Poincar\'e duality in K-theory is a canonical  isomorphism
 $$K^*(T^*M)\rightarrow K_*(M)$$ between the K-homology of $M$  and the K-theory of the total space $T^*M$ of its cotangent bundle.
Such a map can be interpreted as follows: an element of $K^*(T^*M)$ is the homotopy class of an elliptic symbol on $M$. Its image in $K_*(M)$ is the class of an elliptic pseudodifferential operator associated to that symbol. 
In Kasparov theory, one can interpret Poincar\'e duality as the existence of two elements, respectively of $KK(C(M)\otimes C_0(T^*M), {\bf C})$ and of $KK({\bf C},C(M)\otimes C_0(T^*M))$, inverse to each other for the cup product. See the details in \cite{Kasparov88}.

This allows to reformulate the conjecture as follows. For the case of a torsion-free discrete cocompact subgroup $\Gamma$ as above, the map $K^*(T^*M)\rightarrow K_*(C^*_r(\Gamma))$ becomes\footnote{This is actually the same map as the map $\beta$ from section \ref{Novikov}.}
$$K_*(M)\rightarrow K_*(C^*_r(\Gamma)).$$
In general, one needs the $G$-equivariant version of Poincar\'e duality for the space $X=G/K$. There are two elements one of 
$KK_G(C_0(X)\otimes C_0(T^*X), {\bf C})$ and the other of $KK_G({\bf C},C_0(X)\otimes C_0(T^*X))$ that are inverse to each other.

Then for any $G-C^*$-algebra $A$, one has an isomorphism
$$KK^G({\bf C}, C_0(T^*X)\otimes A)\rightarrow KK^G(C_0(X), A).$$
One can show that the first group is isomorphic to $$KK({\bf C}, C^*(G,C_0(T^*X)\otimes A)=KK({\bf C}, C^*(K,C_0(V^*)\otimes A)).$$

The Dirac induction with coefficients in $A$ can therefore be defined as a map $$KK^G(C_0(X), A))\rightarrow K_*(C^*_{\rm red}(G,A))$$
which in the case without coefficients can be written as
$K_*^G(X))\rightarrow K_*(C^*_{\rm red}(G))$.

\subsection {Generalisation to the $p$-adic case}\label{p-adic}

Shortly after the work of Kasparov, it became natural to investigate the analogue of the Kasparov Dirac-dual Dirac method when real Lie groups are replaced by $p$-adic groups. According to the philosophy of F. Bruhat and J. Tits the $p$-adic analogue of the symmetric space is a building of affine type (see \cite{BT72}, \cite{T75}). It shares with symmetric spaces the property of unique geodesics between two points, and the fact that the stabilisers of vertices are maximal compact subgroups (note that there may be several conjugacy classes of such subgroups). In the rank one case, e.g. $SL(2,{\bf Q}_p$), the Bruhat-Tits building is the Bass-Serre tree. 
P. Julg and A. Valette \cite{Julg-Valette88} have constructed an element $\gamma$ for buildings using an operator on the Hilbert space $\ell^2 (X)$ (the set $X$ is seen as the set of objects of all dimensions) which may be seen as the "vector pointing to the origin", generalizing the Julg-Valette element for trees \cite{JV84}.

\medskip
The question of an analogue of the Connes-Kasparov conjecture for $p$-adic groups has been considered by Kasparov and Skandalis in \cite{Kasparov-Skandalis91}. They met the following
difficulty: the building is not a manifold, and it does not satisfy the Poincar\'e duality in the usual sense. However, if $X$ is a simplicial complex, there is an algebra ${\cal A}_X$ which plays the role played by the algebra $C^*(TM)=C_0(T^*M)$ 
in the case of a manifold $M$. The algebra ${\cal A}_X$ is not commutative, it is in fact the algebra of a groupoid associated to the simplicial complex $X$. Moreover, ${\cal A}_X$ is Poincar\'e dual in K-theory to the commutative algebra $C_0(\vert X\vert )$ of continuous functions on the geometric realisation of $X$:  there is a canonical isomorphism $$K_*({\cal A}_X)\rightarrow K^*(C_0(\vert X\vert ))$$
from the K-theory of the algebra $A_X$ to the K-homology of the space $\vert X\vert$. 

Let us now assume that
 $X$ is the Bruhat-Tits building of a reductive linear group over a non-Archimedean local field (e.g. ${\bf Q}_p$). Then the above form of the Poincar\'e duality, in a $G$-equivariant way, shows the isomorphism
 $$KK_G(C_0(\vert X\vert ),A)=K_*(C^*(G, {\cal A}_X\otimes A))$$
 for any $G-C^*$-algebra $A$. 
 
 By analogy with the Lie group case, it was natural to construct a map from the group above to the K-theory group $K_*(C^*_r(G,A))$. G. Kasparov and G. Skandalis  \cite{Kasparov-Skandalis91} construct a Dirac element $\alpha\in KK_G({\cal A}_X,{\bf C})$ which defines as above maps in K-theory:
$$K_*(C^*(G, {\cal A}_X\otimes A))\rightarrow K_*(C^*_r(G,A)).$$
The left-hand side can be computed by Morita equivalence from the K-theory of crossed products of $A$ by the compact subgroups of $G$ stabilizing the vertices of a simplex viewed as a fundamental domain. A special case is the Pimsner exact sequence for trees \cite{Pimsner86}.

G. Kasparov and G. Skandalis have shown the injectivity of the above map (which implies the Novikov conjecture for discrete subgroups of $p$-adic groups) by constructing a dual-Dirac element $\beta\in KK_G({\bf C},{\cal A}_X)$. They show that 
$$\beta\otimes_{A_X}\alpha=\gamma\in KK_G({\bf C},{\bf C}),$$
the Julg-Valette element of  \cite{Julg-Valette88}. A rotation trick shows that $\alpha\otimes_{\bf C}\beta=1$.

At this point we note that the Lie group case and the $p$-adic group case can be unified by the K-homology formulation of the conjecture. If $Z$ denotes the locally compact $G$-space which is the symmetric space $G/K$ in the Lie case, the geometric realization $\vert X\vert$
of the Bruhat-Tits building in the $p$-adic case, the conjecture is that a certain map

$$KK_G(C_0(Z),A)\rightarrow K_*(C^*(G, A))$$
is an isomorphism. This will become more precise with the Baum-Connes-Higson formulation of the conjecture for general locally compact groups: the role of the symmetric spaces or Bruhat-Tits buildings will be clarified as classifying spaces for proper actions, see sections \ref{classifying} and \ref{official}.
In both cases injectivity can be proved by a Dirac-dual-Dirac method, which hints to a general notion of $\gamma$-element, as explained in section \ref{viveTu}.

\section{Towards the official version of the conjecture}

\subsection{Time-dependent left-hand side}\label{timedep}
There is a certain time-dependency in the left-hand side of the Baum-Connes conjecture, hence also in the assembly map. Let us first recall the fundamental concept of proper actions.

\begin{Def} \begin{enumerate}
\item Let $G$ be a locally compact group. A $G$-action on a locally compact space $X$ is said to be {\it proper} if the action map $$G\times X\rightarrow X:(g,x)\mapsto gx$$ is proper, i.e. the inverse image of a compact subset of $X$, is compact.  
\item If $X$ is a locally compact, proper $G$-space, then the quotient space $G\backslash X$ is locally compact, and $X$ is said to be {\it $G$-compact} if $G\backslash X$ is compact.
\end{enumerate}
\end{Def}

In the original paper of Baum-Connes \cite{BC}, the conjecture is formulated only for Lie groups - possibly with infinitely many connected components, so as to include discrete groups. However, the authors take great care in allowing coefficients, in the form of group actions on smooth manifolds. So if $G$ is a Lie group (not necessarily connected) and $M$ is a manifold, the goal is to identify the analytical object $K_*(C^*_r(G,C_0(M)))$ (the K-theory of the reduced crossed product $C^*$-algebra), with something of geometrical nature. 

This is done in two steps. First, let $Z$ be a proper $G$-manifold. Denote by $V^0_G(Z)$ the collection of all $G$-elliptic complexes of vector bundles $(E_+,E_-,\sigma)$ where $E_+,E_-$ are $G$-vector bundles over $Z$, and $\sigma:E_+\rightarrow E_-$ is a $G$-equivariant vector bundle map, which is invertible outside of a $G$-compact set. One also defines $V^1_G(Z)=V^0_G(Z\times \R)$, where $G$ acts trivially on $\R$.

The second - and main - step is to consider an {\it arbitrary} $G$-manifold $M$ and to ``approximate'' it by proper $G$-manifolds; here one can identify, in germ, the presence of the classifying space for $G$-proper actions that will come to the forefront in the ``official'' version of the conjecture in \cite{BCH}; see section \ref{official} below. In \cite{BC}, a {\it K-cocycle} for $M$ will be a triple $(Z,f,\xi)$ where:
\begin{itemize}
\item $Z$ is a proper, $G$-compact, $G$-manifold;
\item $f:Z\rightarrow M$ is a $G$-map;
\item $\xi\in V^*_G(T^*Z\oplus f^*T^*M)$.
\end{itemize}
We denote by $\Gamma(G,M)$ the set of K-cocycles for $M$. If $(Z,f,\xi)$ and $(Z',f',\xi')$ are two equivariant K-cycles for $X$, then their disjoint union is the equivariant K-cycle $(Z\coprod Z',f\coprod f',\xi\coprod\xi')$. It is assumed that manifolds are not necessarily connected, and their connected components do not always have the same dimension. The operation of disjoint union will give addition.

Suppose that the manifolds $Z_1,Z_2,M$ and the $G$-maps $f_1,f_2,g$ fit into a commutative diagram 

$$\xymatrix{Z_1\ar[rr]^{h}\ar[dr]_{f_1}&&Z_2\ar[dl]^{f_2}\\
&M&
}
$$

Then, using the Thom isomorphism, it is possible to construct a ``{\it wrong way functoriality''} Gysin map
$$h_!: K^*_G(T^*Z_1\oplus f_1^*T^*M)\rightarrow K^*_G(T^*Z_2\oplus f_2^*T^*M).$$
Two K-cocycles $(Z_1,f_1,\xi_1),(Z_2,f_2,\xi_2)$ are said to be equivalent\footnote{The fact that it is indeed an equivalence relation does not appear in \cite{BC}.} if there exists a K-cocycle $(\tilde{Z},\tilde{f},\tilde{\xi})$ and $G$-maps $h_1:Z_1\rightarrow\tilde{Z}, h_2:Z_2\rightarrow\tilde{Z}$ making the following diagram commutative:

$$\xymatrix{Z_1\ar[r]^{h_1}\ar[dr]_{f_1}&\tilde{Z}\ar[d]_{\tilde{f}}&Z_2\ar[l]_{h_2}\ar[dl]^{f_2}\\
&M&
},$$

and such that $h_{1,!}(\xi_1)=\tilde{\xi}=h_{2,!}(\xi_2)$. Then we define $K^{top}(G,M)$ as the quotient of $\Gamma(G,M)$ by this equivalence relation. 

To construct the assembly map $\mu_{r,M}: K^{top}(G,M)\rightarrow K_*(C^*_r(G,C_0(M)))$, the construction is roughly as follows. Start from a K-cocycle $(Z,f,\xi)\in\Gamma(G,M)$. Observe that $f=p\circ i$, where $i:Z\rightarrow Z\times M:z\mapsto(z,f(z))$ and $p:Z\times M\rightarrow M$ is the projection onto the second factor. Replacing $Z$ by $Z\times M$ and $f$ by $p$, we may assume that $f$ is a submersion. Let then $\tau$ be the cotangent bundle along the fibers of $f$. By the Thom isomorphism, the class $\xi\in V^*_G(T^*Z\oplus f^*T^*M)$ determines a unique class $\eta\in V^*_G(\eta)$. For $x\in M$, set $Z_x=f^{-1}(x)$. Then, restricting $\eta$ to $Z_x$ we get $\eta_x\in V^*(Z_x)$, which can be viewed as the symbol of some elliptic differential operator $D_x$ on $Z_x$. Then the family $(D_x)_{x\in M}$ is a $G$-equivariant family of elliptic differential operators on $M$, so its $G$-index belongs to $K_*(C^*_r(G,C_0(M)))$ and we set:
$$\tilde{\mu}_{r,M}(Z,f,\xi)=Ind_G(D_x)_{x\in M}.$$
It is stated in Theorem 5 of \cite{BC} that this map $\tilde{\mu}_{r,M}$ is compatible with wrong way Gysin maps, so it descends to a homomorphism of abelian groups:
$$\mu_{r,M}:K^{top}(G,M)\rightarrow K_*(C^*_r(G,C_0(M))),$$
and the main conjecture in \cite{BC} is that $\mu_{r,M}$ is an isomorphism for every Lie group $G$ and every $G$-manifold $M$.

\subsection{The classifying space for proper actions, and its K-homology}\label{classifying}

In the paper \cite{BCH}, P. Baum, A. Connes and N. Higson consider the class of all 2nd countable, locally compact groups $G$. They make a systematic use of the classifying space for proper actions $\underline{EG}$, first introduced in this context in \cite{BC88}. The $G/K$ space associated to a connected Lie group and the Bruhat-Tits building of a $p$-adic group are special cases of classifying space of proper actions as we mentioned already in section \ref{p-adic}. 

\begin{Def} Let $G$ be a 2nd countable locally compact group. A classifying space for proper actions for $G$, is a proper $G$-space $\underline{EG}$ with the properties that, if $X$ is any proper $G$-space, then there exists a $G$-map $X\rightarrow \underline{EG}$, and any two $G$-maps from $X$ to $\underline{EG}$ are $G$-homotopic. 
\end{Def}

When $\Gamma$ is a countable discrete group, we could also define $\underline{E\Gamma}$ as a $\Gamma$-CW-complex such that the fixed point set $\underline{E\Gamma}^H$ is empty whenever $H$ is an infinite subgroup of $\Gamma$, and is contractible whenever $H$ is a finite subgroup (in particular $\underline{E\Gamma}$ is itself contractible).

Back to the general case: even if we refer to $\underline{EG}$ as ``the'' universal space for proper actions of $G$, it is important to keep in mind that $\underline{EG}$ is only unique up to $G$-equivariant homotopy, and the definition of the left hand side $K^{top}_*(G,A)$ will have to account for this ambiguity. So we define
$$K^{top}_*(G,A)=\lim_X KK^G_*(C_0(X),A),$$
where $X$ runs in the directed set of closed, $G$-compact subsets of $\underline{EG}$. This is the left hand side of the assembly map for $G\curvearrowright A$.

\subsection{The Baum-Connes-Higson formulation of the conjecture}\label{official}

For any proper, $G$-compact $G$-space $X$, the space $C_0(X)$ is a module of finite type over the algebra $C^*(G, C_0(X))$ (which is both the full and the reduced one) whose class in  $K_0(C^*(G, C_0(X)))=KK({\bf C},C^*(G, C_0(X)))$ will be denoted by $e_X$. Then for any $G-C^*$-algebra $A$,  Kasparov's descent map $$j_{G,r}:KK_G(C_0(X),A)\rightarrow KK(C^*(G, C_0(X)),C_r^*(G,A))$$
can be composed with the left multiplication by $e_X$ : $$KK(C^*(G, C_0(X)),C_r^*(G,A))\rightarrow KK(\C,C_r^*(G,A))$$
to define a map $KK_G(C_0(X),A)\rightarrow K_*(C_r^*(G,A))$.

When $X$ runs in the directed set of closed, $G$-compact subsets of $\underline{EG}$, those maps are compatible with the direct limit, hence define the {\it assembly map} or {\it index map}:
$$\mu_{A,r}:K^{top}_*(G,A)\rightarrow K_*(C^*_r(G,A)).$$

For $A=\C$, the map $\mu_{A,r}$ is simply denoted by $\mu_r$. The Baum-Connes conjecture is then stated as follows, in its two classical versions:

\begin{Conj}\label{ConjBC}[The Baum-Connes conjecture]
For all locally compact, 2nd countable groups $G$ the assembly map $\mu_r$ is an isomorphism.
\end{Conj}

\begin{Conj}\label{ConjBCcoeff}[The Baum-Connes conjecture with coefficients]
For all locally compact 2nd countable groups $G$ and for all $G$-$C^*$-algebras $A$, the assembly map $\mu_{A,r}$ is an isomorphism. 
\end{Conj}

Conjecture \ref{ConjBCcoeff} has the advantage of being stable under passing to closed subgroups (see \cite{ChabertEchterhoff}), and the disadvantage of being false in general: see sections \ref{countergroupoids} and \ref{exact}.

\medskip
If $G$ is discrete, the classifying space $BG$ classifies actions of $G$ which are free and proper. By forgetting about freeness of the action we get a canonical map
$$\iota_G: K_*(BG)\rightarrow K^{top}_*(G)$$
which is rationally injective. The {\it Strong Novikov conjecture} for $G$ is the rational injectivity of $\mu_r\circ\iota_G$.

\begin{Rem} If $p\in K_*(C^*(G,C_0(X)))=KK_*(\C,C^*(G;C_0(X)))$ is a fixed element, the Kasparov product $p\otimes_{C^*(G,C_0(X))}: x\mapsto p\otimes_{C^*(G,C_0(X))x}$ provides a map $KK_*(C^*(G,C_0(X)),C^*_r(G,A))\rightarrow KK_*(\C,C^*_r(G,A))$. Observe that if $p$ is given by an idempotent of $C^*(G,C_0(X))$, and $x=(E_+,E_-,F)$, with $E_+,E_-$ Hilbert $C^*$-modules over $C^*_r(G,A)$ and $F\in\mathcal{B}_{C^*_r(G,A)}(E_+,E_-)$, then $p\otimes_{C^*(G,C_0(X))}x$ is described simply as $(pE_+,pE_-,pFp)$.  
It turns out that $e_X$ can be described by such an idempotent. Indeed, by properness and $G$-compactness, there exists a Bruhat function on $X$, i.e. a non-negative function $f\in C_c(X)$ such that $\int_G f(g^{-1}x)\,dg=1$ for every $x\in X$. Set then $e(x,g)=\sqrt{f(x)f(g^{-1}x)}$. Recalling that the product in $C_c(X\times G)$ is given by $(a\star b)(x,g)=\int_G a(x,h)b(h^{-1}x,h^{-1}g)\,dh$, one sees immediately that $e^2=e$. Since the set of Bruhat functions is clearly convex, we have a canonical K-theory class $[e_X]\in K_0(C^*(G,C_0(X)))$.
\end{Rem}

\begin{Rem} Assume that $A=\C$. Let $x=(E_+,E_-,F)$ be an element of $KK^G_0(C_0(X),\C)$. Denote by $\pi_\pm$ the representation of $C_0(X)$ on $E_\pm$. Say that $F$ is {\it properly supported} if for every $\phi\in C_c(X)$ there exists $\psi\in C_c(X)$ such that $\pi_-(\psi)F\pi_+(\phi)=F\pi_+(\phi)$. Replacing $F$ by some homotopical operator (so not changing the K-homology class of $(E_+,E_-,F)$, we may assume that $F$ is properly supported. Consider then the linear subspaces $\pi_\pm(C_c(X))E_\pm$ of $E_\pm$: those are not Hilbert spaces in general, but these are $C_c(G)$-modules and $F$ induces a $G$-intertwiner between them. These spaces carry the $C_c(G)$-valued scalar product:
$$\langle \xi,\eta\rangle(g)=:\langle\xi,\rho_\pm(g)\eta\rangle\;(\xi,\eta\in E_\pm),$$
where $\rho_\pm$ denotes the unitary representation of $G$ on $E_\pm$. Completing those spaces into $C^*$-modules over $C^*_r(G)$, and extending $F$ to the completion, we get a triple $\mu_r(x)=(\mathcal{E}_+,\mathcal{E}_-,\mathcal{F})\in KK_*(\C,C^*_r(G))=K_*(C^*_r(G))$, also called the $G$-index of $F$.

The two above approaches, for $A=\C$, were shown to be equivalent in Corollary 2.16 of Part 2 of \cite{MisVal}\footnote{Note that the proof is given there only for discrete groups, but the proof goes over to locally compact group.}. 
\end{Rem}

\begin{Rem}\label{reconcile} It was only in 2009 that P. Baum, N. Higson and T. Schick \cite{BaumHigsonSchick} reconciled the original approach of \cite{BC} with the Kasparov-based approach of \cite{BCH}, in the case of discrete groups. 

For general Lie groups (with arbitrarily many connected components), the equivalence between the approaches in \cite{BC} and \cite{BCH} has not been proved in print so far. However for connected Lie groups both approaches reduce to the Connes-Kasparov conjecture so there is no problem.
\end{Rem}

\begin{Rem} There is also a homotopical approach to the Baum-Connes conjecture, developed by J.L. Davis and W. L\"uck \cite{DL98}; it is valid for discrete groups only. It uses homotopy spectra over the orbit category. More precisely, let $G$ be a group, and denote by $\mathbb{O}_\mathcal{F}(G)$ the category whose objects are homogeneous spaces $G/H$, with $H$ a finite subgroup, and morphisms are $G$-equivariant maps. Equivariant K-homology is obtained by defining some functor from $\mathbb{O}_\mathcal{F}(G)$ to the category of $\Omega$-spectra, extending it to a functor from $G$-spaces to $\Omega$-spectra, and then applying the $i$-th homotopy group to get $K_i^G$ (with $i\geq 0$). It turns out that the value of their functor on $G/H$, for every subgroup $H$ on $G$, is $K_*(C^*_r(G))$. Hence the assembly map, in that framework, is the map functorially associated to the projection $\underline{EG}\rightarrow G/G=\{*\}$. The equivalence with the approach in \cite{BCH} was worked out by I. Hambleton and E. Pedersen \cite{HamPe}.

For the operator algebra inclined reader, we emphasize that the Davis-L\"uck approach, abstract as it may seem, allows for explicit computations of the left-hand side $K_*^{top}(G)$, for $G$ discrete: this is due to the existence of an Atiyah-Hirzebruch spectral sequence relating Bredon homology $H_*^{\mathcal{F}}(\underline{EG}, R_\C)$ to equivariant K-homology. In favorable circumstances (e.g. $\dim\underline{EG}\leq 3$), there are exact sequences allowing one to compute exactly (i.e integrally, not just rationally) $K^{top}_*(G)$ from Bredon homology (see \cite{MisVal}, Theorem I.5.27). For specific classes of groups, the Baum-Connes conjecture can be checked by hand in this way (see e.g. \cite{FPV} for the case of lamplighter groups $F\wr\Z$, with $F$ a finite group).
\end{Rem}

\subsection {Generalizing the $\gamma$-element method}\label{viveTu}
\subsubsection{The case of groups acting on bolic spaces}
The general formulation of the Baum-Connes conjecture suggests the problem of generalizing the $\gamma$-element method, which was first elaborated in the realm of Riemannian symmetric spaces and of their $p$-adic analogues, Bruhat-Tits buildings. Kasparov and Skandalis \cite{Kasparov-Skandalis03}  have explored the case of a combinatorial analogue of simply connected Riemannian manifold with non positive curvature. The good framework is that of
weakly bolic, weakly geodesic metric spaces of bounded coarse geometry (see the definition in their paper). They prove the following:

\begin {Thm} Let $G$ be a group acting properly by isometries on a weakly bolic, weakly geodesic metric space of bounded coarse geometry. Then the Baum-Connes assembly map is injective.
\end{Thm}

The proof involves analogues of the Dirac, dual Dirac and $\gamma$-elements. However $\alpha$ and $\beta$ should no more be thought as defining the Baum-Connes assembly map and the candidate for its inverse. They rather give maps imbedding the K-theory of arbitrary crossed products into the K-theory of crossed products by some proper $G$-algebras, for which the conjecture is known to be true:

\begin{Def} Let $X$ be a $G$-space.  A $G-X-C^*$-algebra is a $G-C^*$-algebra $B$ equipped with a $G$-equivariant homomorphism $C_0(X)\rightarrow Z(M(B))$, the center of the multiplier algebra of $B$.
A $G-C^*$-algebra $B$ is {\it proper} if there exists a proper $G$-space $X$ such that $B$ is a $G-X-C^*$ algebra.
\end{Def}

The following was proved by J. Chabert, S. Echterhoff and R. Meyer \cite{Chabert-Echterhoff-Meyer}\footnote{See also Higson-Guentner \cite[Theorem 2.19]{Higson-Guentner} and Kasparov-Skandalis \cite{Kasparov-Skandalis03}). The case where $G$ is a connected Lie group and $B=C_0(X)$, where $X$ is a proper $G$-space, was previously treated by A. Valette \cite{Valette88}.}:

\begin{Thm}\label{BCproper} The Baum-Connes morphism with coefficients in a {\it proper} $G$-algebra is an isomorphism. 
\end{Thm} 

In the case of a  discrete group $G$ acting properly by isometries on a weakly bolic, weakly geodesic metric space of bounded coarse geometry, Kasparov and Skandalis define a proper algebra $B$, Dirac and dual-Dirac elements $\alpha\in KK_G(B,\C),\beta\in KK_G(\C,B)$ and consider the product $\gamma=\beta\otimes_B \alpha\in KK_{G}(\C,\C)$. In that case, it is no more the case that $\alpha\otimes_{\C}\beta$ is equal to 1 in $KK_G(B,B)$, and this is in fact not needed. However one still has the fact that $\gamma$ becomes 1 when restricted to finite subgroups. This is enough to prove injectivity of the assembly map for such a group $G$.

\subsubsection{Tu's abstract gamma element}
The Kasparov-Skandalis method has been formalized by J.-L. Tu who defined a general notion of $\gamma$ element for a locally compact group, such that the mere existence of $\gamma\in KK_G(\C,\C)$ implies the injectivity of the Baum-Connes map, and that the surjectivity is equivalent to the fact that $\tilde\gamma_A={\rm Id}$ with notations as in Theorem \ref{StrongNovikov}. The techniques use the representable KK-theory of Kasparov and can also be beautifully interpreted in the framework of equivariant KK-theory for groupoids as introduced by P.-Y. Le Gall \cite{LeGall99}.
See Chapter \ref{groupoids} below for details on the groupoid framework.

\begin{Def} A \emph{$\gamma$-element} for $G$ is an element $\gamma$ of the ring $KK_{G}(\C,\C)$ satisfying the following two conditions:
\begin {enumerate}
\item[1)] there exists a proper $G$-$C^*-$algebra $B$
and two elements $\alpha\in KK_G(B,\C)$ and $\beta\in KK_G(\C,B)$ such that $\gamma=\beta\otimes_B \alpha\in KK_{G}(\C,\C)$;
\item[2)]  for any compact subgroup $K$ of $G$, the image of $\gamma$ by the restriction map $KK_G({\bf C}, {\bf C})\rightarrow R(K)$ is the trivial representation $1_K$.
\end{enumerate}
\end{Def}

\begin{Rem} The second condition is technically formulated as follows: for any 
proper $G$-space $X$, we have $p_*(\gamma)=1$ in $RKK_{G}(X;\C,\C)$ (where $p_*$ denotes the induction  homomorphism $KK_{G}(\C,\C)\rightarrow RKK_{G}(X;\C,\C)$). The notations are as follows.
For $X$ a $G$-space, $A$ and $B$ two $G-X-C^*$-algebras, Kasparov defines ${\cal R}KK_G(X;A,B)$ as the set of homotopy classes of $(A,B)$-Fredhom bimodules equipped with a covariant action of the $C^*$-algebra $C_0(X)$, with the usual assumption of compactness of commutators. The beautiful language of groupoids allows to think of $A$ and $B$ as ${\cal G}-C^*$-algebras with ${\cal G}=X\rtimes G$ the groupoid given by the action of $G$ on $X$. Then 
$${\cal R}KK_G(X;A,B)=KK_{\cal G}(A,B).$$
Now for two $G-C^*$-algebras $A$ and $B$ (no action of $C_0(X)$ is needed), Kasparov defines $$RKK_G(A,B)={\cal R}KK_G(X;A\otimes C_0(X),B\otimes C_0(X))=KK_{\cal G}(A\otimes C_0(X),B\otimes C_0(X)).$$
In the definition of a $\gamma$-element, the map $$ p_*: KK_{G}(\C,\C)\rightarrow RKK_{G}(X;\C,\C)$$ is the pull-back by the groupoid homomorphism $p:{\cal G}=X\rtimes G\rightarrow G$.
Note that if $X=G/K$ with $K$ a compact subgroup, then $RKK_{G}(X;\C,\C)=R(K)$.
\end{Rem}

\medskip
J.-L. Tu has proved the following \cite{Tu-survey}:

\begin{Prop} If an element $\gamma$ exists, then it is unique. Moreover, it is an idempotent of the ring $KK_{G}(\C,\C)$, namely $\gamma\otimes_{\C}\gamma = \gamma$.
\end{Prop}

Observe that, if a $\gamma$-element does exist, then it acts as the identity on any group $K^{top}_*(G,A)$, for every $G-C^*$-algebra $A$. 
The relation with the Baum-Connes conjecture can be stated as follows 

\begin{Thm}\label{abstractgamma}[Theorem 4.2 and 4.4 \cite{Tu99}]
Let $G$ be a locally compact group admitting a $\gamma$-element. 
\begin{enumerate} 
\item[1)]The map $\mu_{A,r}$ is injective for every $G$-$C^*$-algebra $A$. 
\item[2)] The map $\mu_{A,r}$ is surjective if and only if the map $\tilde\gamma_A$ (i.e. Kasparov product by $j_{G,r}(\tau_A(\gamma))$) is the identity on $K_*(C^*_r(G,A))$. This is in particular true if $\gamma=1$.
\end{enumerate}
\end{Thm}

\begin{proof}[Proof] Let $\gamma=\beta\otimes_B\alpha$ be a $\gamma$-element, with $B$ a proper $G-C^*$-algebra. Let $A$ be any $G-C^*$-algebra. Then we have a commutative diagram:

$$\xymatrix{ K^{top}(G,A)\ar[rr]^{\scriptscriptstyle\otimes_A \tau_A(\beta)}\ar[d]^{\mu_{A,r}}& & K^{top}_*(G,A\otimes B)\ar[rrr]^{\scriptscriptstyle{\otimes_{A\otimes B}}(\tau_{A}(\alpha))}\ar[d]^{\mu_{A\otimes B,r}}_{\simeq}&&& K^{top}(G,A)\ar[d]^{\mu_{A,r}} \\ 
K_*(C^*_r(G,A))\ar[rr]^{\scriptscriptstyle\otimes_{C^*_r(G,A)}j_G(\tau_A(\beta))}&  & K_*(C^*(G;A\otimes B))\ar[rrr]^{\scriptscriptstyle\otimes_{\scriptscriptstyle{C^*(G,A\otimes B)}}j_G(\tau(\alpha))}&& &K_*(C^*_r(G,A)),}
$$

with $j_G$ the descent map as in section \ref{bifunctor}. Since $A\otimes B$ is a proper $G-C^*$-algebra, the map $\mu_{A\otimes B,r}$ is an isomorphism, by Theorem \ref{BCproper}. The assumption in (1) is that the composition of the two maps on the top row is the identity: this implies that $\mu_{A,r}$ is injective. The assumption in (2) is that moreover the composition of the two maps on the bottom row is the identity: this implies that $\mu_{A,r}$ is also surjective.
\end{proof}

\begin{Rem}
The element $\gamma$ initially defined by Kasparov in \cite{KaspConsp} is of course a special case of $\gamma$-element in the sense of Tu. Note that if $K$ is a maximal compact subgroup of a connected Lie group $G$, the element $\gamma$ is simply characterized by the conditions (cf.  Proposition 4.1 in \cite{Tu-survey}) that 
 it factorizes through a proper  $G-C^*$-algebra and that the image of $\gamma$ by the restriction map $KK_G({\bf C}, {\bf C})\rightarrow R(K)$ is the trivial representation $1_K$.
\end{Rem}

\subsubsection{Nishikawa's new approach}

Very recently (March 2019), Nishikawa \cite{Nishikawa} introduced a new idea in the subject, that amounts to constructing the $\gamma$ element {\it without} having to construct the Dirac and dual Dirac elements. We briefly explain his approach. The standing assumption is that the group $G$ admits a cocompact model for $\underline{EG}$ (in particular $\underline{EG}$ is locally compact).

\begin{Def} Let $x$ be an element of $KK_G(\C,\C)$.
Say that $x$ has property $(\gamma)$ if it can be represented by a Fredhom module $KK_G(\C,\C)$ such that:
\begin{enumerate}
\item For every compact subgroup $K$ of $G$, $x$ restricts to $1_K$ in $R(K)$.
\item The Hilbert space $\HH$ carries a $G$-equivariant non-degenerate representation of $C_0(\underline{EG})$ such that, for every $f\in C_0(\underline{EG})$, the map $g\mapsto[g(f),T]$ is a norm continuous map vanishing at infinity on $G$, with values in the ideal of compact operators.
\item Moreover,  the integral $$\int_G g(c)Tg(c)dg-T=-\int_G g(c)[g(c),T]dg$$ is compact, where $c$ is a compactly supported function on ${\bf E}G$ such that $\int_Gg(c)^2dg=1$.
\end{enumerate}
\end{Def}

It is not known whether the technical condition 3 follows from condition 2 or is really needed.
Nishikawa shows that such a Fredholm module allows to define, for every $G-C^*$-algebra $A$, a map $\nu_A^x:K_*(C^*_r(G,A))\rightarrow K_*^{top}(G,A)=KK_G(C_0(\underline{EG}),A)$, which is a left inverse for the assembly map $\mu_{A,r}$. One has the following theorem:

\begin{Thm} Assume that there exists a Fredholm module $x=(\HH,F)$ with property $(\Gamma)$. Then:
\begin{enumerate}
\item For every $G-C^*$-algebra $A$, the map $\mu_{A,r}$ is injective.
\item For every $G-C^*$-algebra $A$, the map $\mu_{A,r}$ is surjective if and only if the element $x$ defines the identity on $K_*(C^*_r(G,A))$. In particular, 
if $x=1$ in $KK_G(\C,\C)$,  Conjecture \ref{ConjBCcoeff} holds for $G$.
\end{enumerate}
\end{Thm}

Nishikawa also proves the following result:

\begin{Thm} \begin{enumerate}
\item If there exists an element $x$ of $KK_G(\C,\C)$ with property $(\gamma)$, then it is unique and is an idempotent in $KK_G(\C,\C)$. 
\item If $G$ admits a  $\gamma$ element in the sense of Tu, then $x=\gamma$ has the $(\gamma)$ property. 
\end{enumerate}
\end{Thm}

In particular, in the case of groups admitting an abstract $\gamma$ element, any element with the $(\gamma)$ property is in fact equal to $\gamma$.

Using this new approach, Nishikawa can reprove Conjecture \ref{ConjBCcoeff} for Euclidean motion groups, as well as the injectivity of the Baum-Connes map with coefficients $\mu_{A,r}$ for $G$ a semisimple Lie group. He also reproves the conjecture
for groups acting properly on locally finite trees and announces a generalization (with J. Brodzki, E. Guentner and N. Higson) to groups acting properly on $CAT(0)$ cubic complexes.

\subsection{Consequences of the Baum-Connes conjecture}\label{BCconsequences}

\subsubsection{Injectivity: the Novikov conjecture}\label{InjectivityNovikov}

In section \ref{Novikov}, we already emphasized that the Novikov conjecture (conjecture \ref{Nov}) on homotopy invariance of higher signature followed from the (rational) injectivity of Kasparov's map 
$$\beta:K_0(B\Gamma)\rightarrow K_0(C^*_r(\Gamma)).$$
In the case of a cocompact, torsion-free lattice of a connected Lie group $G$, the map $\beta$ coincides with the Dirac induction map
$$K_0(M)\rightarrow K_0(C^*_r(\Gamma))$$
of section \ref{bifunctor}. In general there is a natural injection group $\iota_{\Gamma}: K_0(B\Gamma)\rightarrow K_0^\Gamma(\underline{E\Gamma})$ and its composition with the assembly map $\mu_r$ gives $\beta$. 
That fact, taken for granted for a long time, was proved only fairly recently by M. Land \cite{Land}. 

Therefore, the Novikov conjecture follows from the Strong Novikov conjecture, i.e. from the rational injectivity of the map $\mu_r\circ\iota_{\Gamma}$.
In particular, the Novikov conjecture follows from the injectivity of the assembly map $\mu_r$. 

We must here mention the beautiful recent approach of P. Antonini, S. Azzali and G. Skandalis \cite{AAS} on K-theory with coefficients in the real numbers. They make use of von Neumann theory of $II_1$-factors. For such a factor $N$, 
the trace defines naturally an isomorphism from $K_0(N)$ to ${\R}$ whereas $K_1(N)=0$. The  KK-theory with real coefficients $KK^G_{\R}(A,B)$ is defined as the inductive limit:
of the groups $KK_G(A,B\otimes N)$ for all $N$ a $II_1$-factors $N$. Note that there is a map $KK^G(A,B)\otimes\R\rightarrow KK^G_{\R}(A,B)$ but it is in general not an isomorphism. Any trace on $A$ defines an element of $KK_{\R}(A,\C)$. 
In particular for $\Gamma$ a discrete group, the canonical trace $\tau$ defines an element $[\tau]$ of $KK_{\R}(C^*_r(\Gamma),\C)=KK^{\Gamma}_{\R}(\C,\C)$. The crucial remark of \cite{AAS}  is the following:
\begin{Prop} The element $[\tau]$ is an idempotent of the ring $KK^{\Gamma}_{\R}(\C,\C)$. Moreover for any proper and free space $X$, the identity $1_{C_0(X)}$ of the ring $KK^{\Gamma}_{\R}(\C,\C)$ saisfies $1_{C_0(X)}\otimes [\tau]=1_{C_0(X)}$.
\end{Prop}
The authors define the $KK_{\R}$-groups localized at the identity as the products by the idempotent $[\tau]$, i.e.
$KK^{\Gamma}_{\R}(A,B)_{\tau}=KK^{\Gamma}_{\R}(A,B)\otimes_{\C} [\tau].$
In particular the Baum-Connes map can be localized as  $$\mu_{\tau}: K^{top}_{*,{\R}}(\Gamma)_{\tau}\rightarrow K_{*,{\R}}(C^*_r(\Gamma))_{\tau}$$, where the righthand side is nothing but $KK^{\Gamma}_{\R}({\C},C^*_r(\Gamma))_{\tau}$
and the lefthand side is  $KK^{\Gamma}_{\R}(C_0(X),{\C})_{\tau}$ (assume for simplicity that ${\underline E\Gamma}$ is cocompact).

The results of \cite{AAS} can be summarized as follows
\begin{Thm} 
\begin{enumerate}Let $\Gamma$ be a discrete group. 
\item If the Baum-Connes conjecture (with coefficients) holds for $\Gamma$, then $\mu_{\tau}$ is an isomorphism. 
\item If the map $\mu_{\tau}$ is injective, then the Strong Novikov conjecture holds for $\Gamma$. 
\end{enumerate}
\end{Thm}

The first point uses the Baum-Connes map with coefficients in any $II_1$-factor. 
The second point rests upon the observation that the map from $E\Gamma$ to $\underline E\Gamma$ induces an isomorphism from 
$$K_*(B\Gamma)\otimes{\R}=KK^{\Gamma}_{\R}(C_0(E\Gamma), {\C})
\rightarrow K^{top}_{*,{\R}}(\Gamma)_{\tau}
=KK^{\Gamma}_{\R}(C_0({\underline E}{\Gamma}),{\C})_{\tau}.$$

In other words, the conjecture that $\mu_{\tau}$ is an isomorphism is intermediate between the Baum-Connes conjecture (without coefficents) and the Strong Novikov conjecture.

\subsubsection{Injectivity: the Gromov-Lawson-Rosenberg conjecture}
Let $M$ be a Riemannian manifold of dimension $n$. The {\it scalar curvature} is a smooth function $\kappa:M\rightarrow\R$ that, at a point $p\in M$, measures how fast the volume of small balls centered at $p$ grows when compared to the volume of small balls of the same radius in Euclidean space $\mathbf{E}^n$. More precisely we expand the ratio $\frac{Vol\,B_M(p,r)}{Vol\,B_{\mathbf{E}^n}(0,r)}$ as a power series in $r$:
$$\frac{Vol\,B_M(p,r)}{Vol\,B_{\mathbf{E}^n}(0,r)}=1 -\frac{\kappa(p)}{6(n+2)}r^2 + o(r^2);$$
so positive scalar curvature means that small balls in $M$ grow more slowly than corresponding Euclidean balls.

Let $M$ be now a closed spin manifold, and $D$ the Dirac operator of $M$, the Atiyah-Singer index formula for $D$ is
$$Ind(D)=\langle {\bf \hat{A}(M)},[M]\rangle,$$
where ${\bf \hat{A}(M)}$ is a polynomial in the Pontryagin classes, and $[M]$ is the fundamental class of $M$; see \cite{BoossBleecker}. Let $\Gamma=\pi_1(M)$ be the fundamental group of $M$, and let $f:M\rightarrow B\Gamma$ be the classifying map. Fix $x\in H^*(B\Gamma,\Q)$. The number $\langle {\bf \hat{A}(M)},[M]\rangle$ being called the {\it $\hat{A}$-genus}, it is natural to call the numbers 
$$\hat{A}_x(M)=:\langle {f^*(x)\cup\bf \hat{A}(M)},[M]\rangle$$ 
{\it higher $\hat{A}$-genera}, by analogy with higher signatures. The {\it Gromov-Lawson-Rosenberg conjecture} (GLRC) states:

\begin{Conj}\label{GLRC}(GLRC) Let $M$ be a closed spin manifold $M$ with fundamental group $\Gamma$. If $M$ admits a Riemannian metric with positive scalar curvature, then all higher $\hat{A}$-genera do vanish: $\hat{A}_x(M)=0$ for all $x\in H^*(B\Gamma,\Q)$.
\end{Conj}

GLRC for manifolds with given fundamental group $\Gamma$, follows from injectivity of the assembly map for $\Gamma$, see Theorem 7.11 in \cite{BCH}. The fact that the usual $\hat{A}$-genus vanishes for a closed spin manifold with positive scalar curvature, is a famous result by Lichnerowicz.

See \cite{RosenbergStolz} for a lucid discussion of GLRC, together with speculations about a suitable converse: does the vanishing of a certain K-theory class in the real K-theory of $C^*_r(\Gamma)$ implies the existence of a metric with positive scalar curvature on $M$?

\subsubsection{Surjectivity: the Kadison-Kaplansky conjecture}

Let $\Gamma$ be a discrete group. If $g\in\Gamma$ is a group element of finite order $n>1$, then $e=\frac{1}{n}\sum_{k=0}^{n-1}g^k$ defines a non-trivial element in the complex group ring $\C\Gamma$ (``non-trivial'' meaning: distinct from 0 and 1). When $\Gamma$ is torsion-free, it is not clear that $\C\Gamma$ admits non-trivial idempotents, and around 1950, I. Kaplansky turned this into a conjecture:

\begin{Conj}\label{Kap} If $\Gamma$ is a torsion-free group, then $\C\Gamma$ has no non-trivial idempotent.
\end{Conj}

Around 1954, R.V. Kadison and I. Kaplansky conjectured that this should be even true by replacing $\C\Gamma$ by the larger reduced group $C^*$-algebra:

\begin{Conj}\label{KadKap} If $\Gamma$ is a torsion-free group, then $C^*_r(\Gamma)$ has no non-trivial idempotent.
\end{Conj}

In contrast with the Novikov conjecture (Conjecture \ref{Nov}), Conjecture \ref{KadKap} is easy to state. It is interesting that it follows too from the Baum-Connes conjecture (Conjecture \ref{ConjBC}), actually from the surjectivity part.

\begin{Prop}\label{BCimpliesKK} Let $\Gamma$ be a torsion-free group. If the assembly map $\mu_r$ is onto, then Conjecture \ref{KadKap} holds for $\Gamma$.
\end{Prop}

The proof of Proposition \ref{BCimpliesKK} goes through an intermediate conjecture. To motivate this one, recall that any trace $\sigma$ on a complex algebra $A$ defines a homomorphism 
$$\sigma_*:K_0(A)\rightarrow\C:[e]\mapsto (Tr_n\otimes\sigma)(e)$$
where $e=e^2\in M_n(A)$ and $Tr_n:M_n(A)\rightarrow A$ is the usual trace. If $A$ is a $C^*$-algebra and $\sigma$ is a positive trace, then the image of $\sigma_*$ is contained in $\R$. Consider now the canonical trace $\tau$ on $C^*_r(\Gamma)$. The following conjecture is known as {\it conjecture of integrality of the trace}.

\begin{Conj}\label{integraltrace} If $\Gamma$ is a torsion-free group, then the canonical trace $\tau_*$ maps $K_0(C^*_r(\Gamma))$ to $\Z$.
\end{Conj}

It is then easy to see that Conjecture \ref{integraltrace} implies the Kadison-Kaplansky conjecture (Conjecture \ref{KadKap}). Indeed, take $e=e^2\in C^*_r(\Gamma)$. Since an idempotent in a unital $C^*$-algebra is similar to a projection, we may assume that $e=e^*=e^2$. As $0\leq e\leq 1$ and $\tau$ is a positive trace, we have $0\leq\tau(e)\leq 1$. But $\tau(e)\in\Z$ by Conjecture \ref{integraltrace}, so $\tau(e)$ is either 0 or 1. If $0=\tau(e)=\tau(e^*e)$, then $e=0$ by faithfulness of $\tau$. Replacing $e$ by $1-e$, we see that if $\tau(e)=1$, then $e=1$.

\begin{proof}[Proof of Proposition \ref{BCimpliesKK}] By the previous remarks, it is enough to see that, for a torsion-free group $\Gamma$ such that $\mu_r$ is onto, Conjecture \ref{integraltrace} holds. Actually we prove that, assuming $\Gamma$ to be torsion-free, $\tau_*$ is always integer-valued on the image of $\mu_r$ in $K_0(C^*_r(\Gamma))$.

Thanks to Remark \ref{reconcile}, the domain of $\mu_r$, i.e. the left-hand side of the Baum-Connes conjecture, is the group $K_0(\Gamma,pt)$, whose cycles are of the form $(Z,\xi)$ with $Z$ a proper $\Gamma$-compact manifold and $\xi\in V_\Gamma(T^*Z)$, and by section \ref{timedep} we have $\mu_r(Z,\xi)=Ind_\Gamma(\tilde{D})$, where $\tilde{D}$ is some $\Gamma$-invariant elliptic differential operator on $Z$. As $\Gamma$ is torsion-free, any proper $\Gamma$-action is free and proper, so that the map $Z\rightarrow\Gamma\backslash Z$ is a $\Gamma$-covering and we may appeal to Atiyah's $L^2$-index theorem (Theorem \ref{L2index}): the operator $\tilde{D}$ descends to an elliptic operator on the compact manifold $\Gamma\backslash Z$ and
$$\tau_*(\mu_r(Z,\xi))=Ind_\Gamma(\tilde{D})=Ind(D).$$
Since $Ind(D)\in\Z$, this concludes the proof.\footnote{For a nice proof of that result NOT appealing to Atiyah's $L^2$-index theorem, see lemma 7.1 in \cite{MisVal}.}
\end{proof}

\subsubsection{Surjectivity: vanishing of a topological Whitehead group}

For a group $\Gamma$, denote by $\Z\Gamma$ its integral group ring, and let 
$$K_1^{alg}(\Z\Gamma)=:\varinjlim GL_n(\Z\Gamma)/E_n(\Z\Gamma)$$
be the first algebraic K-theory group of $\Z\Gamma$, where $E_n(\Z\Gamma)$ is the subgroup of elementary matrices. We denote by $[\pm \Gamma]$ the subgroup of $K_1^{alg}(\Z\Gamma)$ generated by the image of the inclusion of $\Gamma\times\{\pm 1\}$ into $GL_1(\Z\Gamma)$. The {\it Whitehead group} $Wh(\Gamma)$ is then
$$Wh(\Gamma)=K_1^{alg}(\Z\Gamma)/[\pm \Gamma].$$
By analogy, using the inclusion of $\Gamma$ in the unitary group of $C^*_r(\Gamma)$, we may define the {\it topological Whitehead group} as $Wh^{top}(\Gamma)=:K_1(C^*_r(\Gamma))/[\Gamma]$. So the vanishing of $Wh^{top}(\Gamma)$ is equivalent to to the fact that every unitary matrix in $M_\infty(C^*_r(\Gamma))$ is in the same connected component as a diagonal matrix $diag(\gamma,1,1,1,...)$ for some $\gamma\in\Gamma$.

\begin{Conj}\label{Whitehead} Assume that there is a 2-dimensional model for $B\Gamma$. Then $Wh^{top}(\Gamma)=0$.
\end{Conj}

The following result appears in \cite{BettaiebMattheyValette}:

\begin{Prop} When $\Gamma$ has a 2-dimensional model for $B\Gamma$, conjecture \ref{Whitehead} follows from the surjectivity of the assembly map $\mu_r$.
\end{Prop}

\begin{proof}[Proof] Let $\Gamma^{ab}$ denote the abelianization of $\Gamma$. The inclusion of $\Gamma$ in the unitary group of $C^*_r(\Gamma)$ induces a map $\beta:\Gamma^{ab}\rightarrow K_1(C^*_r(\Gamma))$, as $K_1$ is an abelian group.

By lemma 7.5 in \cite{BettaiebMattheyValette}, as $B\Gamma$ is 2-dimensional, the Chern character $Ch:K_1(B\Gamma)\rightarrow H_1(B\Gamma,\Z)$ is an isomorphism. Of course we have $H_1(B\Gamma,\Z)=H_1(\Gamma,\Z)=\Gamma^{ab}$. Moreover we have a commutative diagram
$$\xymatrix{K_1(B\Gamma)\ar[dr]^{Ch}\ar[rr]^{\mu_r}   & & K_1(C^*_r(\Gamma)) \\  
&\Gamma^{ab}\ar[ur]^{\beta} &}.$$

So $\beta$ is onto as soon as $\mu_r$ is, and this implies $Wh^{top}(\Gamma)=0$
\end{proof}

\subsubsection{Surjectivity: discrete series of semisimple Lie groups}

Let $G$ be a semisimple connected Lie group with finite centre and maximal compact subgroup $K$. As we shall see in Theorem \ref{LafforgueSS} below, Lafforgue has given a proof of the Baum-Connes conjecture without coefficients for $G$ which is independent of Harish-Chandra theory. On the other hand, let us present here a beautiful argument, also due to Lafforgue \cite{LafforgueICM02}, showing that the surjectivity of the assembly map does say something on the representation theory: namely, surjectivity implies that the Dirac induction $\mu_G$ maps bijectively a subset of the dual $\hat K$ to the discrete series of $G$; compare with Remark \ref{CoKadiscrete}.

Recall that semisimple groups are CCR, i.e. any unitary irreducible representation $\sigma$ of $G$ maps $C^*(G)$ onto the compact operators on $\HH_\sigma$; so in K-theory $\sigma$ induces a homomorphism $\sigma_*:K_0(C^*(G))\rightarrow {\Z}$.

As the main ingredient for Lafforgue's observation, we just need to recall from Remark \ref{CoKadiscrete} that any discrete series $\pi$ of $G$ defines a K-theory class $[\pi]\in K_0(C^*_r(G))$  such that $\pi_*([\pi])=1$. In particular $[\pi]\neq 0$. Note that if $G/K$ is odd dimensional, then the surjectivity part of the conjecture implies that $K_0(C^*_r(G))=0$ so that $G$ has no discrete series, reproving a well-known fact in Harish-Chandra theory. We therefore now assume that $G/K$ has even dimension.

Assume for simplicity that $G/K$ has a $G$-invariant spin structure, i.e. the adjoint representation of $K$ in $V=\mathfrak{g}/\mathfrak{k}$ lifts to $Spin (V)$. 
The Connes-Kasparov map $\mu_G$ then coincides with Kasparov's Dirac map $\tilde\alpha: R(K)=K_0(C^*(K))\rightarrow K_0(C^*_r(G))$. The inverse of the map is Kasparov's dual Dirac map $\tilde\beta : K_0(C^*_r(G))\rightarrow R(K)$. Lafforgue's observation is the following duality:

\begin{Lem}\label{Laffdual} For any discrete series $\pi$ of $G$ and any irreducible representation $\rho$ of $K$, the following integers are equal:
$$\pi_*(\tilde\alpha ([\rho])) =\rho_*(\tilde\beta ([\pi]))$$
\end{Lem}

Indeed, one can show that both are equal to the dimension of the intertwining space ${\rm Hom}_K(S\otimes V_{\rho},H_{\pi})$ where $S$ is the spinor representation of $K$.

Fix $\pi$ a discrete series of $G$. Viewing $R(K)$ as the free abelian group on $\hat{K}$, we may write 
$$\tilde\beta ([\pi])=\sum_{\rho\in\hat{K}} n_{\pi,\rho}[\rho],$$ 
 where $n_{\pi,\rho}$ is the integer defined in two different ways in lemma \ref{Laffdual}.
Now the assumed surjectivity of $\mu_G$ translates into $\tilde\alpha\circ\tilde\beta={\rm Id}$, which implies the following decomposition in $K_0(C^*_r(G))$: 
$$[\pi]=\sum_{\rho}n_{\pi,\rho}\tilde\alpha([\rho]).$$
Now the equality $\pi_*([\pi])=1$ and lemma \ref{Laffdual} yield: 
$$1=\sum_{\rho}n_{\pi,\rho}\pi_*(\tilde\alpha([\rho]))=\sum_{\rho}n_{\pi,\rho}^2.$$

So the integers $n_{\pi,\rho}$ satisfy $\sum_{\rho}n_{\pi,\rho}^2=1$, hence there is a unique $\rho$ such that $n_{\pi,\rho}=\pm 1$, the others being zero. Then $\tilde\alpha([\rho])=\pm [\pi]$, and the Dirac induction maps bijectively a subset of the dual $\hat K$ to the discrete series of $G$; in other words, we have recovered Theorem \ref{AtiyahSchmid} in a qualitative way.

\section{Full and reduced $C^*$-algebras}

\subsection{Kazhdan vs. Haagerup: property (T) as an obstruction}\label{KvsH}

The assembly map could as well be constructed using maximal $C^*$-algebras instead of reduced. There is indeed a map 
$$\mu_{A,{\rm max}}:K_*^G(\underline{EG},A)\rightarrow K_*(C^*_{\rm max}(G,A))$$
so that $\mu_{A,r}$ is the composition of $\mu_{A,{\rm max}}$ with the map $\lambda_A^*$ obtained by functoriality in K-theory from the map $$\lambda_A:C^*_{\rm max}(G,A)\rightarrow C^*_{r}(G,A).$$

In other words we have a commutative diagram

$$\xymatrix{s
K_*^G(\underline{EG},A)\ar[drr]_{\mu_{A,r}}\ar[rr]^{\mu_{A,{\rm max}}}&&K_*(C^*_{\rm max}(G,A))\ar[d]^{\lambda_A^*}\\
&&K_*(C^*_{r}(G,A))),}
$$

The main difficulty in that the Baum-Connes conjecture is a conjecture about  $\mu_{A,{r}}$, not  $\mu_{A,{\rm max}}$. In order to understand that crucial point, it will be enlightening to consider two classes of groups: one for which both 
$\mu_{A,r}$ and $\mu_{A,{\rm max}}$ are isomorphisms, hence also $\lambda_A^*$.; another for which $\lambda_A^*$ is not injective, $\mu_{A,{\rm max}}$ not surjective, and for which the conjectural bijectivity of $\mu_{A,r}$ is difficult and proved only in very few cases. We refer to \cite{Julg98} for more details.

\begin {Def} 
A locally compact second countable group $G$ has the Haagerup property\footnote{Or is a-(T)-menable, according to M. Gromov.} if the following equivalent conditions are satisfied:
\begin{enumerate}
\item[(i)] There exists an action of $G$ by affine isometries on a Hilbert space which is metrically proper.
\item[(ii)] There exists a unitary representation $\pi$ of $G$ on a Hilbert space $\HH$, and a 1-cocycle (i.e. a map $b:G\rightarrow \HH$ such that $b(gg')=b(g)+\pi (g)b(g')$) which is proper.
\item[(iii)] There exists a function of conditional negative type on $G$ which is proper.
\end{enumerate}
\end {Def}

\begin {Def} A locally compact second countable group $G$ has Kazhdan's property (T) if the following equivalent conditions are satisfied:
\begin{enumerate}
\item[(i)] Any action of $G$ by affine isometries on a Hilbert space admits a fixed point.
\item[(ii)] For any unitary representation $\pi$ of $G$ on a Hilbert space $\HH$, any 1-cocycle is bounded.
\item[(iii)] Any function of conditional negative type on $G$ is bounded.
\end{enumerate}
\end {Def}

Note that only compact groups are both Haagerup and Kazhdan. The above definitions can also be expressed in terms of the {\it almost invariant vectors} property: a unitary representation $\pi$ of $G$ on $\HH$ almost admits invariant vectors if for any $\varepsilon >0$ and any compact subset $C$ of $G$, there is a unit vector $x\in\HH$ such that
$\Vert \pi (g)x-x\Vert\leq\varepsilon$ for any $g\in C$.

\begin {Prop}\label{almostinv} A  locally compact group $G$ has property (T) if and only if any unitary representation almost admitting invariant vectors has a non zero invariant vector. It has the Haagerup property if and only if there exists a unitary representation with coefficients vanishing at infinity and almost admitting invariant vectors.
\end {Prop}

The above characterization of property (T) is the original definition of Kazhdan. As to the characterization of the Haagerup property, it is due to P. Jolissaint and implies that all amenable groups have the Haagerup property. 
For examples of groups having Haagerup or Kazhdan property, we refer to \cite{BHV} and to \cite{CCJJV}. Typical examples of non amenable discrete groups with Haagerup property are the free groups $F_n (n\geq 2 )$ or $SL_2({\bf Z})$, whereas typical discrete groups having Kazhdan property are $SL_n({\bf Z})$, $n\geq 3$.

Let us now explain the link with the Baum-Connes conjecture. We begin with a $C^*$-algebraic characterization of property (T) (see \cite{AkemannWalter}), in terms of the existence of a {\it Kazhdan projection}.

\begin{Prop}\label{Kazhdanproj} The locally compact group $G$ has property (T) if and only if there exists an idempotent $e_G\in C^*_{\rm max}(G)$ such that, for every unitary representation $\pi$ of $G$, the idempotent $\pi(e_G)$ is the orthogonal projector on the space of $\pi(G)$-fixed vectors in $\HH_\pi$.
\end{Prop}

From this we deduce a key observation made by A. Connes in the early 1980's: let us consider, for a locally compact group, the map  $\lambda : C^*_{\rm max}(G)\rightarrow C^*_{r}(G)$ associated with the left regular representation of $G$.

\begin {Lem}
If $G$ is non compact with property (T), the map induced in K-theory 
$$\lambda_*: K_*(C^*_{\rm max}(G))\rightarrow K_*(C^*_{r}(G))$$
is not injective: its kernel contains a copy of ${\bf Z}$ which is a direct summand in $K_0(C^*_{\rm max}(G))$.
\end {Lem}

\begin{proof}[Proof] Because of property (T), we have a direct sum decomposition 
$$C^*_{\rm max}(G)=\ker(\epsilon_G)\oplus\C e_G,$$
where $\epsilon_G$ is the trivial one-dimensional representation of $G$. So $K_0(C^*_{\rm max}(G))=K_0(\ker(\epsilon_G))\oplus\mathbf{Z}$. On the other hand, as $G$ is not compact: $\lambda_*(e_G)=0$, which ends the proof.
\end{proof}

\begin {Cor}\label{TnotKamen}
Assume that $G$ is non compact with property (T), and admits a $\gamma$-element. Then $\mu_{\rm max}$ is not surjective. In particular, $\gamma\neq 1$ in $KK_G({\bf C}, {\bf C})$.
\end {Cor}

\begin{proof}[Proof] We have $\mu_r=\lambda_*\circ \mu_{\rm max}$, and the injectivity of $\mu_r$ (see Theorem \ref{abstractgamma}) trivially implies that a non zero element of the kernel of $\lambda_*$ cannot be in the image of $\mu_{\rm max}$. Moreover, if $\gamma=1$, the Kasparov machine, which works also for full crossed products, shows that $\mu_{A,{\rm max}}$ is an  isomorphism, a contradiction.
\end{proof}

\medskip
On the other hand, Higson and Kasparov have proved in the 1990's the following beautiful result: 

\begin {Thm}\label{HigsonKasparov}
Let $G$ be a locally compact group having the Haagerup property. Then $G$ has a $\gamma$-element equal to 1. As a consequence, the three maps $\mu_{A,r}$, $\mu_{A,{\rm max}}$ and $(\lambda_A)_*$ are isomorphisms. In particular Conjecture \ref{ConjBCcoeff} holds for $G$.
\end{Thm}

For a proof (using $E$-theory instead of KK-theory) we refer to  \cite{HigKas} and \cite{Julg98}. We shall only explain how a locally compact proper $G$-space can be constructed from an affine action on a Hilbert space. Consider the space $Z=H\times [0,+\infty[$ equipped by the topology pulled back by the map $(x,t)\mapsto (x, \sqrt {\Vert x\Vert^2+t^2})$ of the topology of the space $H_w\times [0,+\infty[$ where $H_w$ is the space $H$ with weak topology. The space  $Z$ is a locally compact space and carries a proper action defined by $g.(x,t)=(g.x,t)$ for $g\in G$. It is a representative of the classifying space of proper actions $\underline{EG}$.
The space $Z$ can also be defined as a projective limit of spaces 
 $[0,+\infty [\times V$ over all affine subspaces $V$ of $H$, with the maps $[0,+\infty [\times V'\rightarrow [0,+\infty [\times V$ (for all $V\subset V'$) combining the projection to $V$ with the map $x\mapsto \sqrt {\Vert x\Vert^2+t^2}$ on the vector subspace orthogonal to $V$ in $V'$.

A locally compact group $G$ is {\it K-amenable} (see e.g.  \cite{JV84}) if, for any $G-C^*$-algebra $A$, the  full crossed product $C^*_{\rm max}(G,A)$ and the reduced crossed product $C^*_r(G,A)$ do have the same K-theory via the map $(\lambda_A)_*$. So Theorem \ref{HigsonKasparov} says that groups with the Haagerup property are K-amenable, while Corollary \ref{TnotKamen} says that non-compact groups with property (T), are not.

\begin{Rem}
In a recent preprint, S. Gong, J. Wu and G. Yu \cite{GWY} prove the Strong Novikov conjecture for discrete groups acting isometrically and metrically properly on a Hilbert-Hadamard manifold (i.e. an infinite dimensional analogue of simply connected and non-positively curved manifold). This contains of course the case of groups with the Haagerup property, but also the case of geometrically discrete subgroups of the group of volume preserving diffeomorphisms of a compact smooth manifold. Their proof uses a generalization of the 
Higson-Kasparov construction, but also the techniques of Antonini, Azzali and Skandalis \cite{AAS}.
\end{Rem}

\subsection{A trichotomy for semisimple Lie groups}\label{Trichotomy}

Let us now assume that $G$ is a semisimple Lie group, connected
with finite centre. The conjecture without coefficients (Conjecture \ref{ConjBC}) for $G$ is known to be true. There are now three
completely distinct proofs of that fact. In 1984, A. Wassermann \cite{Wassermann87} (following the work of Penington-Plymen \cite {PeningtonPlymen} and
Valette \cite{Valette83, Valette84K}) proved the conjecture using the whole machinery of Harish-Chandra theory together with the work of Knapp-Stein and Arthur, allowing for a precise description of the reduced dual of such groups. The second proof, due to  V. Lafforgue, only uses Harish-Chandra's Schwartz space, but appeals to the whole of his Banach KK-theory, sketched in Chapter \ref{Banach} below. Another idea of proof had been suggested by Baum, Connes and Higson \cite{BCH}
following the idea of Mackey correspondence, i.e. of a very subtle correspondance between the reduced dual of a semisimple Lie group $G$ and the dual of its Cartan motion group, i.e. the semidirect product $G_0=\mathfrak{g}/\mathfrak{k}\rtimes K$ where $K$ is a maximal compact subgroup of $G$. Very recently A. Afgoustidis \cite{Afgoustidis} has given such a proof using the notion of minimal $K$-types introduced by D. Vogan \cite{Vogan}.

But the most difficult problem arises when one is interested in the conjecture 
for a discrete subgroup  $\Gamma$ of $G$. Such groups inherit the geometry from $G$, but there is of course no hope to describe their reduced dual. However, 
 the conjecture (with or without coefficients) for $\Gamma$ follows from the conjecture {\it with coefficients} (Conjecture \ref{ConjBCcoeff}) for the Lie group $G$, a fact stated without proof in \cite{BCH} and first proved by H. Oyono-Oyono \cite {OO}.

As a result  the question of Baum-Connes for $\Gamma$ can be
summarized as follows, resulting from Kasparov's work:
\begin{itemize}
\item[1)] injectivity of the Baum-Connes assembly map for $G$ holds with coefficients in
any $G-C^*$-algebra, hence it also holds for the discrete group $\Gamma$.

\item[2)] the question of surjectivity of the Baum-Connes assembly map for the discrete group $\Gamma$, or more generally
the surjectivity of the Baum-Connes assembly map with coefficients in any $A$ for the Lie group $G$,  are difficult
questions and can be considered as a crucial test for Conjecture \ref{ConjBC}.
\end{itemize}

We shall have to distinguish, among simple Lie groups, the real rank 1 and the higher rank cases. We need to recall the classification of real rank 1 simple Lie groups. Up to local isomorphism, the list is: $SO_0(n,1), SU(n,1), Sp(n,1), F_{4(-20)}$, i.e. the isometry group of the $n$-dimensional hyperbolic space over the division algebras $\R,\C,\mathbf{H}$ (the Hamilton quaternions), and $\mathbf{O}$ (the Cayley octonions); over $\R$, we restrict to orientation-preserving isometries; over $\mathbf{O}$, there is only $n=2$.

\medskip
Assume that $G$ is locally isomorphic to a simple Lie group. There is the following trichotomy:

\begin{enumerate}
\item[(i)] {\it  If $G$ is (locally isomorphic to) one of the real rank one groups $SO_0(n,1)$ or $SU(n,1)$ ($n\geq 2$):} then $G$ is known to have the Haagerup property. Therefore, by Theorem \ref{HigsonKasparov}, $G$ satisfies the Baum-Connes conjecture with
coefficients (conjecture \ref{ConjBCcoeff}), and so do all its discrete subgroups. 

However, it is worth noting that the  $SO_0(n,1)$ and $SU(n,1)$ cases had been solved {\it before} the Higson-Kasparov theorem by more geometric methods in the works of Kasparov \cite{Kasparov84}, Chen \cite{Chen} and Julg-Kasparov \cite{Julg-Kasparov} .  Indeed, the above authors have produced a construction of a representative of $\gamma$ combining a complex on the flag manifold (which is the boundary of the symmetric space) and a Poisson transform, as explained in section \ref{Flag} below. Then a homotopy using the so-called complementary series yields the required equality $\gamma=1$ in $KK_G({\bf C}, {\bf C})$.

\item[(ii)]{\it  If $G$ is (locally isomorphic to) one of the real rank one groups $Sp(n,1)$ ($n\geq 2$) or $F_{4(-20)}$}: then by a result of Kostant, $G$ has Kazhdan's property (T). This fact makes the Baum-Connes conjecture more difficult since the full and reduced crossed product do not have in general the same K-theory. The first deep result in that direction was obtained by V. Lafforgue in 1998 \cite{Lafforgue00} by combining the Banach analogue of the conjecture, explained in Chapter \ref{Banach}, with Jolissaint's rapid decay property (see section \ref{(RD)} below): if $\Gamma$ is a cocompact discrete subgroup of such a group $G$, then $\Gamma$ satisfies Conjecture \ref{ConjBC} (i.e. without coefficients). 

 Moreover, P. Julg has been able to extend to those cases the method of flag manifolds and Poisson transforms, which gives again the construction of a Fredholm module representing $\gamma$. However, 
 it is no longer possible to use  the theory of unitary representations since the complementary series stays away from the trivial representation, in accordance with property (T). An idea proposed by P. Julg in 1994 is to use a family of uniformly bounded representations approaching the trivial representation. Such a family of uniformly bounded representations has been constructed by M. Cowling \cite{Cowling82}: see section \ref{Cowling} for more details. 

It should be possible to show that the element $\gamma$, though not equal to 1 in $KK_G({\bf C}, {\bf C})$, still gives the identity map in $K_*(C^*_r(G,A))$ (but of course not in $K_*(C^*_{\rm max}(G,A))$. Technically the notion of uniformly bounded representations has to be extended to representations with an arbitrary slow exponential growth, following an idea of N. Higson explained in section \ref{slowgrowth} below. The details of the proof announced by P. Julg \cite{Julg02} have not yet been fully written, we refer to \cite{Julg}.

On the other hand, there is a detailed proof of a similar result by V. Lafforgue \cite{LaffHyp} : any Gromov hyperbolic group $\Gamma$ satisfies Conjecture \ref{ConjBCcoeff} (with coefficients). His proof uses the same idea of arbitrary slow exponential growth representations, see section \ref{hyperLafforgue} below. 

The result of Lafforgue and the announced result of P. Julg have in common the following important case, namely the case of a cocompact lattice $\Gamma$ of $Sp(n,1)$ ($n\geq 2$) or $F_{4(-20)}$. Note however that Lafforgue's result applies to general Gromov hyperbolic groups (many do have property (T)), whereas Julg's claim would apply to all discrete subgroups of $Sp(n,1)$ ($n\geq 2$) or $F_{4(-20)}$, including non-cocompact lattices\footnote{A concrete example of a non-cocompact lattice in $Sp(n,1)$, is $Sp(n,1)(\mathbf{H}(\Z))$, the group of points of the real algebraic group $Sp(n,1)$ over the ring $\mathbf{H}(\Z)$ of integral quaternions. For such a group Conjecture \ref{ConjBCcoeff} is still open.}, which also have property (T).

\item[(iii)] {\it If $G$ is a simple group of real rank greater or equal to 2}: this is the very difficult case. Actually Lafforgue found that for higher rank Lie groups an obstruction persists : they satisfy a stronger version of property (T), explained in section \ref{strong(T)}, that prevents the use of representations of arbitrary small exponential growth to succeed (see \cite{LafforgueT} and \cite{Lafforgue2010}).
In this case very few is known. The only results are for the cocompact discrete subgroups $\Gamma$ of a simple Lie group 
$G$ of rank 2 locally isomorphic to $SL_3({\bf R})$, $SL_3({\bf C})$, $SL_3({\bf H})$ or $E_{6(-26)}$. The proof combines 
 again V. Lafforgue's result on the Banach analogue of the Baum-Connes conjecture (see Chapter \ref{Banach}), and Jolissaint's $(RD)$ property that we recall in \ref{(RD)}.
\end{enumerate}

\subsection{Flag manifolds and KK-theory}\label{Flag}
Let $G$ be a semisimple Lie group, connected with finite centre. Kasparov \cite {Kasparov84} has made the following remark. Let $P=MAN$ be the minimal parabolic (or Borel) subgroup. The flag manifold $G/P$ is a compact $G$-space satisfying the following proposition:

 \begin{Prop}
  An element of $KK_G({\bf C}, {\bf C})$ which is in the image of the map $KK_G(C(G/P),{\bf C})\rightarrow KK_G({\bf C}, {\bf C})$ and restricts
to 1 in $R(K)$ is equal to $\gamma$.
\end{Prop}

This result follows from the fact that the restriction of $\gamma$ to the amenable connected Lie group $P$ is equal to 1. Hence $(1-\gamma )KK_G(C(G/P),{\bf C})=0$.

A stronger statement is used by Julg-Kasparov\cite{Julg-Kasparov} and Julg [Julg]. Let us compactify the symmetric space $X=G/K$ by adding at infinity the flag manifold $G/P$. Consider $\bar X=G/K\cup G/P$. They prove the following:

\begin{Prop}\label{gammaflag}
  An element of $KK_G({\bf C}, {\bf C})$ which is in the image of the map $KK_G(C(\bar X),{\bf C})\rightarrow KK_G({\bf C}, {\bf C})$ and restricts to 1 in $R(K)$ is equal to $\gamma$.
\end{Prop}

\subsubsection{The BGG complex}
An important object associated to flag manifolds is the so called Bernstein-Gelfand-Gelfand (BGG) complex on $G/P$. The following  construction is due to A. \v Cap, J. Slov\'ak and V. Sou\v cek \cite{CSS}.

\begin{Lem} The cotangent bundle $T^*G/P$ carries a $G$-equivariant structure of Lie algebra bundle.
\end{Lem}

\begin{proof}[Proof] The group $G$ acts transitively on the flag manifold $G/P$. Let us consider a point $x\in G/P$. Its stabilizer in $G$ is a parabolic subgroup $P_x$, a conjugate of $P$.  The tangent space at $x$  is the quotient of Lie algebras ${\mathfrak g}/{\mathfrak p}_x$. The Killing form on $G$ identifies the cotangent space $T^*_xG/P=({\mathfrak g}/{\mathfrak p}_x)^*$ with the Lie algebra ${\mathfrak n}_x$ of the maximal nilpotent normal subgroup $N_x$ of $P_x$. 
The Lie algebras ${\mathfrak n}_x$ form a Lie algebra bundle on $G/P$, which is, as a vector bundle, $G$-equivariantly isomorphic to $T^*G/P$.
\end{proof}

Let $\delta_x :{ \bigwedge}^k{\mathfrak n}_x\rightarrow { \bigwedge}^{k-1}{\mathfrak n}_x$ be the boundary operator defining the {\it homology} of the Lie algebra ${\mathfrak n}_x$ for each $x\in G/P$. Recall the formula for $\delta_x$:

$$\delta_x (X_1\wedge ...\wedge X_k)=\sum_{i<j }(-1)^{i+j}[X_i,Xi_j]\wedge X_1\wedge...\wedge \hat X_i\wedge...\wedge \hat X_j\wedge ...\wedge X_k.$$
 Transporting $\delta_x$ from ${\mathfrak n}_x$ to $T_x^*G/P$ defines a bundle map $$\delta :{ \bigwedge}^kT^*G/P\rightarrow { \bigwedge}^{k-1}T^*G/P.$$ 

Let $\Omega =\Omega (G/P)$ be the graded algebra of differential forms on the flag manifold $G/P$. We consider on $\Omega$ the two operators $d$ and $\delta$, respectively of degree $1$ and $-1$.
Since $d^2$ and $\delta^2$ are both zero, the degree zero map $d\delta+\delta d$ commutes both with $d$ and $\delta$. In fact, as proved by \v Cap and Sou\v cek \cite{CS}:

\begin{Prop} The Casimir operator of $G$ acting on $\Omega$ is equal to $-2(d\delta+\delta d)$.
\end {Prop}

Let $\Omega_0$ be  kernel of the Casimir operator in $\Omega$, which is  a subcomplex of the de Rham complex.
 \begin{Thm}
\begin{enumerate}
\item[1)] $\Omega_0 ={\rm ker}(d\delta+\delta d)={\rm ker}\delta\cap{\rm ker}{\delta}d$. 
\item[2)] The  canonical injection $\Omega_0 \rightarrow\Omega$ induces an isomorphism in cohomology. 
\item[3)] The canonical map ${\rm ker}\delta\cap{\rm ker} \delta d\rightarrow {\rm ker}\delta/{\rm im}\delta$ is a $G$-equivariant isomorphism from $\Omega_0$ to the space of sections of the bundle ${\rm ker}\delta/{\rm im}\delta$, whose  fibers are the {\it homology groups} $H_k( {\mathfrak n}_x )$
  of the Lie algebra $T_x^*G/P={\mathfrak n}_x$.
 \item[4)] The complex $D$ transported from the complex $d$ on $\Omega_0$ is a complex of differential operators on the space of sections of the above bundle.
\end{enumerate}
\end{Thm}

\begin{Rem} The adjoint action of $N_x$ on ${\mathfrak n}_x $ induces the identity on $H_k( {\mathfrak n}_x )$, a classical fact about Lie algebra homology. Therefore the BGG complex is defined on a space of sections of a bundle on $G/P$  
obtained from a representation of $P$ {\it which is trivial on the nilpotent normal subgroup $N$}, i.e. factors through $P/N=MA$. In the language of representation theory, it means that the representation involved in the BGG complex are finite sums of (non-unitary) principal series of $G$. 
\end{Rem}

\subsubsection {The model: $SO_0(n,1)$}

Let us now explain Kasparov's proof \cite {Kasparov88}  of the Connes-Kasparov conjecture with coefficients for the Lorentz groups $G=SO_0(2n+1,1)$. The flag manifold $G/P$ is the sphere $S^{2n}$, which is the boundary of the hyperbolic space of dimension $2n+1$. Because the nilpotent group $N$ is abelian, the operator $\delta$ is zero and the BGG complex is nothing but the de Rham complex. Kasparov 
constructs a Fredholm module representing the element $\gamma$ using the crucial fact that $G/P=S^{2n}$ carries 
 a $G$-invariant {\it conformal structure}.  Indeed, let us equip the sphere with its $K$-invariant metric. The action of $g\in G$ transforms the metric into its multiple by some scalar function 
$\lambda_g^2$. 
\begin {enumerate}
\item[1)]  We make the action of $G$ unitary by twisting the representation by a cocycle thanks to the conformal structure. More precisely, let $$\pi (g)\alpha=\lambda_g^{n-k}g^{-1*}\alpha.$$ The representation $\pi$ is unitary on the Hilbert space of $L^2$ forms of degree $k$.

\item[2)]We make the operator $d$ bounded by considering $F=d(1+\Delta )^{-1/2}$, where $\Delta=dd^*+d^*d$ is the Laplace-Beltrami operator. The bounded complex $F$ is no more $G$-invariant, but the natural action of $g\in G$ takes the zero order pseudodifferential operator $F$ to $\lambda_gF$ plus a negative order pseudodifferential operator, as easily seen at the principal symbol level. 

\item[3)]Combining the two preceeding items (and the fact that $F$ maps $k$-forms to $(k+1)$-forms) we easily see that the conjugate $\pi (g)F\pi(g)^{-1}$ equals $F$ plus a negative order pseudodifferential operator, hence the compactness of the 
commutator $[F,\pi (g)]$.
\end{enumerate}

The Fredholm module thus obtained is not quite the good one, since its index is 2 (the Euler characteristic of $S^{2n}$). To solve the problem,
 Kasparov cuts the complex in the middle: the group acts on the sphere $S^{2n}$ by conformal transformations and the Hodge $*$ operation on forms of degree $n$ is therefore $G$-invariant. The half-complex consists in taking the forms of degree $0$ to $n-1$, and in degree $n$ only half of them, those for which $*=i^n$. Then the index is 1. In the smallest dimension case $n=1$ ($G=PSL(2,{\bf C})$) it amounts to take the $\bar\partial$ operator instead of the $d$ operator on $S^2=P^1({\bf C})$.
The $G$-Fredholm module thus obtained represents the element $\gamma$ by proposition 5.10.
 
In \cite{Kasparov84}, the case of $SO_0(2n,1)$ was settled as a mere corollary of the case of $SO_0(2n+1,1)$. Indeed  $SO_0(2n, 1)$ is a subgroup of $SO_0(2n+1,1)$ and the element  $\gamma$ restricts to closed subgroups. However it was most interesting to treat the case of $SO_0(2n, 1)$ in itself before passing to the other rank one groups. The direct proof for $SO_0(2n, 1)$ has been treated by
Z. Chen in his thesis \cite{Chen}. The $G$-equivariant de Rham complex on $S^{2n-1}$ is again turned, thanks to the conformal structure,  into a $G$-Fredholm module, but this time the index is 0 (the Euler characteristic of $S^{2n-1}$). To get an Fredholm module of index 1, something new is needed, which has no analogue in the $SO_0(2n+1,1)$ case. One must use the $L^2$-cohomology of the hyperbolic space of dimension $2n$, i.e. the Hilbert space ${\cal H}^n$ of square integrable harmonic forms (which are of degree $n$), which is a sum of two discrete series of $G$. The truncated module (with index 1) is obtained by considering forms of degree $\leq n-1$, and completing by a map from $\Omega^{n-1}(S^{2n-1})$ to ${\cal H}^n$. For $n=1$, the map $\Omega^0(S^1)\rightarrow {\cal H}^1$ is just the composition of the classical Poisson transform with the de Rham differential. In general one must use P.-Y. Gaillard's Poisson transform for forms \cite{Gaillard}. One thus obtains an element of $KK_G({\bf C}, {\bf C})$ which is equal to $\gamma$ by proposition \ref{gammaflag}.

\subsubsection{Generalization to other rank one groups}

 The de Rham complex is replaced by the BGG complex on the flag manifold. This is done by P. Julg and G. Kasparov in \cite{Julg-Kasparov} for $G=SU(n,1)$ where the BGG complex is the so called Rumin complex associated to the $G$-invariant contact structure on $G/P=S^{2n-1}$, and for $Sp(n,1)$ or $F_{4(-20)}$ by Julg \cite{Julg}. In order to turn the BGG-complex into a $G$-Fredholm module, one has to replace, in the above $SO_0(n,1)$-picture, conformal structure by quasi-conformal structure: the tangent bundle has a $G$-equivariant subbundle $E$ of codimension 1, 3 or 7 respectively for $G=SU(n,1), Sp(n,1)$ or $F_{4(-20)}$, whose fiber $E_x$  at any point $x\in G/P$  is defined as the subspace of 
$T_xG/P={\mathfrak n}_x^*$  orthogonal to the subalgebra $[{\mathfrak n}_x,{\mathfrak n}_x]$  of the Lie algebra ${\mathfrak n}_x$. The $K$-invariant metric is no more conformal, but quasi-conformal in the following sense: consider the action of $G$ on the subbundle $E$ and on the quotient $TM/E$ (note that there is no invariant supplementary subbundle), then under $g\in G$ the metric on $E$ is multiplied by some scalar function $\lambda_g^2$, and  on the quotient $TM/E$ by $\lambda_g^4$. The action of $G$ on forms is not conformal, but  after passing to the $\delta$-homology $H_*( {\mathfrak n}_x )$ splits into conformal components. Such a splitting is defined by the weight, i.e. the action of the abelian group ${\bf R}_+^*$ seen as a subgroup of the quotient $P_x/N_x$. This is closely related to the splitting of the representation in the BGG complex into (non-unitary) principal series of $G$ mentioned in remark 5.15.
It follows that one can modify the action of $G$ into a unitary representation $\pi (g)$ on the space of $L^2$ sections of the BGG complex.

Exploring the analytical properties of the complex $D$ requires to replace the ordinary $K$-invariant Laplacian by the $K$-equivariant {\it sub-Laplacian} on $G/P$. Namely, $\Delta=-\sum X_i^2$ where the vector fields $X_j$ form at each point $x\in G/P$ an orthonormal basis (for a $K$-equivariant metric) of the subspace of $T_xG/P={\mathfrak n}_x^*$ orthogonal to the subalgebra $[{\mathfrak n}_x,{\mathfrak n}_x]$ of ${\mathfrak n}_x$. The operator $\Delta$ is not elliptic but hypoelliptic. It has a parametrix which is not a classical pseudodifferential operator, but belongs to a new pseudodifferential calculus in which Fourier analysis is replaced by  representation theory of nilpotent Lie groups. 
Such calculi have been constructed in special cases by Beals and Greiner \cite{BG} or by Christ, Geller, G\l owacki and Polin \cite{CGGP}. However what is needed here is the general construction, which seems to appear only in Melin's 1982 preprint \cite{Melin}, unfortunately very difficult to find. It is worth to mention that noncommutative geometry has motivated a revival of work on the subject, in particular the groupoid approach. The groupoid adapted to the situation has been constructed  by various authors: R. Ponge \cite{Ponge}, 
E. van Erp and R. Yuncken \cite{vE-Y} , see also \cite{vE-J}. The most beautiful construction of the groupoid using the functoriality of the deformation to the normal cone can be found in the recent thesis of O. Mohsen \cite{Moh}. The link between the groupoid and the pseudodifferential calculus is discussed in \cite{DS} and \cite{vE-Y19}.

The following theorem explains how to combine the sublaplacian and the weight grading to produce an element of $KK_G(C(G/P),{\bf C})$ out of the BGG complex. See \cite{Ru}, \cite{Julg} and \cite{DH}.

\begin {Thm} Let $\Delta^{W/2}$ be the pseudo-differential operator equal to the power $\Delta^{w/2}$ on the $w$ weight component of the BGG complex. Then the conjugate $F=\Delta^{W/2}D\Delta^{-W/2}$ is a bounded operator satisfying $F^2=0$ on the Hilbert space of $L^2$ sections of the BGG complex. The commutators $[F,f]$ and $[F,\pi (g)]$ are compact operators for any $f\in C(G/P)$ and $g\in G$. Moreover $F$ admits a parametrix, i.e. a bounded operator $Q$ such that $FQ+QF-1$, $Q^2$, as well as $[Q,f]$ and $[Q,\pi (g)]$ for $f\in C(G/P)$ and $g\in G$ are compact operators.
\end {Thm}

As above in the $SO_0(2n,1)$ case, one has to modify the complex in order to get a truncated complex of index 1 in $R(K)$. Then proposition 5.11 will ensure that its class in $KK_G({\bf C},{\bf C})$ is the element $\gamma$. Here again discrete series must be involved, namely 
 those entering the $L^2$-cohomology of the symmetric space $G/K$, i.e. the Hilbert space $\mathcal{H}^m$ of harmonic $L^2$ forms of degree $m={\dim G/K\over 2}$, namely $m=n$, $n$, $2n$ and $8$ respectively in the cases of  $SO_0(2n,1)$, $SU(n,1)$, $Sp(n,1)$ and $F_{4(-20)}$. Note that the Casimir operator vanishes on $\mathcal{H}^m$ (since the Casimir operator is equal to $-\Delta$ (the Laplace Beltrami operator) on $L^2(\Omega (G/K)$), allowing to build an adequate Poisson transform \cite{CHJ} sending the BGG component in degree $m$ to harmonic forms of degree $m$ in $G/K$. 
 The half complex is then obtained by taking the BGG complex up to degree $m-1$ and to complete by the composition of $D$ with  such a Poisson transform \cite{Julg}.
 
\subsubsection {Generalization to higher rank groups}
More difficult is the case where $G$ is a simple Lie group of real rank $\geq 2$. So far only the case of $SL(3,{\bf C})$ has been treated, by R. Yuncken \cite{Yuncken} who has been able to produce a $G$-Fredholm module representing $\gamma$  out of the BGG complex. Here the flag manifold $G/P$, where $P$ is the minimal parabolic of $G$ comes with two $G$-equivariant fibrations $G/P\rightarrow G/P_i$ ( $i=1,2$) onto smaller flag manifolds coresponding to $P_1$ and $P_2$ the two other parabolics containing $P$. The operators in the BGG complex turn out to be longitudinally elliptic differential operators along the fibers. Considering a class of pseudodifferential operators on multifiltered manifolds, and making an unexpected use of Kasparov's technical lemma yields a Fredholm module representing an element of $KK_G(C(G/P),{\bf C})$. Its index can be taken as 1 in $R(K)$ if one considers the holomorphic BGG complex (as in the $SL(2,{\bf C})$ case of \cite{Kasparov84}, where $d$ is replaced by $\bar\partial$). Its class in $KK_G({\bf C},{\bf C})$ is therefore $\gamma$ by proposition \ref{gammaflag}.

\section{Banach algebraic methods}\label{Banach}
As Julg pointed out in \cite{Julg97}, once non-unitary representations appear, one can no longer work with $C^*$-algebras but with more general topological algebras, for instance, Banach algebras. Unfortunately, Kasparov's KK-theory is a purely $C^*$-algebraic tool. However K-theory can be defined for all kind of topological algebras (see the appendix of \cite{Bost90} for the notion of good topological algebras for which the K-theory can be defined); consequently, one has to be able to work in a more flexible framework whose foundations were laid by Lafforgue.

\subsection{Lafforgue's approach}\label{Laff}

\subsubsection{Banach KK-theory}\label{KKban}

In \cite{LaffInv}, Lafforgue defined a bi-equivariant KK-theory, denoted by $KK^{\mathrm{ban}}$, for general Banach algebras. The basic idea to start with, was to define a group $KK^{\mathrm{ban}}_G(\C,\C)$, in the same way as Kasparov defined $KK_G(\C,\C)$, but where unitary representations on Hilbert spaces are replaced by isometric representations on Banach spaces and, therefore, replacing $C^*$-algebras by Banach algebras. More generally, Lafforgue defined a group $KK^{\mathrm{ban}}_{G,\ell}(\C,\C)$ using representations on Banach spaces that are not necessarily isometric but for which the growth is exponentially controlled by a length function on the group $G$. 

Knowing that the trivial representation is not isolated among representations on Banach spaces \footnote{See the discussion of strong property (T) in \ref{strong(T)}.}, Lafforgue was able to prove for a class of groups called $\mathcal{C'}$ in \cite{LaffInv}, which is contained in the class $\mathcal{C}$ and hence for which a $\gamma$-element has been constructed, that such a $\gamma$ is equal to $1$ in $KK^{\mathrm{ban}}_{G,\ell}(\C,\C)$. The class $\mathcal{C}'$ contains:
\begin{itemize}
\item semisimple real Lie groups and their closed subgroups;
\item simple algebraic groups over non-Archimedean local fields, and their closed subgroups;
\item hyperbolic groups.
\end{itemize}

The equality $\gamma=1$ in $KK^{\mathrm{ban}}_{G,\ell}(\C,\C)$ allowed Lafforgue to prove, for all groups in $\mathcal{C'}$, an analogue of the Baum-Connes conjecture where $C^*(G)$ is replaced by $L^1(G)$, which for general $G$ is a conjecture of Bost. More precisely, Lafforgue used his equivariant $KK^{\mathrm{ban}}$ to define a morphism 
$$\mu^A_{L^1}:K^{\mathrm{top}}_*(G,A)\to K_*(L^1(G,A)),$$ 
for all locally compact groups $G$ and all $G$-$C^*$-algebra $A$. \\
More precisely, Lafforgue used his equivariant $KK^{\mathrm{ban}}$ to define a morphism 
$$\mu^A_{L^1}:K^{\mathrm{top}}_*(G,A)\to K_*(L^1(G,A)),$$ 
for all locally compact groups $G$ and all $G$-$C^*$-algebras $A$. \\

Let us recall the important features of Lafforgue's Banach KK-theory that allow one to define the morphism $\mu^A_{L^1}$. If $A$ and $B$ are two Banach algebras endowed with an action of a locally compact group $G$ then there exists a descent map 
$$j_G^{L^1}:KK^{\mathrm{ban}}_G(A,B)\to KK^{\mathrm{ban}}(L^1(G,A),L^1(G,B)).$$
Unlike Kasparov's bivariant theory, Banach KK-theory does not have a product but nevertheless, it still acts on K-theory, i.e there is a morphism 
$$K_0(A)\times KK^{\mathrm{ban}}(A,B)\to KK^{\mathrm{ban}}(\C,B)$$
and for every Banach algebra $B$, the group $KK^{\mathrm{ban}}(\C,B)$ is isomorphic to $K_0(B)$. Consequently, following the Baum-Connes-Higson formulation of the conjecture and hence the construction of the assembly map (see \ref{official}), one gets, without too much effort, the morphism $\mu^A_{L^1}:K^{\mathrm{top}}_*(G,A)\to K_*(L^1(G,A)).$\\

Let us stress in addition that, unlike Hilbert spaces, Banach spaces are in general not self-dual; so to define the groups $KK^{\mathrm{ban}}(A,B)$ Lafforgue has to replace Hilbert modules by pairs of Banach modules together with a duality condition. For details, see Chapter 10 in Valette's book \cite{Valette02}.

\subsubsection{Bost conjecture and unconditional completions}\label{Bostcon}
Bost's conjecture (with coefficients) is stated as follows 
\begin{Conj}[Bost]
For all locally compact groups $G$ and for all $G$-$C^*$-algebras $A$ the map $\mu^A_{L^1}$ is an isomorphism.
\end{Conj} 

The moral is that when using representations with controlled growth on Banach spaces to construct an homotopy between a $\gamma$-element and $1$, as one gets out of the $C^*$-algebraic context, the K-theory that we are able to compute is the K-theory of a Banach algebra. In the case of the Bost conjecture, the Banach algebra that we get is $L^1(G)$. \\

There are two good things about the Bost conjecture, the first one is that it is easier to prove than the Baum-Connes conjecture (meaning that it has been proven by Lafforgue for a wide class of groups containing all semisimple Lie groups as well as their lattices) and no counter-example to the Bost conjecture is known, to the best of our knowledge. Secondly, the original map $\mu_{A,r}$ of Baum-Connes-Higson  factors through the map $\mu^B_{L^1}$ so that the following diagram is commutative :

$$\xymatrix{
K^{\mathrm{top}}_*(G,B)\ar[dr]_{\mu^B_{L^1}}\ar[r]^{\mu^B_r}&K_*(C^*_r(G,B))\\
&K_*(L^1(G,B))\ar[u]^{\iota_*},}
$$
where $\iota$ denotes inclusion $L^1(G,B)\to C^*_r(G,B)$. Therefore, if we take $G$ to be a group for which the Bost conjecture has been proven, for example a semisimple Lie group or a lattice in such a group, trying to prove the Baum-Connes conjecture for $G$ amounts to prove that $\iota_*$ is an isomorphism, in other words that $\iota$ induces an isomorphism in K-theory. \\

Unfortunately, the usual criteria to prove that the continuous inclusion of $L^1(G)$ in $C^*_r(G)$ induces an isomorphism in K-theory, is not true for most of the locally compact groups. For example, the algebra $L^1(G)$ is not stable under holomorphic calculus if $G$ is a non-compact semisimple Lie group \cite{Leptin-Poguntke79}. To illustrate this, let us recall the usual criterion to determine whether an injective morphism of Banach algebras induces an isomorphism at the level of their K-theory groups.

\begin{Def} 
Let $A$ and $B$ be two \emph{unital} Banach algebras and $\phi: A\to B$ is a morphism of Banach algebras between them. Then $\phi$ is called spectral, if for every $x\in A$ the spectrum of $x$ in $A$ equals the spectrum of $\phi(x)$ in $B$. It is called dense if $\phi(A)$ is dense in $B$.
\end{Def}
This terminology is taken from \cite{Nica2008}. When $\phi$ is injective, $A$ can be considered as a subalgebra of $B$. In this case, $A$ is said to be "{\it stable under holomorphic calculus in $B$}", because, for every $x\in A$ and every holomorphic function $f$ on a neighborhood of the spectrum of $x$ in $B$, the element $f(x)$ constructed using holomorphic functional calculus in $B$ belongs to $A$ (see \cite{Bost90}).\\

The theorem below is a classical result known as the {\it Density Theorem}; it is due to Swan and Karoubi (see \cite[Section 2.2 and 3.1]{Swan}, \cite[p. 109]{Karoubi}, \cite[Appendix 3]{ConThom} and \cite[Th\'eor\`eme A.2.1]{Bost90}).

\begin{Thm}
If $A$ and $B$ are two unital Banach algebras and $\phi: A\to B$ is dense and spectral morphism of Banach algebras then $\phi$ induces an isomorphism $\phi_*:K_*(A)\to K_*(B)$. 
\end{Thm}

What Bost noticed is that the condition of been spectral is, somehow, too strong: if $\phi$ is spectral it induces \emph{strong isomorphisms} in K-theory :

\begin{Def}{\cite{Bost90}}\label{Bost90Strong}
An injective morphism between two unital Banach algebras $\phi: A\to B$, induces a strong isomorphism in K-theory if for every $n\geq 1$ the maps
$$M_n(\phi):P_n(A)\to P_n(B)\quad\text{and}\quad GL_n(\phi): GL_n(A)\to GL_n(B),$$
induced by $\phi$ are homotopy equivalences. 

\end{Def} 
Here for an algebra $A$ and for an integer $n$, we denote by $M_n(A)$ the set of $n\times n$ matrices with coefficients in $A$ and $P_n(A)=\{p\in M_n(A)\mid p^2=p\}$ is the set of idempotent matrices.\\
If the maps $M_n(\phi)$ and $GL_n(\phi)$ above are homotopy equivalences then the morphism induced by $\phi$, say $\phi_*:\mathcal{P}(A)\to\mathcal{P}(B)$ is an isomorphism (where $\mathcal{P}(A)$ denotes the semi-group of isomorphism classes of projective $A$-modules of finite type). 
This is stronger than inducing an isomorphism in K-theory. \\

The next example shows that, because of this strength, it is not easy to pass from Bost conjecture to the Baum-Connes conjecture.

\begin{Ex} Set $G=SL_2(\R)$. Then $G$ has two representations in the discrete series (i.e. square integrable representations) that are not integrable (i.e the matrix coefficients do not  belong to $L^1(G)$) and are therefore not isolated in the dual space of $L^1(SL_2(\R))$. This implies that the idempotent of $C^*_r(SL_2(\R))$ associated to one of these discrete series (which is equal to a matrix coefficient) does not belong to $L^1(SL_2(\R))$; hence, if $\pi_0$ denotes the set of connected components, the map from $\pi_0\Big(P_1(L^1(SL_2(\R))\Big)\to \pi_0\Big(P_1(C^*_r(SL_2(\R))\Big)$, which is induced by $\iota$, is not surjective. Therefore applying what is known for the Bost conjecture to the Baum-Connes conjecture is not in any case automatic. 
\end{Ex}

Fortunately, Lafforgue's proof of the Bost conjecture can actually be used to compute the K-theory of a class of Banach algebras more general than $L^1(G)$ called \emph{unconditional completions} of $C_c(G)$.

\begin{Def}
Let $G$ be a locally compact group. A Banach algebra completion $\mathcal{B}(G)$ of $C_c(G)$ is called \emph{unconditional} if the norm $\|f\|_{\small\mathcal{B}(G)}$ only  depends on the map $g\mapsto |f(g)|$, i.e. for $f_1,f_2\in C_c(G)$, $\|f_1\|_{\small\mathcal{B}(G)}\leq \|f_2\|_{\small\mathcal{B}(G)}$ if $|f_1(g)|\leq |f_2(g)|$ for all $g\in G$. 
\end{Def}

\begin{Ex}
For a locally compact group $G$, the algebra $L^1(G)$ is an unconditional completion of $C_c(G)$. 
\end{Ex}

\begin{Ex}\label{schwartz}
If $G$ is a connected semisimple Lie group and $K$ is a maximal compact subgroup, let $t\in \R^+$ and let $\mathcal{S}_t(G)$ be the completion of $C_c(G)$ for the norm given by :
$$\|f\|_{\small{\mathcal{S}_t(G)}}=\sup\limits_{g\in G} |f(g)|\phi(g)^{-1}(1+d(g))^t,$$
where $\phi$ is the Harish-Chandra function on $G$ (see Chapter 4 in \cite{Knapp01}) and for $g\in G$, $d(g)$ is the distance of $gK$ to the origin in $G/K$. Then, for $t$ large enough, $\mathcal{S}_t(G)$ is an unconditional completion (see Section 4 in \cite{LaffInv}).
\end{Ex}

Another important example of unconditional completions appears in connexion with the rapid decay property, to be discussed in subsection \ref{(RD)} below.

Inspired by the definitions of the algebras $L^1(G,A)$, one can define analogues of crossed products in the context of Banach algebras using unconditional completions as follows: if $A$ is a $G$-$C^*$-algebra and $\mathcal{B}(G)$ is an unconditional completion of $C_c(G)$, we define the algebra $\mathcal{B}(G,A)$ as the completion of $C_c(G,A)$ for the norm $$\|f\|_{\small\mathcal{B}(G,A)}=\big\|g\mapsto\|f(g)\|_A\big\|_{\small\mathcal{A}(G)}.$$
For all locally compact group $G$, all $G$-$C^*$-algebra $A$ and all unconditional completions $\mathcal{B}(G)$ Lafforgue used his Banach KK-theory to construct a morphism  $$\mu^A_{\mathcal{B}(G)}:K_{*}^{\mathrm{top}}(G,A)\to K_{*}(\mathcal{B}(G,A)).$$

He then obtained an analogue of the "Dirac-dual Dirac method" in this context :

\begin{Thm}[Lafforgue]\label{Unconditionnal}
If the group $G$ has a $\gamma$-element in $KK_G(\C,\C)$ and if there exists a length function $\ell$ on $G$, such that, for all $s>0$, $\gamma=1$ in $KK_{G,s\ell}^{\mathrm{ban}}(\C,\C)$, then $\mu^A_{\mathcal{B}(G)}$ is an isomorphism for all unconditional completions $\mathcal{B}(G)$ and for all $G$-algebras $A$.
\end{Thm}

Lafforgue proved the equality $\gamma=1$ in $KK_{G,s\ell}^{\mathrm{ban}}(\C,\C)$ for all groups in the class $\mathcal{C'}$ (see \cite[Introduction]{LaffInv}). All real semisimple Lie groups and all $p$-adic reductive Lie groups as well as their closed subgroups, all discrete groups acting properly, cocompactly, continuously and by isometries on a $CAT(0)$ space and all hyperbolic groups belong to this class. For all these groups $G$ and all $G$-algebras $A$ the map $\mu^A_{\mathcal{B}(G)}$ is an isomorphism and hence the Bost conjecture holds (see \cite{LaffInv}). For a nice expository explanation on how the homotopy between $\gamma$ and $1$ is constructed using Banach representations, see \cite{Skandalis02} where the combinatorial case is explained in details, i.e the case containing $p$-adic groups.\\

\subsubsection{Application to the Baum-Connes conjecture}

Let $\mathcal{B}(G)$ be an unconditional completion of $C_c(G)$ that embeds in $C^*_r(G)$. In that case, the Baum-Connes map $\mu_r$ factors through $\mu_{\mathcal{B}(G)}$ so that the following diagram is commutative 

$$\xymatrix{ K^{\mathrm{top}}_*(G)\ar[rr]^{\mu_{\mathcal{B}(G)}}\ar[drr]_{\scriptstyle{\mu_r}} & &K_*(\mathcal{B}(G))\ar[d]^{i_*} \\
&&K_*(C^*_r(G)) },$$
where $i_*$ is the inclusion map induced by the map $i : \mathcal{B}(G)\to C^*_r(G)$. \\

\begin{Prop}
Let $G$ be a group in Lafforgue's class $\mathcal{C'}$. Suppose there exists an unconditional completion $\mathcal{B}(G)$ which is a dense subalgebra stable under holomorphic calculus in $C^*_r(G)$. Then the Baum-Connes assembly map $\mu_r$ is an isomorphism.
\end{Prop}

Using Example \ref{schwartz}, we can state the first result of Lafforgue concerning connected Lie groups (see the discussion in subsection \ref{Trichotomy} regarding those groups)

\begin{Thm}\label{LafforgueSS}[Lafforgue]
Let $G$ be a connected semisimple Lie group. Then Conjecture \ref{ConjBC} (without coefficients) is true for $G$.
\end{Thm}

\begin{proof}
For $t\in\R^+$ large enough, the algebra $\mathcal{S}^t(G)$ from example \ref{schwartz} is an unconditional completion which is dense and stable under holomorphic calculus in $C^*_r(G)$ (cf. Section 4 in \cite{LaffInv}).

\end{proof} 

As a matter of fact, Lafforgue's theorem is much more general. Let $G$ be a locally compact group. A quadruplet $(G,K,d,\phi)$ is a {\it Harish-Chandra quadruplet} if $G$ is unimodular with Haar measure denoted by $dg$, $K$ is a compact subgroup endowed with his unique Haar measure of mass equal to $1$, $d$ is a length function on $G$ such that $d(k)=0$ for all $k\in K$ and $d(g^{-1})=d(g)$ for all $g\in G$ and $\phi:G\to ]0,1]$ is a continuous function satisfying the following 5 properties : 
\begin{enumerate}
\item $\phi(1)=1$, 
\item  $\forall g\in G$, $\phi(g^{-1})=\phi(g)$,
\item $\forall g\in G$, $\forall k, k'\in K$, $\phi(kgk')=\phi(g)$,
\item $\forall g,g'\in G$, $\int_K\phi(gkg') dk =\phi(g)\phi(g'),$
\item for all $ t\in \R_+$ large enough, $\int_G\phi^2(g)(1+d(g))^{-t}dg$ converges. 

\end{enumerate}

When one has a Harish-Chandra quadruplet, then one can define a Schwartz space on $G$ following Example \ref{schwartz} : $\mathcal{S}_t(G)$ is the Banach space completion of $C_c(G)$ with respect to the norm given by $$\|f\|_{\small{\mathcal{S}_t(G)}}=\sup\limits_{g\in G} |f(g)|\phi(g)^{-1}(1+d(g))^t.$$
Lafforgue's result is then stated a follows:

\begin{Prop}
Let $(G,K,d,\phi)$ be a Harish-Chandra quadruplet. Then, for $t\in\R_+$ large enough, $\mathcal{S}_t(G)$ is a unconditional completion of $C_c(G)$ which is a subalgebra of $C^*_r(G)$ dense and stable under holomorphic calculus.

\end{Prop}
In Section 4 of \cite{LaffInv}, Lafforgue constructed a Harish-Chandra quadruplet for all linear reductive Lie groups on local fields.

\begin{Rem} 
The method of finding a Schwartz type unconditional completion dense and stable under holomorphic calculus in $C^*_r(G)$ like the algebra $\mathcal{S}_t(G)$ for semisimple Lie groups, does not work with coefficients (see the remark after the Proposition 4.8.2 of \cite{LaffInv}). If $\Gamma$ is a lattice in a semisimple Lie group $G$, we can define an algebra $\mathcal{S}_t(\Gamma)$ in the same manner as for $G$: it is the completion of $C_c(\Gamma)$ for the norm $\|f\|_{\mathcal{S}^t(\Gamma)}=\sup\limits_{\gamma\in\Gamma}|f(\gamma)|(1+d(\gamma))^t\phi(\gamma)^{-1}$, where $\phi$ is the Harish-Chandra function of $G$ and the $d$ is the appropriate distance in $G$ (see \cite{Boyer} where this algebras are studied). Suppose now that $\Gamma$ is cocompact. Then $\mathcal{S}_t(G,C(G/\Gamma))$ is not stable under holomorphic calculus in $C^*_r(G,C(G/\Gamma))$ as these algebras are Morita equivalent to $\mathcal{S}_t(\Gamma)$ and $C^*_r(\Gamma)$, respectively, and $\mathcal{S}_t(\Gamma)$ is not stable under holomorphic calculus in $C^*_r(\Gamma)$. Indeed, if $\gamma\in\Gamma$ is an hyperbolic element, since $d(\gamma^n)$ grows linearly in $n$ if we denote by $e_\gamma$ the corresponding Dirac function in $\C\Gamma$, its spectral radius as an element of $C^*_r(\Gamma)$ is $1$ whereas its spectral radius in $\mathcal{S}_t(\Gamma)$ is $>1$. To see this we use the following classical estimate on the Harish-Chandra $\phi$-function (see Proposition 7.15 in \cite{Vogan}): there are positive constants $C,\ell>0$ such that for every $g\in G$
: 
$$\phi(g)\leq Ce^{-d(g)}(1+d(g))^\ell.$$
Hence 
$$\|e_\gamma^n\|_{\mathcal{S}_t(\Gamma)}=\frac{(1+d(\gamma^n))^t}{\phi(\gamma^n)}\geq C^{-1}(1+d(\gamma^n))^{t-\ell}e^{d(\gamma^n)}.$$
Since $d(\gamma^n)$ grows linearly in $n$, we have for the spectral radius of $e_\gamma$ in $\mathcal{S}_t(\Gamma)$:
$$\lim_{n\rightarrow\infty}\|e_\gamma^n\|_{\mathcal{S}_t(\Gamma)}^{1/n} > 1.$$

\end{Rem}

\subsubsection{The rapid decay property}\label{(RD)}

To state Lafforgue's results concerning lattices in connected Lie groups, and hence examples of discrete groups having property (T) and verifying the Baum-Connes conjecture (without coefficients), we need to introduce the property of rapid decay, denoted by (RD).

Recall that, for $G$ a locally compact group, a continuous function $\ell:G\rightarrow\R^+$ is a {\it length function} if $\ell(1)=0$ and $\ell(gh)\leq\ell(g)+\ell(h)$ for every $g,h\in G$. 

\begin{Ex} If $\Gamma$ is a finitely generated group and $S$ is a finite generating subset, then $\ell(g)=|g|_S$ (word length with respect to $S$) defines a length function on $\Gamma$.
\end{Ex}

The following definition is due to P. Jolissaint \cite{Jolissaint}.

\begin{Def} Il $\ell$ is a length function on the locally compact group $G$, we say that $G$ has the property of rapid decay with respect to $\ell$ (abridged property (RD)) if there exists positive constants $C,k$ such that, for every $f\in C_c(G)$:
$$\|\lambda(f)\|\leq C\cdot\|f(1+\ell)^k\|_2.$$
\end{Def}

In other words the norm of $f$ in $C^*_r(G)$, i.e. the operator norm of $f$ as a convolutor on $L^2(G)$, is bounded above by a weighted $L^2$-norm given by a polynomial in the length function.\\

The relevance of property (RD) regarding Baum-Connes comes from the following fact. If $\Gamma$ is a discrete group with property (RD)  with respect to a length function $\ell$, then, for a real number $s$ which is large enough, the space 
$$H_{\ell}^s(\Gamma)=\{f:\Gamma\to\C\,|\,\|f\|_{\ell,s}=\big(\sum\limits_{\gamma\in\Gamma}|f(\gamma)|^2(1+\ell(\gamma))^{2s}\big)^{\frac{1}{2}}<\infty\},$$
is a convolution algebra and an unconditional completion of $C_c(\Gamma)$ that is stable under holomorphic calculus in $C^*_r(\Gamma)$ (see for example \cite{Valette02}, 8.15, Example 10.5). Note that functions in $H_\ell^s(\Gamma)$, with $s\gg 0$, are decaying fast at infinity on $\Gamma$, hence the name {\it rapid decay}.

We can now state Lafforgue's result concerning discrete groups (Corollaire 0.0.4 in \cite{LaffInv}):

\begin{Thm}\label{RDimpliesBC}

Let $\Gamma$ be a group with property (RD) in Lafforgue's class $\mathcal{C}'$ (see subsection \ref{KKban}). Then Conjecture \ref{ConjBC} (without coefficients) for $\Gamma$ is true.
\end{Thm}

Jolissaint \cite{Jolissaint} has shown that property (RD) holds for cocompact lattices in real rank 1 groups, a fact generalized in two directions:
\begin{itemize} 
\item by P. de la Harpe \cite{dlH} to all Gromov hyperbolic groups;
\item by I. Chatterji and K. Ruane \cite{ChatterjiRuane} to {\it all} lattices in real rank 1 groups.
\end{itemize}
By Theorem \ref{RDimpliesBC}, those groups do satisfy the Baum-Connes conjecture (without coefficients).

\begin {Rem}The first spectacular application of property (RD) was the proof 
of the Novikov conjecture for Gromov-hyperbolic groups by A. Connes and H. Moscovici \cite{ConnesMosco90}. Their result is the following:

\begin {Thm}\label{Novikovhyper} Assume that the group $\Gamma$ satisfies both Jolissaint's (RD) property and the bounded 
cohomology property (i.e. that any group $n$-cocycle is cohomologous to a bounded one, for $n\geq 2$).
Then $\Gamma$ satisfies the Novikov conjecture.
\end{Thm}

\begin{proof}[Sketch of proof]
Let $x\in H^n(\Gamma,\Q)$ be a cohomology class. Let $M$ be a closed, Spin manifold and $f:M\rightarrow B\Gamma$ a map; let $\tilde{M}$ be the pull-back of $\widetilde{B\Gamma}$ via $f$. Let $D$ be a $\Gamma$-invariant Dirac operator on $\tilde{M}$. Connes and Moscovici show that the index of $D$ in $K_0(C^*_r(\Gamma ))$ has a more refined version living in $K_0({\bf C}(\Gamma)\otimes{\cal R})$ where ${\cal R}$ is the algebra of smoothing operators. They deduce a cohomological formula for the higher signature $\sigma_x(M,f)$ (defined in section \ref{Novikov}) by evaluating a cyclic cocycle $\tau_x$ associated with $x$ on the index in $K_0({\bf C}(\Gamma)\otimes{\cal R})$. The two assumptions of Theorem \ref{Novikovhyper} ensure that the cocycle $\tau_x$ extends from ${\bf C}(\Gamma)\otimes{\cal R}$ to a subalgebra of the $C^*$-algebra $C^*_r(\Gamma )\otimes \mathcal{K}$ which is stable under holomorphic functional calculus. Therefore $\sigma_x(M,f)$ only depends on the index $\mu_r(f_*[D])\in K_0(C^*_r(\Gamma ))$, which is a homotopy invariant by Theorem \ref{KasNovikov}. The hypothesis in the theorem hold in particular for Gromov's hyperbolic groups:
the fact that they do satisfy the bounded cohomology property is a result stated by Gromov and proved by Mineyev \cite{Mineyev}.
\end{proof}
\end{Rem}

 In higher rank it can be proved that non-cocompact lattices do not satisfy property (RD). 
However we have a conjecture by A. Valette (see p.74 in \cite{FRR}):

 \begin{Conj} Let $\Gamma$ be a group acting properly, isometrically, with compact quotient, either on a Riemannian symmetric space or on a Bruhat-Tits building. Then $\Gamma$ has the $(RD)$ property. 
 \end{Conj}

Valette's conjecture holds in higher rank for the following special cases, all in rank 2: assume $G$ is locally isomorphic to $SL_3({\bf R})$ or $SL_3({\bf C})$: V. Lafforgue has shown  that any cocompact lattice $\Gamma$ of $G$ satisfies property (RD). I. Chatterji has generalized this proof to $SL_3({\bf H})$ and $E_{6(-26)}$, see \cite{Chatterji03}. Their proofs are based on ideas of Ramagge, Robertson and Steger for $SL_3(\Q_p)$ (\cite{RRS98}). Conjecture \ref{ConjBC} therefore follows for such lattices. As mentioned in section \ref{Trichotomy} above, this gave the first examples of infinite discrete groups having property (T) and satisfying the Baum-Connes conjecture.

It is not known whether such a group $\Gamma$ (or the Lie group $G$ itself) satisfies the conjecture with coefficients. Moreover, nothing is known about the Baum-Connes conjecture for general discrete subgroups of $G$. In particular it is not known whether $SL_3({\Z})$ satisfies Conjecture \ref{ConjBC}, or similarly whether $SL_3({\bf R})$ satisfies Conjecture \ref{ConjBCcoeff}.

On the other hand, regarding lattices in another real rank 2 simple Lie group (e.g. the symplectic group $Sp_4(\R)$), or in a simple group with real rank at least 3, absolutely nothing is known, in particular for lattices in $SL_n({\bf R})$ or $SL_n({\bf C})$ when $n\geq 4$.

\begin{Rem}
The group $\Gamma=SL_3(\mathbf{Z})$ does not have property (RD) (see \cite{Jolissaint}). Moreover, there is no unconditional completion $\mathcal{B}(\Gamma)$ that is a dense subalgebra of  $C^*_r(\Gamma)$ stable under holomorphic calculus. The following argument is due to Lafforgue (see \cite{Lafforgue2010}). Let us consider the action of $\mathbf{Z}$ on $\mathbf{Z}^2$ induced by the map from $\mathbf{Z}$ to $\mathbf{Z}^2$ that sends $n\in\mathbf{Z}$ to $\begin{pmatrix}3&1\\2&1\end{pmatrix}^n$ and the semi-direct product $H:=\mathbf{Z}\rtimes \mathbf{Z}^2$ constructed using this action. The group $H$ is solvable, hence amenable, and can be embedded as a subgroup of $SL_3(\mathbf{Z})$ using the map: $(n,\begin{pmatrix}a\\b\end{pmatrix})\mapsto \begin{pmatrix}\begin{pmatrix}3&1\\ 2&1\end{pmatrix}^n &\begin{pmatrix}a\\b\end{pmatrix}\\ 0&1 \end{pmatrix}$. Suppose by contradiction that there is an unconditional completion $\mathcal{B}(G)$ that is a subalgebra of $C_r^*(G)$. Then the algebra $\mathcal{B}(H)=\mathcal{B}(G)\cap C^*_r(H)$ is contained in $\ell^1(H)$ because as $H$ is amenable, for every non-negative function $f$ on $H$, one has $\norme{f}_{C^*_r(H)}=\norme{f}_{L^1(H)}$. However, $\ell^1(H)$ is not spectral in $C^*_r(H)$ (see \cite{Jenkins}).
\end{Rem}

 \subsection {Back to Hilbert spaces}
 
 The motto of this section is the following: in the case where property (T) imposes that $\gamma\neq 1$ in $KK_G({\bf C},{\bf C})$, the idea for showing that $\gamma$ nevertheless acts by the identity in the K-theory groups $K_*(C^*_{r}(G,A))$ is to make the  $\gamma$-element homotopic to the trivial representation in a weaker sense, getting out of the class of unitary representations, but staying within the framework of Hilbert spaces.
 
 \subsubsection {Uniformly bounded and slow growth representations}\label{slowgrowth}
 
 The idea of  using uniformly bounded representations is a remark that P. Julg made  in 1994. A uniformly bounded representation of a locally group $G$ is a strongly continuous representation by bounded operators on a Hilbert space $H$, such that 
 there is a constant $C$ with 
$\Vert\pi (g)\Vert\leq C$ for any $g\in G$. Equivalently, it is a representation by isometries for a Banach norm equivalent to a Hilbert norm.
 
Following Kasparov \cite{KaspConsp}, let us denote $R(G)=KK_G({\bf C},{\bf C})$. Let $R_{\rm ub}(G)$ be the group of homotopy classes of $G$-Fredholm modules, with uniformly bounded representations replacing unitary representations, as in \cite{Julg97}. 

\begin{Prop}
 For any $G-C^*$-algebra $A$, the Kasparov map $$R(G)\rightarrow {\rm End}K_*(C_r^*(G,A))$$ factors
through the map $R(G)\rightarrow R_{\rm ub}(G)$.
\end{Prop}

This follows from an easy generalization of the classical Fell lemma: indeed, if $\pi$ is a uniformly bounded representation of a group $G$ in a Hilbert space $H$, and $\lambda$ 
is the left regular representation of $G$ on $L^2(G)$, there exists a bounded invertible operator $U$ on $H\otimes L^2(G)$, such that
$$\pi (g)\otimes\lambda (g)=U(1\otimes \lambda (g))U^{-1}$$
When $\pi$ is a unitary representation, $U$ is of course a unitary operator.

To any Hilbert space $H$ equipped with a uniformly bounded representation $\pi$,  let us associate as in the construction of the map $j_{G,r}$ from \cite{KaspConsp, Kasparov88}, the
Hilbert module $E=H\otimes C_r^*(G,A)$ and  the covariant representation of $(G,A)$ with
values in ${\cal L}_{C_r^*(G,A)}(E)$ defined by:

$$a\mapsto 1\otimes a,\  g\mapsto\pi(g)\otimes\lambda (g).$$ 

Then the representation
 $\pi_A : C_c(G,A)\rightarrow {\cal L}_{C_r^*(G,A)}(E)$ extending the above covariant
representation factors through the reduced crossed product $C_r^*(G,A)$.

To a $G$-Fredholm module $(H,\pi ,T)$ we can therefore associate the triple $(H\otimes C_r^*(G,A), \pi_A, T_A)$ where 
$\pi_A: C_r^*(G,A)\rightarrow {\cal L}_{C_{\rm red}^*(G,A)}(E)$ is the Banach algebra homomorphism defined above, and
$T_A=T\otimes 1\in {\cal L}_{C_{\rm red}^*(G,A)}(E)$.
The Banach $G$-Fredholm module thus obtained defines a map from the group $K_*(C_{\rm red}^*(G,A))$ to itself. Note that such a construction has no analogue for $C_{\rm max}^*(G,A)$ since it relies upon a specific feature of the regular representation. 
\medskip

As in the case of Lafforgue's Banach representation, it often happens that a family of representations can be deformed to a representation containing the trivial representation, but with a uniform boundedness constant tending to infinity. One must therefore use a more general class, as we now explain. Fix $\varepsilon >0$. Let $l$ be a length function on $G$.

\begin {Def} We say that a representation $\pi$ of $G$ is {\it of $\varepsilon$-exponential type} if there is
a constant $C$ such that for any
$g\in G$, 
$$\Vert\pi (g)\Vert\leq C e^{\varepsilon l(g)}$$
\end{Def}

The following ideas come from a discussion between N. Higson, P. Julg and V. Lafforgue in 1999. We define as above a $G$-Fredholm module of $\varepsilon$-exponential
type, and similarly a homotopy of such modules. Let $R_{\varepsilon}(G)$ be the abelian group of homotopy classes.
The obvious maps $R_{\varepsilon}(G)\rightarrow R_{\varepsilon '}(G)$ for $\varepsilon <\varepsilon '$ form a projective system and we 
consider the projective limit $\varprojlim R_{\varepsilon}(G)$ when $\varepsilon\rightarrow 0$.

We would like to have an analogue of the above proposition with  the group $\varprojlim R_{\varepsilon}(G)$ instead of $R_{\rm ub}(G)$. In fact there is a slightly weaker result, due to N. Higson and V. Lafforgue (cf \cite{LaffHyp} Th\'eor\`eme 2.3) which is enough for our purpose. We assume now that $G$ is a connected Lie group.

 \begin{Thm} The kernel of the map $$R(G)\rightarrow \varprojlim R_{\varepsilon}(G)$$ is included in the kernel of the map $$R(G)\rightarrow {\rm
End}K_*(C^*_r(G,A)).$$
\end{Thm}

Let us sketch the proof following \cite{LaffHyp}. As above, to any representation $\pi$ of $G$ is associated an algebra homomorphism $$\pi_A : C_c(G,A)\rightarrow {\cal
L}_{C^*_r(G,A)}(E)$$ where $E=H\otimes C_{\rm red}^*(G,A)$.

For all $\varepsilon>0$ there is a Banach algebra $C_{\varepsilon}$ which is a completion of $C_c(G,A)$ such that for any representation $\pi$ of $\varepsilon$-exponential type,
 the above map
$\pi_A$ extends to a bounded homomorphism $C_{\varepsilon}\rightarrow{\cal L}_{C^*_r(G,A)}(E)$. The Banach Fredholm module thus obtained defines a map 
$$R_{\varepsilon}(G)\rightarrow {\rm Hom}(K_*(C_{\varepsilon}),K_*(C^*_r(G,A))).$$
This being done for each $\varepsilon$, we have a system of maps compatible with the maps $C_{\varepsilon}\rightarrow C_{\varepsilon'}$ for $\varepsilon'<\varepsilon$, so that 
there is a commutative diagram (cf \cite{LaffHyp} prop 2.5)

$$\xymatrix{R(G)\ar[d]\ar[r]&{\rm End}K_*(C^*_r(G,A))\ar[d]\\
 \varprojlim R_{\varepsilon}(G)\ar[r]& \varprojlim {\rm Hom}(K_*(C_{\varepsilon}),K_*(C^*_r(G,A))).}$$

The theorem of Higson-Lafforgue then follows immediately, thanks to the following lemma:

\begin{Lem} The group $K_*(C^*_r(G,A))$ is the union of the images of the maps $K_*(C_{\varepsilon})\rightarrow K_*(C^*_r(G,A))$.
\end{Lem}
To prove the lemma, Higson and Lafforgue use the fact that the symmetric space $Z=G/K$ has finite asymptotic dimension. They give an estimate of the form (prop 2.6 in \cite{LaffHyp})
$$\Vert f\Vert_{C_{\varepsilon}}\leq k_{\varepsilon}e^{\varepsilon (ar+b)}\Vert f\Vert_{C^*_r(G,A)}$$
for $f\in C_c(G,A)$ with support in a ball of radius $r$ (for the length $l$).

The spectral radius formula in Banach algebras then implies for such an $f$,
$$\rho_{C_{\varepsilon}}(f)\leq e^{\varepsilon ar}\rho_{C^*_r(G,A)}(f),$$
so that $\rho_{C^*_r(G,A)}(f)=\inf \rho_{C_{\varepsilon}}(f)$. This fact, by standard holomorphic calculus techniques, implies the lemma.

 \subsubsection {Cowling representations and $\gamma$}\label{Cowling}
 
 The beautiful work of M. Cowling and U. Haagerup on completely bounded multipliers of the Fourier algebras for rank one simple Lie groups \cite{Cowling-Haagerup} inspired P. Julg to use Cowling's strip of uniformly bounded representations to prove the Baum-Connes conjecture for such groups.
  Consider the Hilbert space $L^2(G/P)$ associated to a $K$-invariant measure on the flag manifold $G/P$. Let $\pi_1$ be the natural action of $G$, i.e. $\pi_1 (g)f=f\circ g^{-1}$, and let $\pi_0$ be the unitary representation obtained by twisting $\pi_1$ by a suitable cocycle: $\pi_0 (g)=\lambda_g^{N/2}\pi_1 (g)$. One can interpolate between $\pi_0$ and $\pi_1$ by taking 
 $$\pi_s(g)=\lambda_g^{\frac{(1-s)N}{2}}\pi_1(g),$$ 
 with $s$ being a complex number. The result of M. Cowling \cite{Cowling82, ACdB} is the following:
 
 \begin {Thm} The representation
 $(1+\Delta_E)^{(1-s)N/4}\pi_s (g)(1+\Delta_E)^{-(1-s)N/4}$ is uniformly bounded for any $s$ in the strip $-1< \Re s <1$.
 \end {Thm}
 
 In particular this holds for $-1< s <1$. The important point is to compare with Kostant's result on the unitarizability of $\pi_s$. The representations $\pi_s$ are by construction unitary if $\Re s=0$. Otherwise they are unitarizable (i.e. admit an intertwining operator  $T_s$ such that $T_s^{-1}\pi_s(g)T_s$ is unitary) if and only if $-c<s<c$  for a certain $c\leq 1$. This is the so called complementary series. The critical value $s$ is as follows:
 \begin{enumerate}
\item[1)]
If $G=SO_0(n,1)$ or $SU(n,1)$, $c=1$.
\item[2)]
If $G=Sp(n,1)$, $c={2n-1\over 2n+1}$
\item[3)] If  $G=F_{4(-20)}$, $c={5\over 11}$.
\end{enumerate}
In case 1, $G$ has the Haagerup property, and the complementary series approaches the trivial representation. In cases 2 and 3 one has $c<1$ so that there is a gap between the complementary series and the trivial representation, as expected from property (T). 

The above family $\pi_s$ ($0\leq s<1$) and its generalizations to the other principal series are the tool for constructing a homotopy between $\gamma$ and $1$. Indeed the proofs of $\gamma=1$ by Kasparov \cite{Kasparov84}, Chen \cite{Chen} and Julg-Kasparov \cite{Julg-Kasparov}  rest upon the complementary series. In the general case, Julg \cite{Julg} constructs a similar homotopy involving Cowling uniformly bounded representations. Modulo some (not yet fully clarified) estimates, that would prove that $\gamma$ is $1$ in $R_{\varepsilon}(G)$ for all $\varepsilon >0$ (with the above notations).

 \subsubsection{Lafforgue's result for hyperbolic groups}\label{hyperLafforgue}
 
 In 2012, in a very long and deep paper, Vincent Lafforgue has proved the following result.

 \begin{Thm}\label{hypcoeff} Let $G$ be a word-hyperbolic group. Then $G$ satisfies the Baum-Connes conjecture with coefficients (conjecture \ref{ConjBCcoeff}).
  \end{Thm}
  
 \begin {Rem} Lafforgue proves more generally the same result for $G$ a locally compact group acting continuously, isometrically and properly on a metric space $X$ which is hyperbolic, weakly geodesic and uniformly locally finite.  \end{Rem} 
 
 Let us sketch the main steps of Lafforgue's proof. The basic geometric object is the Rips complex $\Delta =P_R(G)$ of the group $G$ seen as a metric space with respect to the word-metric $d_S$ associated with a set of generators $S$. 

\begin{Def}\label{Rips} Let $Y$ be a locally finite metric space (i.e. every ball in $Y$ is finite). Fix $R\geq 0$. The {\it Rips complex} $P_R(Y)$ is the simplicial complex with vertex set $Y$, such that a subset $F$ with $(n+1)$-elements spans a $n$-simplex if and only if $diam(F)\leq R$.
\end{Def}

Because $G$ is hyperbolic, one can choose the radius $R$ big enough so that  $\Delta$ is contractible.
 Let $\partial$ be 
 the coboundary 
 $${\bf C}[\Delta^0]\leftarrow {\bf C}[\Delta^1]\leftarrow {\bf C}[\Delta^2]\leftarrow ...$$
 of the Rips complex. Let us recall the formula for $\partial$:
 $$\partial \delta_{g_0,g_1,...,g_k}=\sum_{i=0}^k (-1)^i\delta_{g_0,...,\hat g_i,...,g_k}$$
 
Contractibility of the Rips complex implies that the homology of the complex $\partial$ is zero in all degrees, except in degree $0$ where it is one-dimensional. But a concrete contraction onto the origin $x_0$ of the graph gives rise to a parametrix, i.e. maps $h: {\bf C}[\Delta^k] \rightarrow {\bf C}[\Delta^{k+1}]$ such that $\partial h+h\partial=1$ (except in degree zero where it is $1-p_{x_0}$ where $p_0$ has image in ${\bf C}\delta_{x_0}$)
 ) and $h^2=0$.  
 The prototype is the case of a tree, where $h\delta_x=\sum \delta_e$, the sum being extended to the edges on the geodesic from $x_0$ to $x$. The case of a hyperbolic group is more subtle, and the construction of $h$ has to involve some averaging over geodesics. Suitable parametrices have been considered by V. Lafforgue in the Banach framework. 

 Kasparov and Skandalis  in  \cite{Kasparov-Skandalis91} have shown that hyperbolic groups admit a $\gamma$-element which can be represented by an operator on the space $\ell^2(\Delta)$.
 Lafforgue considers the following variant of the Kasparov-Skandalis construction. 
  Let us conjugate the operator $\partial+h$  by a suitable function of the form $e^{t\rho}$ where $\rho$ is the (suitably averaged) distance function to the point $x_0$. Then for $t$ big enough, the operator $e^{t\rho}(\partial+h)e^{-t\rho}$, on the Hilbert space $\ell^2(\Delta)$ equipped with the even/odd grading and the natural representation $\pi$ of $G$, represents the $\gamma$-element. 
 
Lafforgue's {\it tour de force} is to modify the construction of the operator $h$ and to construct Hilbert norms $\Vert .\Vert_{\varepsilon}$ on ${\bf C}[\Delta]$ such that the operators $e^{t\rho}(\partial +h)e^{-t\rho}$ become a homotopy between $\gamma$ (for $t$ big ) and $1$ (for $t=0$), this homotopy being through $\varepsilon$-exponential representations. Let us give the precise statement:

\begin{Thm} Let  $G$ be a word hyperbolic group; let $\Delta$ and $\partial$ be  as above. Fix $\varepsilon >0$. There exists a suitable parametrix $h$ satisfying the conditions above, a Hilbert completion $H_{\varepsilon}$  of the space ${\bf C}[\Delta]$, and a distance function $d$ on $G$ differing from $d_S$ by a bounded function such that : 
\begin{enumerate}
\item 
the operator $F_t=e^{t\rho}(\partial +h)e^{-t\rho}$ (where $\rho$ is the distance to the origin $x_0$) extends to a bounded operator on $H_{\varepsilon}$ for any $t$, \item the representation $\pi$ of $G$ extends to a representation on $H_{\varepsilon}$ with estimates 
$\Vert\pi (g)\Vert_{\varepsilon}\leq Ce^{\varepsilon d(gx_0,x_0)}$, \item the operators $[F_t,\pi (g)]$ are compact on $H_{\varepsilon}$. \end{enumerate}
\end{Thm}

Let us give an idea of how the Hilbert norms $\Vert .\Vert_{\varepsilon}$ on ${\bf C}[\Delta]$ are contructed. It is most enlightening to consider the prototype case of trees. Let $S^n$ denote the sphere of radius $n$, i.e. the set of vertices at distance $n$ from the origin $x_0$ and $B^n$ the ball of radius $n$, i.e. the set of vertices at distance $\leq n$ of $x_0$. Suppose that $f\in {\bf C}[\Delta^0]$ has support in $S^n$. Then $$\Vert f\Vert_{\varepsilon}^2=e^{2\varepsilon n}\sum_{z\in B^n}\vert \sum_{x\rightarrow z} f(x)\vert^2$$
where the last sum is over all $x\in S^n$ such that $z$ lies on the path from $x_0$ to $x$. For general $f\in {\bf C}[\Delta^0]$, one defines $\Vert f\Vert_{\varepsilon}^2=\sum_{n=0}^{\infty} \Vert f_n\Vert_{\varepsilon}^2$ where $f$ is the restriction of $f$ to $S^n$.
A similar formula defines the norm $\Vert .\Vert_{\varepsilon}$ on ${\bf C}[\Delta^1]$. The way the norm $\Vert .\Vert_{\varepsilon}$ is constructed makes relatively easy to prove the continuity of the operator $e^{t\rho}(\partial +h)e^{-t\rho}$ for any $t$ (and uniformly with respect to $t$). More subtle is the estimate for the action $\pi (g)$ of a group element $g$.  Equivalently, it amounts to compare the norms $\Vert .\Vert_{\varepsilon}$ for two choices of $x_0$. Lafforgue establishes an inequality of the form 
 $$\Vert \pi (g)\Vert_{\varepsilon}\leq P(l(g))e^{\varepsilon l(g)}$$ with a certain polynomial $P$. In particular $\Vert \pi (g)\Vert_{\varepsilon}\leq C e^{\varepsilon' l(g)}$ for any $\varepsilon'>\varepsilon$.
 
 According to the philosophy of Gromov, the geometry of trees is a model for the geometry of general hyperbolic spaces. The implementation of that principle can however be technically hard. In our case, Lafforgue needs almost 200 pages of difficult calculations to construct the analogue of the norms $\Vert .\Vert_{\varepsilon}$ above and for all the required estimates. We refer to \cite{LaffHyp} and \cite{Puschnigg} for the details.

\subsection{Strong property (T)}\label{strong(T)}
 
Theorem \ref{hypcoeff} yields examples of discrete groups with property (T) satisfying Conjecture \ref{ConjBCcoeff}. Indeed, many hyperbolic groups have property (T). On the other hand, as a by-product of his proof, Lafforgue shows that hyperbolic groups do not satisfy a certain strengthening of property (T), in which unitary representations are replaced by $\varepsilon$-exponential representations. To that effect, let us consider the representation $\pi$ of $G$ on the completion of ${\bf C}[\Delta^0]$ for the norm 
$\Vert .\Vert_{\varepsilon}$. 

\begin{Lem} The representation $\pi$ on $H_{\varepsilon}$ has no non zero invariant vector, whereas its contragredient  $\check {\pi}$ does have non zero invariant vectors.
\end {Lem}

\begin{proof}[Proof] The first fact is obvious since a constant function is not in $H_{\varepsilon}$. On the other hand, the $G$-invariant form $f\mapsto\sum_{g\in G} f(g)$ extends to a continuous form on $H_{\varepsilon}$. Let us explain that point in the case of  a tree: it follows immediately from the definition of the norm $\Vert .\Vert_{\varepsilon}$ that any $f\in {\bf C}[\Delta^0]$ satisfies the inequality

 $$\sum_{n=0}^{\infty}e^{2\varepsilon n}\vert \sum_{x\in S^n} f(x)\vert^2\leq \Vert f\Vert_{\varepsilon}^2$$
 hence by Cauchy-Schwarz inequality,
$$\vert \sum f(x)\vert^2\leq 
(\sum_{n=0}^{\infty}e^{-2\varepsilon n})(\sum_{n=0}^{\infty}e^{2\varepsilon n}\vert\sum_{x\in S^n} f(x)\vert^2)\leq
(1-e^{-2\varepsilon})^{-1}\Vert f\Vert_{\varepsilon}^2$$

The identification of $H_{\varepsilon}$  with its dual therefore gives a non zero invariant vector for the contragredient representation $\check {\pi}$.
\end{proof}

Let $G$ be a locally compact group, $l$ a length function on $G$, and real numbers $\varepsilon >0$, $K>0$. Let ${\cal F}_{\varepsilon , K}$ the family of representations $\pi$ of $G$ on a Hilbert space satisfying $\Vert \pi (g)\Vert \leq Ke^{\varepsilon l(g)}$, and let 
 ${\cal C}_{\varepsilon , K}(G)$ be the Banach algebra defined as the completion of $C_c(G)$ for the norm
$\sup \Vert \pi (f)\Vert $, where the supremum is taken over representations $\pi$ in ${\cal F}_{\varepsilon , K}$.

\begin{Def} A Kazhdan projection in the Banach algebra ${\cal C}_{\varepsilon , K}(G)$ is an idempotent element $p$ satisfying the following  condition:
for any representation $\pi$ belonging to ${\cal F}_{\varepsilon , K}$,  on a Hilbert space $H$, the range of the idempotent $\pi (p)$ is the space $H^{\pi}$ of $G$-invariant vectors.
\end{Def}

\begin {Rem} The above definition is given in a more general setting by M. de la Salle \cite{dlS16}, whose  Proposition 3.4 and Corollary 3.5 also show that, since the family ${\cal F}_{\varepsilon , K}$ is stable under contragredient, a Kazhdan projection is necessarily central, hence unique and self-adjoint.
\end {Rem}

 The above lemma has the following consequence:

\begin {Cor} Let $G$ be a hyperbolic group. Then for any $\varepsilon>0$ there exists $K>0$ such that the Banach algebra ${\cal C}_{\varepsilon , K}(G)$ has no Kazhdan projection.
\end {Cor}

Indeed, assume there is such a projection $p$.  By the above remark $p$ is self-adjoint, so that $\pi (p)^*=\check{\pi}(p)$, where $\pi$ is the representation of $G$ in $H_{\varepsilon}$. 
But by the lemma, $\pi (p)=0$ and $\check{\pi}(p)\neq 0$,  a contradiction.

The following definition should be thought as a strengthening of the caracterization of Kazhdan's property (T) by a Kazhdan projection in $C^*_{\rm max}(G)$, cf. proposition \ref{Kazhdanproj}. 

\begin{Def}\label{StrongT} The group $G$ has strong property (T) for Hilbert spaces if for any length function $l$, there exists an  $\varepsilon >0$ such that for every $K$ there is a Kazhdan projection in ${\cal C}_{\varepsilon ,K}(G)$.
\end{Def}

We thus conclude:

\begin{Thm} Gromov-hyperbolic groups do not satisfy the strong property (T) for Hilbert spaces.
\end {Thm}

On the other hand, it follows from the works of V. Lafforgue, B. Liao, T. de Laat and M. de la Salle (see \cite{LafforgueT, Benben14, delaSalleb, dLdlS}) that in higher rank the situation is completely different.

\begin {Thm}\label{StrongThigher} Let $G$ be a simple connected Lie group of real rank $\geq 2$ or a simple algebraic group of split rank $\geq 2$ over a non archimedian local field. Then $G$ has strong property (T) in Hilbert spaces. The same holds for any lattice in such a $G$.
\end {Thm}

Lafforgue more generally defines strong property (T) for a given class ${\cal E}$ of Banach spaces. The theorem above also holds provided the class of Banach spaces ${\cal E}$ has a nontrivial type, i.e. if the Banach space $\ell^1$ is not finitely representable in ${\cal E}$.

Strong property (T) had been introduced by Lafforgue \cite{Lafforgue2010}  to understand the obstruction, if not to the Baum-Connes conjecture, at least to the proofs considered so far. But in fact, he has been led to introduce the following variant of strong property (T). We consider a locally compact group $G$ and a compact subgroup $K$. Let $l$ be a $K$-biinvariant length function on $G$ and $\varepsilon>0$.

\begin {Def} An $\varepsilon$-exponential $K$-biinvariant Schur multiplier is a $K$-biinvariant function $c$ on $G$ such that for any $K$-biinvariant function $f$ on $G$ with values in $C_c(G)$ and support in the ball of radius $R$ for the length $l$, $$\Vert cf\Vert\leq e^{\varepsilon R}\Vert f\Vert$$ where $cf$ is the pointwise product on $G$ and $\Vert .\Vert$ is the norm in the crossed product $C^*(G, C_0(G))=\mathcal{K}(L^2(G))$.
\end{Def}

\begin {Def} The group $G$ has Schur property (T) relative to the compact subgroup $K$ if for any $K$-biinvariant length function $l$, there exists $\varepsilon >0$ and a $K$-biinvariant function $\varphi$ on $G$ with nonnegative values and vanishing at infinity satisfying the following property:  any $\varepsilon$-exponential $K$-biinvariant Schur multiplier $c$ has a limit $c_{\infty}$ at infinity and  satisfies $\vert c(g)-c_{\infty}\vert\leq\varphi (g)$ for any $g\in G$.
\end{Def}

Lafforgue explains in \cite{Lafforgue2010} that Schur property (T) for a group $G$ relative to a compact subgroup $K$ is an obstacle to the above attempts to prove the Baum-Connes conjecture. It contradicts the existence, for any $G-C^*$ algebra $A$ and any $\varepsilon>0$, of a Banach subalgebra ${\cal B}$ of the reduced crossed product $C^*_r(G,A)$ satisfying the inequality $\Vert f\Vert_{\cal B}\leq e^{\varepsilon R}\Vert f\Vert_{C^*_r(G,A)}$ for any $f\in C_c(G,A)$ supported in the ball of radius $R$. In particular, supposing that $G$ admits a $\gamma$ element, it is hopeless to try to prove the Baum-Connes conjecture with coefficients using a homotopy of $\gamma$ to 1 through $\varepsilon$-exponential representations as suggested above. It is also shown in  \cite{Lafforgue2010} that $SL_3(\R)$ and $SL_3(\Q_p)$ do satisfy Schur property (T) with respect to their maximal compact subgroups. B. Liao \cite{Benben16} has a similar result for the group $Sp_4$ over a nonarchimedian local field of finite characteristic. It is very likely, but as far as we know not yet proved, that it is also the case for simple groups of higher rank and with finite centre. 

\begin{Rem} The logical link between strong property (T) and Schur property (T) is not completely clear. One would expect that Schur property (T) for $G$ relative to some compact subgroup $K$ implies strong property (T) for $G$. But as noted by Lafforgue, this is not quite the case. As suggested to us by M. de la Salle, there should be a natural strengthening of Schur property (T) implying strong property (T). 
\end{Rem}

\subsection{Oka principle in Noncommutative Geometry}

As explained in the previous section, Lafforgue observed that the "Dirac-dual Dirac"-like methods used so far, would probably not work to prove the Baum-Connes conjecture with arbitrary coefficients for simple Lie groups of higher rank, mainly because of the presence of a variant of strong property (T) (see section \ref{strong(T)}). In \cite{Lafforgue2010}, he even gave a necessary condition for this kind of approach to work and proved that these methods would certainly not succeed, leaving very few hope in proving further cases of the conjecture using the classical techniques. Nonetheless, he indicates that Bost's ideas on Oka principle are still open and he leaves them as a path for investigating the problem of surjectivity.\\

\subsubsection{Isomorphisms in K-theory}

In analytic geometry, the reduction of holomorphic problems to topological problems is known as Oka principle, whose classical version is the so-called \emph{Oka-Grauert principle}. In its simplest form, it states that the holomorphic classification of complex vector bundles over an analytic Stein space agrees with their topological classification. The case of line bundles was proven by Oka in 1939 and it was then generalized by Grauert in 1958 (\cite{Grauert}; see also \cite{GromovOka} for a seminal paper on the theory and \cite{Forstneric} for a survey). Let us state Grauert's Theorem regarding complex vector bundles. 

\begin{Thm}[Grauert]\label{grauert}
Let $X$ be an analytic Stein space. Then,
\begin{enumerate}
\item if $E$ and $F$ are two complex holomorphic  vector bundles over $X$ which are continuously isomorphic, then $E$ and $F$ are holomorphically isomorphic. 
\item every continuous vector bundle over $X$ carries an holomorphic vector bundle structure that is uniquely determined.
\item the inclusion $\iota : \mathcal{O}(X,GL_n(\C))\hookrightarrow {C}(X,GL_n(\C))$ of the space of all holomorphic maps $X\to GL_n(\C)$ into the space of all continuous maps is a \emph{weak homotopy equivalence} with respect to the compact-open topology, i.e. $\iota$ induces isomorphisms of all homotopy groups : $$\xymatrix{\pi_k(\iota):\pi_k\big(\mathcal{O}(X,GL_n(\C))\big)\ar[r]^-{\simeq}&\pi_k\big({C}(X,GL_n(\C))\big),\quad k=0,1,2,...}$$
\end{enumerate}
\end{Thm}

Let us assume $X$ is compact. Let $\mathcal{O}(X)$ be the set of all continuous functions on $X$ which are holomorphic on the interior of $X$, as a Banach subalgebra of $C(X)$. Then the injection $\iota : \mathcal{O}(X)\rightarrow C(X)$ is a strong isomorphism in K-theory. \\

In \cite{Bost90}, Bost asks the following question : \emph{Let $A$ and $B$ be two Banach algebras and $\iota: A\to B$ a continuous injective  morphism with dense image. What can be said about the map $\iota_*:K(A)\to K(B)$ ? More precisely, under which conditions on $\iota$ is the map $\iota_*$ an isomorphism?  } As we have already mentioned in subsection \ref{Bostcon}, the most classical criteria for the map $\iota_*$ to be an isomorphism is the fact that $A$ is a dense subalgebra stable under holomorphic calculus in $B$ (see \cite[p. 209]{Karoubi}, \cite[2.2 and 3.1]{Swan}). The discussion from section \ref{Laff} makes it clear why having a good criteria to ensure that $\iota_*$ is an isomorphism, can be very helpful when trying to prove the Baum-Connes conjecture. We will see that a closer relation can be stated. 

The link between Bost's question and Grauert's Theorem \ref{grauert} can be philosophically thought as follows. We start from a Banach algebra $B$, e.g. a $C^*$-algebra, that we may think as the algebra of continuous functions on some non-commutative space $T$. Assume that  $T$ can be imbedded in some neighborhood $X$ which is homotopic to $T$ and carries a (non commutative analogue of) complex structure. The dense subalgebra $A$ is the set of functions on $T$ which extend to functions on $X$ which are holomorphic. Then the injection $\iota: A\to B$ can be seen as the composition of the Banach space injection $A=\mathcal{O}(X)\subset C(X)$ and a restriction map $C(X)\rightarrow C(T)=B$. The first should be an isomorphism in K-theory by a non-commutative analogue of Oka-Grauert's principle, and the second by the homotopy invariance of K-theory.\\

More precisely, Bost considers the following situation. 
Let $B$ be a Banach algebra endowed with a continuous action of $\R^n$ denoted by $\alpha$. Let $F$ be a compact and convex subset of $\R^n$ containing $0$ and with nonempty interior. 
Then one defines $A=\mathcal{O}(B,\alpha,F)$  as the set of elements $a$ in $B$ such that the continuous map $t\mapsto \alpha_t(a)$ from $\R^n$ to $B$ has a continuous extension on $\R^n+i F\subset \C^n$ which is holomorphic on $\R^n+i\overset{\circ}{F}$, where $\overset{\circ}{F}$ is the interior of $F$. 

For $z\in\R^n+iF$, denote by $\alpha_z(a)\in B$ the value of the map that extends $\alpha$ at $z$. Then $A=\mathcal{O}(B,\alpha,F)$ is a Banach algebra endowed with the norm $$\norme{a}_F=\sup_{z\in\R^n+ iF}\norme{\alpha_z(a)},$$
and the inclusion map $\iota:A=\mathcal{O}(B,\alpha,F)\to B$ is dense (see \cite[3.1 and Corollaire 3.2.4]{Bost90}). As mentioned by Bost, the algebra $\mathcal{O}(B,\alpha,F)$ is not in general stable under holomorphic calculus in $B$  (see \cite[1.3.1]{Bost90}), but the map $\iota$ still induces a strong isomorphism in K-theory (see Definition \ref{Bost90Strong}, see also \cite{Nica2008} for other criteria on $\iota$ so that $\iota_*$ is an isomorphism):

\begin{Thm}{\cite[Th\'eor\`eme 2.2.1]{Bost90}} \label{Bost} Let $B$ be a complex Banach algebra endowed with an action of $\R^n$ denoted by $\alpha$. For all compact and convex subset $F$ of $\R^n$, containing $0$ and with non zero interior, the inclusion map  $\iota: A=\mathcal{O}(B,\alpha,F)\to B$ induces a strong isomorphism in K-theory.
\end{Thm}

The idea of the proof is the following. The map which to $a\in A$ associates the function $\tau\mapsto \alpha_{i\tau}(a)$ provides an isometric embedding of the Banach algebra $A=\mathcal{O}(B,\alpha,F)$ into $C(F,B)$. Bost's proof then imitates the proof of Theorem \ref{grauert} to show that the canonical injection $A\rightarrow C(F,B)$ is a strong isomorphism in K-theory. Composing with the evaluation at 0 from $C(F,B)$ to $B$ (which is also a strong isomorphism theorem by the usual homotopy argument) yields the result.

The following examples are the basic examples of \cite{Bost90}. Example \ref{exemple1Bost} is equivalent to Grauert's theorem for a corona $U=\{z\in\C\,|\,\rho_1\leq|z|\leq\rho_2\}$:

\begin{Ex}\label{exemple1Bost}
Let $\mathbf{S}^1=\{z\in\C\,|\, |z|=1\}$ denote the unit circle, and let $B$ be the algebra $C(\mathbf{S}^1)$ of continuous functions on $\mathbf{S}^1$ with complex values. Let $\rho_1$ and $\rho_2$ be two real numbers such that $0<\rho_1<1<\rho_2$, consider the closed corona $U=\{z\in\C\,|\,\rho_1\leq|z|\leq\rho_2\}$ and let $A$ be the subalgebra of $C(U)$ of continuous functions $\phi:U\to\C$ which are holomorphic in $\stackrel{\circ}{U}$ . The algebra $A$, endowed with the norm of uniform convergence, is closed in $C(U)$ and hence it is a Banach algebra. Then, Theorem \ref{Bost} says that the inclusion map $\iota:A\to B$ induces an isomorphism in K-theory. Indeed, let $(\alpha_t f)(z)=f(e^{-it}z)$, then $(\alpha_t)_{t\in\R}$ defines a one parameter group of isometric algebra automorphisms of $B$ and $\mathcal{O}(B,\alpha,I)=A$ for $I=[\log\rho_1,\log\rho_2]\subset\R$.
\end{Ex}

\begin{Ex}
Let $B$ be the convolution algebra $l^1(\mathbf{Z})$. Let $R>0$ be a real number and let $A=\big\{(a_n)\in\C^{\mathbf{Z}}\,\big|\,\sum\limits_{n=-\infty}^{+\infty}e^{R|n|}|a_n|<+\infty\big\}$. Hence $A$ endowed with the norm $\|(a_n)\|_R=\sum\limits_{n=-\infty}^{+\infty}e^{R|n|}|a_n|$ is a Banach algebra which is densely embedded in $B$. Theorem \ref{Bost} says that the inclusion map $\iota:A\hookrightarrow B$ induces an isomorphism in K-theory. In this case, the one parameter group of isometric automorphisms of $B$ is defined by $(\alpha_t(a_n)=(e^{int}a_n)$, and if $I=[-R,R]\subset\R$, then $\mathcal{O}(B,\alpha,I)=A$.
\end{Ex}

\begin{Ex}
The previous example can be also considered with coefficients so that things can be formulated in a noncommutative way : if $A$ is a Banach algebra and $\alpha$ is an action of $\mathbf{Z}$ by isometric automorphism of $A$, let $B:=\ell^1(\mathbf{Z},A)$ be the completion of the convolution algebra $C_c(\mathbf{Z},A)$ given by $\|(b_n)_n\|_{1}=\sum\limits_{n\in\mathbf{Z}}\|b_n\|_A$, for $(b_n)_n\in C_c(\mathbf{Z},A)$. The product in $B$ is given by twisted convolution, i.e $(bb')_n=\sum\limits_{k\in\mathbf{Z}}b_k\alpha(k)(b_{n-k})$, for $b,b'\in C_c(\mathbf{Z},A)$. For all $t\in\R$, set $\beta_t((b_n)_n)=(e^{-int}b_n)_n$ and  
$$\mathcal{O}(B,\beta,I)=\big\{(b_n)_n\in \ell^1(\mathbf{Z},A)\,\big|\,\sum\limits_{n=-\infty}^{+\infty}e^{R|n|}\|a_n\|_A<+\infty\big\},$$ where $I=[-R,R]$. Then Theorem \ref{Bost} applies and $\mathcal{O}(B,\beta,I)\hookrightarrow  B=\ell^1(\mathbf{Z},A)$ induces an isomorphism in K-theory. 
\end{Ex}

Theorem \ref{Bost} can be applied to more general crossed products algebras for which it states that a certain subalgebra defined using an exponential decay condition on $L^1(G)$ has the same K-theory as $\L^1(G)$. For a general locally compact group $G$, a Banach $G$-algebra $B$ and a continuous function $a:G\to\R^+$ such that $a(g_1g_2)\leq a(g_1)+a(g_2)$, for $g_1,g_2\in G$, define a subspace  $\mathrm{Exp}_a(G,B)$ of $L^1(G,B)$ by the following decay condition :  $$\phi\in\mathrm{Exp}_a(G,B)\quad\text{if and only if}\quad e^a\phi\in L^1(G,B).$$  Then, endowed with the norm given by $\|\phi\|_a=\|e^a\phi\|_1$, $\mathrm{Exp}_a(G,B)$ is a Banach dense subalgebra of $L^1(G,B)$. Bost proved that if $G$ is an elementary abelian group, then $K_*(\mathrm{Exp}_a(G,B))$ is isomorphic to $K_*(L^1(G,B))$. Let us state his result more precisely, 

\begin{Thm}{\cite[Th\'eor\`eme 2.3.2]{Bost90}}\label{elementaryabelian}
Let $G$ be a locally compact group and $B$ a Banach algebra endowed with an action of $G$. If $G$ is an extension by a compact group of a group of the form $\mathbf{Z}^p\times \mathbf{R}^q$ (i.e. there is a compact group $K$ and a short exact sequence $1\to K\to G \to \mathbf{Z}^p\times \mathbf{R}^q\to 1$), then, for every subadditive function $a:G\to \R_+$, the inclusion morphism $$\mathrm{Exp}_a(G,B)\hookrightarrow L^1(G,B)$$ induces an isomorphism in K-theory.
\end{Thm}

\subsubsection{Relation with the Baum-Connes conjecture}
Since we are dealing with $K$-theoretic issues, we focus on the right hand side of the assembly map and therefore we are interested in surjectivity: let $G$ be a group for which injectivity of the Baum-Connes assembly map is known (take for example any group in Lafforgue's class $\mathcal{C}$), and let $A$ be a $G$-$C^*$-algebra. Let $\rho:G\to\mathrm{End}(V)$ be a representation of $G$ on a complex hermitian vector space $V$ of finite dimension. Then the norm of $\rho(g)$ can be used as a weight to define exponential decay subalgebras of crossed product algebras. In the case of $L^1(G,A)$, these are easy to define : using the notation of the previous paragraph and taking $a(g)=\log{\|\rho(g)\|}$, denote by $\mathrm{Exp}_{\rho}(G,B):=\mathrm{Exp}_a(G,B)$ which is the completion of $C_c(G,A)$ for the norm $$\|f\|_{\small 1,\rho}=\int_G\|f(g)\|_A(1+\|\rho(g)\|_{\small\mathrm{End}(V)})dg.$$

Hence $\mathrm{Exp}_{\rho}(G,A)$ is a dense subalgebra of $L^1(G,A)$ and the representation $\rho$ is used as a weight to define exponential decay subalgebras of $L^1$. An Oka principle applied to this case, would state that these two algebras have the same K-theory. \\
Notice that for all groups belonging to the class $\mathcal{C'}$, as the algebra $\mathrm{Exp}_{\rho}(G,\C)$ is an unconditional completion, by Theorem \ref{Unconditionnal} we know that $$K_*(\mathrm{Exp}_{\rho}(G,A))\simeq K_*(L^1(G,A)).$$

Furthermore, we can use $\rho$ to define exponential decay subalgebras of any unconditional completion, 
\begin{Def}
Let $\mathcal{B}(G)$ be an unconditional completion of $C_c(G)$ and $A$ a $G$-$C^*$-algebra. Let $\mathcal{B}^{\rho}(G,A)$ be the completion of $C_c(G,A)$ for the norm 
$$\norme{f}_{\mathcal{B}^{\rho}}=\norme{g\mapsto \norme{f(g)}_A\norme{\rho(g)}_{\mathrm{End}(V)}}_{\mathcal{B}(G)}.$$
\end{Def}
When $\mathcal{B}(G)=L^1(G)$, if $\rho$ is satisfies the following growth condition $$\int_G\frac{1}{\|\rho(g)\|}\,dg<+\infty,$$then $L^{\rho}(G,A)$ is embedded in $L^1(G,A)$.\\ 

In the case of the reduced (resp. maximal) $C^*$-crossed products, an algebra that we call {\it weighted crossed product} and denoted by $\mathcal{A}^{\rho}_r(G,A)$ (resp. $\mathcal{A}^{\rho}(G,A)$) was defined in \cite{Gomez2010} (for more details see \ref{weighted} below). Taking $\rho$ to be {\it very large} (meaning that $\int_G\frac{1}{\|\rho(g)\|}\,dg<+\infty$) this algebra plays the same role in $C^*_r(G,A)$ as $\mathrm{Exp}_{\rho}(G,A)$ in $L^1(G,A)$; they are constructed to be some kind of ``exponential decay subalgebras'' of $C^*_r(G,A)$. Suppose now that $G$ is a group for which the Bost conjecture is known to be true, in other words, the map $\mu^A_{L^1}:K_*^{\mathrm{top}}(G,A)\to K_*(L^1(G,A))$ is an isomorphism. We will see that taking $\rho$ very large allows us to have a morphism $\iota : K_*(\mathcal{A}^{\rho}_r(G,A))\to K_*(L^1(G,A))$ and hence a morphism $\varphi: K_*(\mathcal{A}^{\rho}_r(G,A))\to K_*(C^*_r(G,A))$ (see Proposition \ref{iota} below), so that the following diagram is commutative : 

$$\xymatrix{
&K_*(C^*_r(G,A))& K_*(\mathcal{A}^{\rho}_r(G,A))\ar[l]_{\varphi}\ar@{.>}[ddl]^{\iota}\\
K_*^{\mathrm{top}}(G,A)\ar[ur]^{\mu_{A,r}}\ar[dr]_{\simeq}^{\mu^A_{L^1}}&&\\
&K_*(L^1(G,A))\ar[uu]&K_*(L^{1,\rho}(G,A))\ar[l]_{\simeq}\ar[uu]
}$$

A suitable Oka principle applied to these crossed products states that the weighted group algebras $\mathcal{A}^{\rho}_r(G,A)$, have the same K-theory as $C^*_r(G,A)$, i.e. $\varphi$ is an isomorphism. This would then imply the surjectivity of $\mu_{A,r}$ and hence the Baum-Connes conjecture with coefficients for $G$.

\subsubsection{Weighted group algebras}\label{weighted}
In this section, we will recall the construction of weighted group algebras constructed in \cite{Gomez2010}. 
Let us first recall some definition and establish some notation.\\ 
Let $G$ be a locally compact group and let $dg$ a left Haar measure on $G$. Let $\Delta$ be the modular function on $G$ (i.e $dg^{-1}=\Delta(g)^{-1}dg$ for all $g\in G$).\\

 Let $A$ be a  $G$-$C^*$-algebra. For all $g\in G$ and for all $a\in A$, let $g.a$, or $g(a)$, be the action of $g$ on $a$. The space of continuous functions with compact support on $G$ with values in $A$, denoted by $C_c(G,A)$, is endowed with a structure of involutive algebra where the multiplication and the involution are given, respectively, by the formulas:
$$(f_1*f_2)(g)=\int_Gf_1(g_1)g_1(f_2(g_1^{-1}g))dg_1,$$ for $f_1,f_2\in C_c(G,A)$  and 
$$f^*(g)=g(f(g^{-1}))^*\Delta(g^{-1}),$$ for $f\in C_c(G,A)$ and $g\in G$.
In a general, we write every element  $f$ in $C_c(G,A)$ as the formal integral $\int_Gf(g)e_gdg$, where $e_g$ is a formal letter satisfying the following conditions:
$$e_ge_{g'}=e_{gg'},\quad e_g^*=(e_g)^{-1}=e_{g^{-1}}\quad\text{and}\quad e_gae_g^*=g.a,$$ for all $g,g'\in G$ and for all $a\in A$.

We denote by $C^*_{\rm max}(G,A)$ and $C^*_r(G,A)$ the maximal and the reduced crossed product of $G$ and $A$, respectively.  Moreover, we denote by
$$ L^2(G,A)=\{f\in C_c(G,A)|\int_G f(g)^*f(g)dg\,\,\text{converges in}\, A\},$$
and $\lambda_{G,A}$ the left regular representation of $C_c(G,A)$ on $L^2(G,A)$
which is given by the formula: $$\lambda_{G,A}(f)(h)(t)=\int_Gt^{-1}(f(s))h(s^{-1}t)ds,$$ for $f\in C_c(G,A)$, $h\in L^2(G,A)$ and $t\in G$. Recall that $\lambda_{G,A}$ induces a unique morphism of $C^*$-algebras from $C^*_{\rm max}(G,A)$ to $C^*_r(G,A)$; we also denote that morphism by $\lambda_{G,A}$, by abuse of notation.\\

Let $(\rho,V)$ be a finite dimensional representation of $G$. We then consider the map\begin{align*}
 C_c(G,A)&\rightarrow C_c(G,A)\otimes\End(V)\\
\int_Gf(g)e_gdg&\mapsto\int_Gf(g)e_g\otimes\rho(g)dg.
\end{align*}

\begin{Def}
The reduced crossed product weighted by $\rho$ of $G$ and $A$, denoted by  $\mathcal{A}_r^{\rho}(G,A)$, is the completion of $C_c(G,A)$
for the norm : 
$$\|\int_Gf(g)e_gdg\|_{A\G}=\|\int_Gf(g)e_g\otimes\rho(g)dg\|_{C_r^*(G,A)\otimes\End(V)},$$
for $f\in C_c(G,A)$. If $A=\C$, we denote it by $\mathcal{A}^{\rho}_r(G):=\mathcal{A}_r^{\rho}(G,\C).$
\end{Def}

It is then easy to prove that the reduced weighted crossed product $\mathcal{A}_r^{\rho}(G,A)$ is a Banach algebra. When $\rho$ is an unitary representation of $G$ then  $\mathcal{A}_r^{\rho}(G,A)=C_r^*(G,A)$, up to norm equivalence.
\begin{Rem}
In the same manner, we can define weighted maximal crossed products, however, we don't treat them here because of the discussion held in \ref{KvsH}.
\end{Rem}

\begin{Ex}
Let $G=\mathbf{Z}$ and let $\rho:\mathbf{Z}\to \C^*$ be a character of $\mathbf{Z}$. Let $S^{\rho}:=\{z\in \C\,|\,|z|=|\rho(1)|\}$ the circle of radius $|\rho(1)|$. Hence, $\A$ is the algebra of continuous functions on $S^{\rho}$.
\end{Ex}

\begin{Ex}
Let $G=\mathbf{Z}$ and let $\rho_1:\mathbf{Z}\to \C^*$ and $\rho_2:\mathbf{Z}\to \C$ be two characters of $\mathbf{Z}$ such that $R_1<R_2$, where $R_1=|\rho_1(1)|$ and $R_2=|\rho_2(1)|$. Then, $\mathcal{A}^{\rho_1\oplus\rho_2}(G)$ is the algebra of continuous functions on the closed corona $U:=\{z\in\C\,|\,|\rho_1(1)|\leq|z|\leq|\rho_2(1)|\}$ holomorphic on $\stackrel{\circ}{U}$. Indeed, we have the following diagram :
$$\xymatrix{
\mathcal{A}^{\rho_1\oplus\rho_2}(G)\ar[r]&C(\mathbf{S}^1,\mathrm{End}(\C^{2}))\\
\ell^{1,\rho_1\oplus\rho_2}(G)\ar[r]\ar[u]&\ell^1(\mathbf{Z},\mathrm{End}(\C^{2}))\ar[u],
}$$
where the vertical arrows are given by Fourier series and the norm in $\ell^{1,\rho_1\oplus\rho_2}(G)$ is given by $\|(a_n)_n\|=\sum\limits_{n\in\mathbf{Z}}|a_n|\|(\rho_1\oplus\rho_2)(n)\|$. It is then clear that the algebra $\ell^{1,\rho_1\oplus\rho_2}(G)$ can be identified with the algebra $$A=\big\{(a_n)\in\C^{\mathbf{Z}}\,\big|\,\sum\limits_{n=-\infty}^{+\infty}e^{|n|\log r}|a_n|<+\infty,\text{ for all }r\in ]R_1,R_2[\big\},$$
which is identified by Fourier series with the algebra of continuous functions on $U$ holomorphic on $\stackrel{\circ}{U}$. Applying Theorem \ref{Bost}, taking $R_1<1<R_2$, we get that the algebras $\mathcal{A}^{\rho_1\oplus\rho_2}(\mathbf{Z})$ and $C^*_r(\mathbf{Z})$ have the same K-theory.
\end{Ex}

In \cite{Gomez2010}, a weighted version of the Baum-Connes morphism was constructed using Lafforgue's Banach KK-theory : $$\mu_{r,A}^{\rho}:K^{\mathrm{top}}(G,A)\to K(\mathcal{A}_r^{\rho}(G,A));$$ it computes the K-theory of this weighted algebras. Analogues of Kasparov's and Lafforgue's Dirac-dual Dirac methods were proven in this context. We state them as the following two theorems,

\begin{Thm}[\cite{Gomez2010}]
Let $G$ be a locally compact group with a $\gamma$-element. Then, for every $G$-$C^*$-algebra $A$ and every finite dimensional representation $\rho$ of $G$, the weighted morphism $\mu_{r,A}^{\rho}$ is injective. If moreover, $\gamma=1$ in $KK_G(\C,\C)$, then $\mu_{r,A}^{\rho}$ is surjective.

\end{Thm}

\begin{Thm}[\cite{Gomez2009}]
Let $G$ be a locally compact group with a $\gamma$-element. If $\gamma=1$ in $KK^{\mathrm{ban}}_{G,\ell}(\C,\C)$ and there is an unconditional completion stable under holomorphic calculus in $C^*_r(G)$, then $\mu_{r}^{\rho}$ is an isomorphism for every finite-dimensional representation $\rho$ of $G$.

\end{Thm}

Hence the morphism $\mu^{\rho}_{r,A}$ is an isomorphism for example, for all  groups with the Haagerup property and more general, for all K-amenable groups; and  when $A=\C$, the morphism $\mu_r^{\rho}$ is an isomorphism for all semisimple Lie groups and all cocompact lattices in a semisimple Lie group.\\

It is worth nothing to mention that, proving that the weighted map is an isomorphism is not easier than proving the Baum-Connes conjecture; one of the reasons is that, even though the algebras $\mathcal{A}^{\rho}_r(G,A)$ are in general not $C^*$-algebras, there are constructed in a very $C^*$-algebraic way. However, the following proposition shows that the weighted crossed products can be very small when the representation $\rho$ is very large.

\begin{Prop}{\cite[Proposition 1.5]{Gomez2010}}\label{iota}
Let $\Gamma$ be a discrete group and $A$ a $\Gamma$-$C^*$-algebra. Let $\rho:\Gamma\to\mathrm{End}(V)$ a finite-dimensional representation of $\Gamma$ such that $\sum\limits_{\gamma\in\Gamma}\frac{1}{\norme{\rho(\gamma)}}$ converges. Then 
$\mathcal{A}^{\rho}_r(\Gamma,A)$ embeds into $\ell^1(\Gamma,A)$.
 \end{Prop}
We then have the inclusions $\mathcal{A}^{\rho}_r(\Gamma,A)\hookrightarrow \ell^1(\Gamma,A)\hookrightarrow C^*_r(\Gamma,A)$ and hence, if we take a group $\Gamma$ for which we know that $K_*^{\mathrm{top}}(\Gamma)\simeq K_*(\ell^1(\Gamma,A))$, proving that $\mathcal{A}^{\rho}_r(\Gamma,A)$ and $C^*_r(\Gamma,A)$ have the same K-theory would prove the surjectivity of the Baum-Connes map with coefficients for $\Gamma$.  These ideas also work for more general locally compact groups, but we don't always have a continuous map from  $\mathcal{A}^{\rho}_r(G,A)$ to $C^*_r(G,A)$ (this map exists if and only if the regular representation $\lambda_G$ is weakly contained in $\lambda_G\otimes \rho$). Nevertheless, thanks to the following proposition, we have a map at the level of K-theory:
\begin{Prop}\label{rhomap}
Let $G$ be a locally compact group and let $\rho:G\to\mathrm{End}(V)$ a finite dimensional representation of $G$ such that $\int_G\frac{1}{\norme{\rho(g)}}dg$ converges. Then, if $A$ is a $G$-$C^*$-algebra, $\mathcal{A}^{\rho}_r(G,A)\cap L^1(G,A)$ is relatively spectral in $\mathcal{A}^{\rho}_r(G,A)$.
\end{Prop}

\begin{Def}
A morphism $\phi: A\to B$ between two algebras is \emph{relatively spectral} if $\mathrm{sp}_B(\phi(x))=\mathrm{sp}_A(x)$ for all $x$ in some dense subalgebra $X$ of $A$. It is a weaker condition than being stable under holomorphic calculus and it induces an isomorphism in K-theory (see \cite{Nica2008}).
\end{Def}

As a result, we have a map from $K_*(\mathcal{A}^{\rho}_r(G,A))$ to $K_*(L^1(G,A))$ defined through $K(\mathcal{A}^{\rho}_r(G,A)\cap L^1(G,A))$ and we can prove that the following diagram is commutative:

$$\xymatrix{ K_*^{\mathrm{top}}(G,A)\ar[rd]_{\mu^{\rho}_r}\ar@/^3pc/[rr]^{\mu^A_r} \ar[r]_{\simeq}^{\mu_{L^1} }&K_*(L^1(G,A))\ar[r] &K_*(C^*_r(G,A)) \\
&K_*(\mathcal{A}^{\rho}_r(G,A))\ar[u]\ar@{.>}[ur]_{\varphi}& }.$$

Hence, we get a morphism $\varphi:K_*(\mathcal{A}^{\rho}_r(G,A))\to K_*(C^*_r(G,A))$. The following result is then straightforward:

\begin{Thm}\label{sslarge}
Let $G$ be a locally compact group with a $\gamma$-element and let $(\rho,V)$ be a finite dimensional representation of $G$ such that $\int_G\frac{1}{\norme{\rho(g)}}dg$ converges. If $\varphi$ is an isomorphism then the Baum-Connes conjecture with coefficients in $A$ is true for $G$. 
\end{Thm}
Let us give two examples of groups having a "very large" finite dimensional representation.
\begin{Ex} 
\begin{enumerate}
\item Let $G=\R$ and let $\rho:\R\to\mathrm{GL}_3(\C)$, be the representation of $G$ defined by $t\mapsto \mathrm{Exp}(tX)$ where 
$X=\begin{pmatrix}
0&1&0\\
0&0&1\\
0&0&0
\end{pmatrix}$. Then, $$\mathrm{Exp}(tX)=1+tX+\frac{t^2}{2}X^2=\begin{pmatrix}
1&t&\frac{t^2}{2}\\
0&1&t\\
0&0&1
\end{pmatrix}$$ and hence, $\norme{\mathrm{Exp}(tX)}\geq (\frac{t^4}{4}+t^2+1)^{\frac{1}{2}}=1+\frac{t^2}{2}$. It follows that,

$$\int_{-\infty}^{+\infty}\frac{dt}{\norme{\mathrm{Exp}(tX)}}\leq\int_{-\infty}^{+\infty}\frac{dt}{1+\frac{t^2}{2}}<+\infty.$$
\item Take $G=SL_2(\R)$. Set $K=SO(2)$, and let
$$A=\{a_t=\left(\begin{array}{cc}e^t & 0 \\0 & e^{-t}\end{array}\right):t\in\R\}$$
be the diagonal subgroup. Recall that the Haar measure in the Cartan decomposition $G=KA^+K$ is expressed as
$$\int_G f(g)\,dg=\int_K\int_0^\infty \int_K \sinh(2t)f(k_1a_tk_2)\;dk_1\,dt\,dk_2$$
for $f\in C_c(G)$. Let $\rho_n$ be the $(n+1)$-dimensional representation of $G$ on homogeneous polynomials of degree $n$ on $\C^2$. Then $\|\rho_n(a_t)\|=e^{nt}$ for $t\geq 0$, so that $\rho_n$ is very large exactly when $n\geq 3$.
\end{enumerate}
\end{Ex}

Accordingly, proving the Baum-Connes conjecture with coefficients for a group for which injectivity is known (for example a semisimple Lie group or one of its closed subgroups) amounts to prove that the map $\varphi$ is surjective. To illustrate the fact that proving the surjectivity of $\varphi$ fits in the framework of Oka's principle as introduced by Bost in \cite{Bost90}, let us state the following proposition. The first point is a generalization of Theorem \ref{elementaryabelian} concerning $L^1$ algebras; even though this result does not appear in \cite{Bost90}, the proof is due to Bost.

\begin{Prop}\label{TheoBost} 
Let $G$ be a locally compact group and let $\rho: G\to \mathrm{GL}_n(\R)$ a representation of $G$. 

\begin{enumerate}
\item If $\overline{\rho(G)}$ is amenable and $a(g)=\log(\norme{\rho(g)})$, then the map $K_*(\mathrm{Exp}_a(G,B))\to K_*(L^1(G,B))$ is an isomorphism.
\item If $\overline{\rho(G)}$ is amenable and $\int_G\frac{1}{\norme{\rho(g)}}dg$ converges then the map $$K_*(\mathcal{A}^{\rho}_r(G,B))\to K_*(C^*_r(G,B))$$ defined using proposition \ref{rhomap} is an isomorphism. 
\end{enumerate}
\end{Prop}

The conditions that $\int_G\frac{1}{\norme{\rho(g)}}dg$ converges and that $\overline{\rho(G)}$ is amenable imply that $G$ is amenable. This is because the condition that $\int_G\frac{1}{\norme{\rho(g)}}dg$ converges implies that $\rho $ is proper. Hence, Theorem \ref{TheoBost} does not give anything new apart from proving that the Baum-Connes conjecture is true for an amenable group. Yet, it seemed to us that this result gives a good idea of how Bost's version of Oka principle works, and therefore we give the main ideas of the proof below.\\
We will use the following properties of weighted algebras. Analogous properties are satisfied by $\mathrm{Exp}_{\rho}(G,B)$.
\begin{Lem}\label{observation} Let $\rho,\rho',\pi,\sigma$ finite dimensional representations of a locally compact group $G$.
\begin{enumerate}
\item If $\rho'$ is either a sub-representation or a quotient of $\rho$, then $\mathcal{A}^{\rho}_r(G,B)\subset \mathcal{A}^{\rho'}_r(G,B)$.
\item If $\rho=\pi\otimes \sigma$ and $\sigma$ is unitary, then $\mathcal{A}^{\rho}_r(G,B)=\mathcal{A}^{\pi}_r(G,B)$.
\item If $\rho=\bigoplus\limits_{k}\rho_k$, then $\mathcal{A}^{\rho}_r(G,B)\subset \bigcap\limits_k \mathcal{A}^{\rho_k}_r(G,B)$.
\end{enumerate}
\end{Lem}

\begin{Lem}\label{semisimplify}
Let $\rho:G\to \mathrm{GL}_n(\R)$ be a representation of a locally compact group. If $\R^n$ has a $G$-invariant filtration of the form $0=V_0\subset V_1\subset \dots \subset V_r=\R^n$ and $\sigma_k:G\to \mathrm{End}(V_k/V_{k-1})$ is the corresponding representation on $V_k/V_{k-1}$ and $\sigma=\bigoplus\limits_k \sigma_k$ is its semi-simplification, then $\mathcal{A}^{\rho}_r(G,B)\subset \mathcal{A}^{\sigma}_r(G,B)$ and, moreover, $\mathcal{A}^{\rho}_r(G,B)$ is stable under holomorphic calculus in $\mathcal{A}^{\sigma}_r(G,B)$.
\end{Lem}

If $\overline{\rho(G)}$ is amenable then the Zariski closure of $\rho(G)$ is also amenable by a result of Moore (see for example \cite[page 64]{Zimmerbook}). Using Furstenberg's Lemma we may suppose that $\rho(G)$ is contain in the a subgroup of $GL_n(\R)$ of the form $$\begin{pmatrix}
\R^*_+\times SO(n_1)&*&\dots&*\\
0& \R^*_+\times SO(n_2)&*&\vdots\\
\vdots&&\ddots&*\\
0&\dots&0&\R^*_+\times SO(n_k)

\end{pmatrix}.$$
Hence, we may apply lemma \ref{observation} with $\sigma_i=\chi_i\otimes u_i$ where $\chi_i$ is a character of $\R^*_+$ and $u_i$ is an unitary representation of $SO(n_i)$. Using the fact that $\mathcal{A}^{\sigma_i}_r(G,B)=\mathcal{A}^{\chi_i}_r(G,B)$, we get a injective morphism $$\mathcal{A}^{\rho}_r(G,B)\to \mathcal{A}^{\pi}_r(G,B)$$ where $\pi=\bigoplus\limits_{k=1}^{m}\chi_k$ and this morphism is dense and stable under holomorphic calculus. Therefore,

$$K_*\big(\mathcal{A}^{\rho}_r(G,B)\big)\simeq K_*\big(\mathcal{A}^{\pi}_r(G,B)\big).$$

It remains to prove that the inclusion $$\mathcal{A}^{\pi}_r(G,B)\to C_r^*(G,B)$$ induces an isomorphism in K-theory. \\

Let $W$ be the space of real-valued functions on $G$ defined as $W=\sum\limits_{k}\R\log(\chi_k)$. We define an action of $W$ on $C^*_r(G,B)$ by the formula  
$\alpha_{\xi}(f)(g)=f(g)e^{-i\xi(g)}$, for $f\in C_c(G,B)$ and $\xi\in W$. Then, we need to check that $$\mathcal{A}^{\pi}_r(G,B)=\mathcal{O}(K,C^*_r(G,B),\alpha)$$ where $K$ is the convex hull of $\{0, \log{\chi_k}\}$. We conclude by applying Theorem \ref{Bost}.

\section{The Baum-Connes conjecture for groupoids}\label{groupoids}
Let $\mathcal{G}$ be a locally compact, $\sigma$-compact, Hausdorff groupoid with Haar system and let $C^*_r(\mathcal{G})$ be its reduced $C^*$-algebra (see the definition below). The Baum-Connes conjecture for $\mathcal{G}$ states that a certain map $$\mu_r:K_*^{\mathrm{top}}(\mathcal{G})\to K_*(C^*_r(\mathcal{G}))$$ is an isomorphism.
Many important examples of operator algebras may be realized as the $C^*$-algebra associated to a groupoid. This is the case for example for $C^*$-algebras associated to a foliation, to an action of a group on a space as well as the $C^*$-algebra associated to a group. Therefore, a version of the Baum-Connes conjecture for groupoids allows to study the K-theory of all of these algebras in a very general framework; we will see that it is also the case for the coarse Baum-Connes conjecture developed in Chapter \ref{Coarse}. \\
The Baum-Connes map $\mu_r$ for groupoid $C^*$-algebras appeared in the work of Baum and Connes on the Novikov conjecture for foliations (see \cite{Connes-survey} for a very nice survey on the subject). In \cite{Baum-Connes83}, Baum and Connes gave a proof of the injectivity of $\mu_r$ in the case of groupoids coming from foliations that have negatively curved leaves which is based on the construction of a Dual-Dirac element following ideas of both Kasparov and Mishchenko. Using a construction of a Kasparov bivariant theory which is equivariant with respect to the action of a groupoid defined by Le Gall in \cite{LeGall99}, Tu stated in \cite{Tu99} the Dirac-dual Dirac method in a very general context. He then proved injectivity of $\mu_r$ for a class of groupoids called bolic, generalizing Kasparov and Skandalis's work for groups, and that $\mu_r$ is an isomorphism for amenable groupoids generalizing the results of Higson and Kasparov (see \cite{Tu99,Tu99Moyennables}).
\subsection{Groupoids and their $C^*$-algebras}

In this section, we recall the definition of the $C^*$-algebras associated to groupoids and the Baum-Connes conjecture for those. It is mostly taken from the survey written by Tu on the subject (\cite{Tu-survey}).

A groupoid is a small category in which all morphisms are invertible. More concretely, it is given by the following data :  

\begin{enumerate}
\item the set of objects $\mathcal{G}^{0}$, also called the unit space,
\item the set of morphisms $\Group$,
\item an inclusion $i:\Group^{0}\hookrightarrow \Group$,
\item two maps "range" and "source" $r,s:\Group\to\Group^{0}$ such that $r\circ i=s\circ i=\mathrm{Id}$,
\item an involution $\Group\to \Group$, denoted by $g\mapsto g^{-1}$ such that $r(g)=s(g^{-1})$ for every $g\in\Group$,
\item a partially defined product $\Group^{2}\to \Group$, denoted by $(g,h)\mapsto gh$, where $\Group^{2}:=\{(g,h)\in \Group\times\Group\,|\,s(g)=r(h)\}$ is the set of composable pairs.
\end{enumerate}
It is assumed moreover that the product is associative (i.e if $(g,h),(h,k)\in\Group^{2}$ then the products $(gh)k$ and $g(hk)$ are defined and are equal), that for all $g\in\Group$, $i(r(g))g=gi(s(g))=g$ and for all $g\in\Group$, $gg^{-1}=i(r(g))$.\\

A topological groupoid is a groupoid such that $\Group$ and $\Group^{0}$ are topological spaces and all maps appearing in the definitions are continuous. When a topological groupoid $\Group$ is locally compact and Hausdorff,  it is said to be 
\begin{enumerate}[label=(\alph*)]
\item \emph{principal} if $(r,s):\Group\to \Group^{0}\times \Group^{0}$ is injective,
\item \emph{proper} if $(r,s):\Group\mapsto  \Group^{0}\times \Group^{0} $ is proper,
\item \emph{\'etale}, or \emph{$r$-discrete}, if the range map $r:\Group\mapsto \Group^{0}$ is local homeomorphism, i.e if every $x\in\Group$ admits an open neighborhood $U$ such that $r(U)$ is an open subset of $\Group^{0}$ and $r:U\mapsto r(U)$ is a homeomorphism
\end{enumerate}

Before giving some examples of groupoids, let us introduce some notations : for all $x,y\in\Group^0$, $\Group_x:=s^{-1}(x)$, $\Group^x=r^{-1}(x)$, $\Group^x_y=\Group_x\cap \Group^y$.

\begin{Ex}
\begin{enumerate}
\item \emph{Groups and Spaces}. A group $G$ is a groupoid with $G^0=\{1\}$, the unit element. A space $X$ is a groupoid where $\Group=\Group^0=X$ and $r=s=\mathrm{Id}_X$.
\item An equivalence relation $R\subset X\times X$ on a set $X$ can be endowed with a groupoid structure; the unit space is $X$, the range and source maps are $r(x,y)=x$, $s(x,y)=y$, respectively, composition is defined by $(x,y)(z,t)=(x,t)$ if $y=z$ and inverses by $(x,y)^{-1}=(y,x)$. In particular, the space $X\times X$ is a groupoid. 
\item If a group $\Gamma$ acts on the right on a space $X$, then one obtains a groupoid $\Group=X\rtimes \Gamma$ by taking as a set $\Group=X\times \Gamma$ as unit space $\Group^0=X\times \{1\}\simeq X$, $r(x,\gamma)=x$, $s(x,\gamma)=x\gamma$, $x,\gamma)^{-1}=(x\gamma,\gamma^{-1})$, $(x,\gamma)(x\gamma,\gamma')=(x,\gamma\gamma')$. If $X$ is a topological space, $\Gamma$ a topological group and the cation is continuous then $X\rtimes$ is a topological groupoid, which is Hausdorff if $X$ and $\Gamma$ are. In that case, if $\Gamma$ is discrete, $X\rtimes \Gamma$ is \'etale and it is principal if the action is free. 
\item Let $X$ be topological space and take $\Group$ to be the set of equivalence classes of paths $\varphi : [0,1]\to X$ where two paths are equivalent if and only if they are homotopic with fixed endpoints. Then $\Group^0\simeq X$ is the set of equivalence classes of constant paths on $X$. If $\varphi$ is a path on $X$ and $g=[\varphi]$ is its class in $\Group$, then $r(g)=\varphi(1)$, $s(g)=\varphi(0)$, $g^{-1}=[\varphi^{-1}]$, where $\varphi^{-1}(t)=\varphi(1-t)$ and $[\varphi][\psi]=[\varphi*\psi]$, where $\varphi*\psi(t)=\varphi(2t)$ for $t\in[0,\frac{1}{2}]$ and $\varphi*\psi(t)=\psi(2t-1)$ for $t\in [\frac{1}{2},1]$. $\Group$ is called \emph{the fundamental groupoid of $X$}.
\item Let $(V,F)$ be a foliation. \emph{The holonomy groupoid} $\Group$ is the set of equivalence classes of paths whose support is contained in one leaf, where two paths are identified if they have the same end points and they define the same holonomy element. Composition and inverse are defined in the same way as for the fundamental groupoid. The space of units is $V$; if $V$ is of dimension $n$ and the foliation of codimension $q$ then $\Group$ is a differentiable groupoid of dimension $2n-q$. It is not Hausdorff in general. If $T$ is a transversal that meets all leaves of the foliation, then the restriction of the holonomy groupoid to $T$ is an \'etale groupoid equivalent to $\Group$.

\end{enumerate}
\end{Ex}

From now on let $\Group$ be a locally compact Hausdorff groupoid. 
An action (on the right) of $\Group$ on a space $Z$ is given by a map $p:Z\to \Group^{0}$, called the source map, and a continuous map from $Z\times_{\Group^{0}}\Group=\{(z,g)\,|\,p(z)=r(g)\}$ to $Z$, denoted by $(z,g)\mapsto zg$, such that $(zg)h=z(gh)$ whenever $p(z)=r(g)$ and $s(g)=r(h)$ and $zp(z)=z$. A space endowed with an action of $\Group$ is called a $\Group$-space. \\
We can then define a groupoid denoted by $Z\rtimes \Group$ with underlying set $Z\times\Group$, unit space $Z\simeq\{(z,p(z))\,|\,z\in Z\}$, source and range maps $s(z,g)=zg$, $r(z,g)=z$, inverse $(z,g)^{-1}=(zg,g^{-1})$ and products $(z,g)(zg,h)=(z,gh)$. Note that $Z\rtimes \Group$ is \'etale if $\Group$ is. If $Z$  and $\Group$ are locally compact Hausdorff, the action of $\Group$ on $Z$ is free (resp. proper) if and only if the groupoid $Z\rtimes \Group$ is principal (resp. proper). A $\Group$-space $Z$ is said to be $\Group$-compact if the action is proper and the quotient $Z/\Group$ is compact. \\

A $\Group$-algebra is an algebra $A$ endowed with an action of $\Group$ i.e $A$ is a $C(\Group^{0})$-algebra and the action of $\Group$ on $A$ is given by an isomorphism of $C(\Group)$-algebras $\alpha:s^*A\to r^*A$ such that the morphisms $\alpha_g:A_{s(g)}\to A_{r(g)}$ satisfy the relation $\alpha_g\circ\alpha_h=\alpha_{gh}$. Recall that if $X$ is a locally compact Hausdorff space, a $C(X)$-algebra is a $C^*$-algebra endowed with a $*$-homomorphism $\theta$ from $C_0(X)$ to the center $Z(M(A))$ of the multiplier algebra of $A$, such that $\theta(C_0(X))A=A$. If $p:X\to Y$ is a map between two locally compact Hausdorff spaces and $A$ is a $C(X)$-algebra, then $p^*A=A\otimes_{C_0(X)} C_0(Y)$ is a $C(Y)$-algebra. If $x\in X$, the fiber $A_x$ of $A$ over $x$ is defined by $i^* A$ where $i_x:\{x\}\to X$ is the inclusion map. \\

Suppose $\Group$ is $\sigma$-compact and has a Haar system $\lambda=\{\lambda_x\,|\, x\in \Group^{0}\}$ (we can take for example $\Group$ to be \'etale and then $\lambda_x$ is the counting measure on $\Group^{x}$). A \emph{cutoff function} on $\Group$ is a continuous function $c:\Group^0\to \mathbb{R}^+$ such that for every $x\in\Group^0$, $\int_{g\in\Group^x}c(s(g))d\lambda^x(g)=1$, and for every compact $K\subset \Group^0$, $\mathrm{supp}(c)\cap s(G^K)$ is compact. Such a function exists if and only if  $\Group$ is proper \cite[Propositions 6.10, 6.11]{Tu99}. \\

Let $A$ be a $\Group$-algebra. The full and reduced crossed-products of $A$ by $\Group$, denoted $C^*(\Group,A)$ and $C^*_r(\Group,A)$ respectively are defined in the following way : let $C_c(\Group, r^*A)$ be the space of functions with compact support $g\mapsto \varphi(g)\in A_{r(g)}$ continuous in the sense of \cite{LeGall99}. The product and adjoint are defined respectively by 
\begin{align*}
&\varphi*\psi(g)=\int\limits_{h\in \Group^{r(g)} }\varphi(h)\alpha_h(\psi(h^{-1}g))d\lambda^{r(g)}(h),\\
&\varphi^*(g)=\alpha_g(\varphi(g^{-1}))^*
\end{align*}
Then,  $L^1(\Group,r^*A)$ denotes the completion of $C_c(\Group,r^*A)$ for the norm $$\|\varphi\|=\max(|\varphi|_1,|\varphi^*|_1),$$ where $|\varphi |_1=\sup\limits_{x\in\Group^0}\int_{g\in\Group^x}\|\varphi(g)\|d\lambda^{x}(g)$ and $C^*(\Group,A)$ is the enveloping $C^*$-algebra of $L^1(\Group,r^*A)$ and $C^*_r(\Group,A)$ is the closure of $L^1(\Group,r^*A)$ in $\mathcal{L}(L^2(\Group,r^*A))$.\\

When the $\Group$-algebra $A$ is  the algebra $C_0(\Group^{0})$ of continuous functions vanishing at infinity on the space of objects $\Group^{0}$, the crossed products $C^*(\Group,A)$ and $C^*_r(\Group,A)$ will simply be denoted $C^*(\Group)$ and $C^*_r(\Group)$, and called groupoid full and reduced $C^*$-algebras.

In \cite{LeGallCRAS, LeGall99}, for every pair $(A,B)$ of graded $\Group$-algebras, Le Gall defined a bifunctor $KK_{\Group}(A,B)$ generalizing Kasparov's $KK$-bifuntor for groups (see section \ref{bifunctor}) that has mostly the same features, in particular, there is an associative product $KK_{\Group}(A,D)\times KK_{\Group}(D,B)\to KK_{\Group}(A,B)$ that satisfies the same naturality properties as in case of the non-equivariant $KK$-functor. The product of two elements $\alpha\in KK_\Group(A,D)$, $\beta\in KK_\Group(D,B)$ is denoted by $\alpha\otimes_D\beta$. And there are descent morphisms $$j_{\Group}:KK_{\Group}(A,B)\to KK(C^*(\Group,A), C^*(\Group,B)),$$ $$j_{\Group,r}:KK_{\Group}(A,B)\to KK(C^*_r(\Group,A), C^*_r(\Group,B)),$$
compatible with the product.\\

Suppose that $\Group$ is proper and that $\Group^0/\Group$ is compact and let $c$ be a cutoff function for $\Group$. The function $g\mapsto \sqrt{c(r(g))c(s(g))}$, which is continuous with compact support, defines a projection in $C^*(\Group)=C^*_r(\Group)$ whose homotopy class is independent of the choice of the cutoff function and hence defines a canonical element $\lambda_{\Group}\in K_0(C^*(\Group))$. If $Z$ is a $\Group$-compact proper space and $B$ is a $\Group$-algebra, the map

$$\xymatrix{ KK^*(C_0(Z),B)\ar[r]^-{j_{\Group,r}} &KK^*(C^*(Z\rtimes\Group),C^*_r(\Group,B))\ar[rr]^-{\scriptstyle\lambda_{Z\rtimes G}\otimes .} & &K_*(C^*_r(\Group,B))}$$

induces the Baum-Connes map with coefficients
$$\mu^B_r: K^{\mathrm{top}}_*(\Group;B)=\lim \limits_{\rightarrow}KK_{\Group}^*(C_0(Z),B)\to K_*(C^*_r(\Group,B)),$$
where the inductive limit is taken among all the $Z$ subspace of $\underline{E}\Group$ that are $\Group$-compact and $\underline{E}\Group$ is the classifying space for proper actions of $\Group$. As shown in  \cite{Tu99}, one can take $\underline{E}\Group$ to be the se of positive measures $\mu$ on $\Group$ such that $s_*\mu$ is a Dirac measure on $\Group^0$ and $|\mu |\in(\frac{1}{2},1]$.\\
The Baum-Connes conjecture with coefficients for groupoids can be stated as follows 
\begin{Conj}
For every locally compact Haussdorf groupoid with Haar system $\Group$ and every $\Group$-algebra, $\mu^B_r(\Group)$ is an isomorphism.
\end{Conj}
When $B=C_0(\Group^0)$, we get the Baum-Connes map without coefficients : 
$$\mu_r: K^{\mathrm{top}}_*(\Group)=K^{\mathrm{top}}_*(\Group;C_0(\Group))=\lim KK_{\Group}^*(C_0(Z),B)\to K_*(C^*_r(\Group)),$$
And the conjecture without coefficients states that $\mu_r(\Group)$ is an isomorphism.\\

Tu's general definition of the dual Dirac method as discussed in section \ref{viveTu} is stated in terms of groupoids as follows. Let $\Group$ be a locally compact, $\sigma$-compact groupoid with Haar system. Suppose there exists a proper $\Group$-algebra $A$ and elements $$\eta\in KK_{\Group}(C_0(\Group^0,A),\quad D\in KK_{\Group}(A,C_0(\Group^0)),$$
$$\gamma\in KK_\Group(C_0(\Group^0),\Group^0))$$
such that $\eta\otimes_A D=\gamma$ and $p^*\gamma=1\in KK_{\underline{E}\Group\rtimes\Group}(C_0(\underline{E}\Group),C_0(\underline{E}\Group))$, where $p:\underline{E}\Group\to\Group^0$ is the source map for the action of $\Group$ on $\underline{E}\Group$. Then this element is unique and $\Group$ is said to have a $\gamma$-element. It is the same element as the one constructed by Kasparov for every connected locally compact group \cite{KaspConsp} (see section \ref{gamma}). Tu's result is stated as follows,

\begin{Thm}[{\cite[Proposition 5.23]{Tu99},\cite[Theorem 2.2]{Tu99Trees}}]\label{Tugamma}
If the groupoid $\Group$ has a $\gamma$-element, then the Baum-Connes maps with coefficients $\mu$ and $\mu_r$ are split injective. Moreover, if $\gamma=1$ in $KK_{\Group}(C_0(\Group^0),C_0(\Group^0))$, then $\mu$ and $\mu_r$ with coefficients are isomorphisms and $\Group$ is K-amenable. 

\end{Thm}
As explained by Tu in \cite{Tu-survey}, proofs of injectivity of $\mu_r$ based in Theorem \ref{Tugamma} are constructive : they require explicit constructions of a proper $C^*$-algebra and the elements in $KK_{\cal G}$ appearing in the definition of a $\gamma$-element; and to do so one uses the existence of an action of the corresponding groupoid on some space with particular geometric properties. \\

Using Theorem \ref{Tugamma} Tu proved that the assembly map $\mu_r$ is injective for bolic foliations (cf. \cite{Tu99}[D\'efinition 1.15]) and that it is a isomorphism for groupoids satisfying the Haagerup property, for example, amenable groupoids (cf. \cite{Tu99Moyennables}). \\

As an example, let us mention that Higson and Roe proved that a discrete group $\Gamma$ has property $A$ if and only if the groupoid $\beta \Gamma\rtimes \Gamma$ is amenable, where $\beta\Gamma$ is the Stone-\v{C}ech compactification of $\Gamma$ (see section \ref{propA} for a discussion on property $A$ and \cite{Higson-Roe}). \\
Higson also proved that if $\Gamma$ has property $A$, then the Baum-Connes map with coefficients $\mu_r$ for $\Gamma$ is injective and $C^*_r(\Gamma)$ is an exact $C^*$-algebra (\cite{Higson2000})\\

On the other hand, Skandalis, Tu and Yu proved in \cite{Skandalis-Tu-Yu} that $\Gamma$ can be coarsely embedded into a Hilbert space if and only if $\beta \Gamma\rtimes \Gamma$ has Haagerup property. If this is the case, then the Baum-Connes map with coefficients for $\Gamma$ is injective. \\

We mention here that there is also a Banach version of the dual Dirac technique for groupoids developped by Lafforgue in \cite{Lafforguegroupoides}. He defined a KK-theory for Banach algebras that is equivariant with respect to the action of a groupoid and he used a notion of unconditionnal completion that he established in this context to prove the Baum-Connes conjecture with commutative coefficients for hyperbolic groups.

\subsection{Counterexamples for groupoids}\label{countergroupoids}

This section is based on sections 1 and 2 of \cite{HLS}. Let $\Group$ be a locally compact, Hausdorff groupoid. Say that a closed subset $F$ of the unit space $\Group^0$ is {\it saturated} if every morphism with source in $F$ has also range in $F$. Set $U=\Group\backslash F$. Let $\Group_F$ be the groupoid obtained by restricting $\Group$ to $F$, and let $\Group_U$ be the open subgroupoid of $\Group$ comprising those morphisms with source and range in $U$. Then there is a short exact sequence at the level of maximal $C^*$-algebras:
$$0\rightarrow C^*_{\rm max}(\Group_U)\rightarrow C^*_{\rm max}(\Group)\rightarrow C^*_{\rm max}(\Group_F)\rightarrow 0,$$
but the corresponding sequence at the level of {\it reduced} $C^*$-algebras
$$0\rightarrow C^*_r(\Group_U)\rightarrow C^*_r(\Group)\rightarrow C^*_r(\Group_F)\rightarrow 0$$
may fail to be exact; in favorable circumstances this lack of exactness can even be detected at the level of K-theory. This can be exploited to produce counter-examples to the Baum-Connes conjecture.

\begin{Lem}\label{notonto} Assume that the sequence
\begin{equation}\label{notexact}
K_0(C^*_r(\Group_U))\rightarrow K_0(C^*_r(\Group))\rightarrow K_0(C^*_r(\Group_F))
\end{equation}
is NOT exact in the middle term. If the assembly map $K^{top}_0(\Group_F)\rightarrow K_0(C^*_r(\Group_F))$ is injective, then the assembly map $K^{top}_0(\Group)\rightarrow K_0(C^*_r(\Group))$ is NOT surjective.
\end{Lem}

\begin{proof}[Proof] By contrapositive, we assume that $K^{top}_0(\Group)\rightarrow K_0(C^*_r(\Group))$ is surjective, and prove that the sequence \ref{notexact} is exact. For this we chase around the commutative diagram:

$$\xymatrix{   &                                                           &K_0^{\mathrm{top}}(\Group)\ar[r]\ar[d]           &K_0^{\mathrm{top}}(\Group_F)\ar[d]                  &\\
0\ar[r]&K_0(C^*_{\mathrm{max}}(\Group_U))\ar[r]\ar[d]&K_0(C^*_{\mathrm{max}}(\Group))\ar[r]\ar[d]&K_0(C^*_{\mathrm{max}}(\Group_F))\ar[r]\ar[d]&0\\
0\ar[r]&K_0(C^*_{r}(\Group_U))\ar[r]&K_0(C^*_{r}(\Group))\ar[r]&K_0(C^*_{r}(\Group_F)\ar[r]&0
}$$

Let $y$ be in the kernel of $K_0(C^*_r(\Group))\rightarrow K_0(C^*_r(\Group_F))$. By the assumed surjectivity of the assembly map for $\Group$, we write $y$ as the image of $x\in K^{top}_0(\Group)$. Then the image of $x$ in $K^{top}_0(\Group_F)$ is zero, by the assumed injectivity of the assembly map for $\Group_F$. So $\mu_{\rm max}(x)$ is in the kernel of $K_0(C^*_{\rm max}(\Group))\rightarrow K_0(C^*_{\rm max}(\Group_F))$ and therefore in the image of $K_0(C^*_{\rm max}(\Group_U))$, by exactness of the middle row. So $y=\mu_r(x)$ is in the image of $K_0(C^*_r(\Group_U))$.
\end{proof}

Let us give a simple example where this happens. 

\begin{Def}\label{resfin} A group $\Gamma$ is {\it residually finite} if $\Gamma$ admits a {\it filtration}, i.e. a decreasing sequence $(N_k)_{k>0}$ of finite index normal subgroups with trivial intersection. 
\end{Def}

We recall that finitely generated linear groups are residually finite, which provides a wealth of examples. If $(N_k)_{k>0}$ is a filtration of $\Gamma$, we denote by $\lambda_{\Gamma/N_k}$ the representation of $\Gamma$ obtained by composing the regular representation of $\Gamma/N_k$ with the quotient map $\Gamma\rightarrow\Gamma/N_k$, and by $\lambda^0_{\Gamma/N_k}$ the restriction of $\lambda_{\Gamma/N_k}$ to the orthogonal of constants.

\begin{Def}\label{tau} If $(N_k)_{k>0}$ is a filtration of $\Gamma$, the group $\Gamma$ has {\it property $(\tau)$ with respect to the filtration $(N_k)_{k>0}$} if the representation $\oplus_{k>0}\lambda^0_{\Gamma/N_k}$ does not almost admit invariant vectors.
\end{Def}

It follows from Proposition \ref{almostinv} that a residually finite group with property (T) has property $(\tau)$ with respect to every filtration. For a group like the free group, this property depends crucially on the choice of a filtration.

\medskip
Fix now a filtration $(N_k)_{k>0}$ in the residually finite\footnote{Until the end of Proposition \ref{counter}, we denote a countable group by $\Gamma_\infty$ rather than $\Gamma$, as we view $\Gamma_\infty$ as the limit of its finite quotients $\Gamma_k$.} group $\Gamma_\infty$, let $q_k:\Gamma_\infty \rightarrow \Gamma_k=\Gamma_\infty /N_k$ be the quotient homomorphism. Let $\overline{\N}=\N\cup\{\infty\}$ be the one-point compactification of $\N$, endow $\overline{\N}\times\Gamma_\infty$ with the following equivalence relation: 
$$(m,g)\sim(n,h)\Leftrightarrow\left\{\begin{array}{cccc}either & m=n=\infty & and & g=h \\or & m=n\in\N & and & q_m(g)=q_m(h)\end{array}\right.$$
Let $\Group$ be the groupoid with set of objects $\Group^0=\overline{\N}$, and with set of morphisms $\Group^1=(\overline{\N}\times \Gamma_\infty)/\sim$, with the quotient topology; observe that $\Group$ is a Hausdorff groupoid, as $(N_k)_{k>0}$ is a filtration. We may view $\Group$ as a continuous field of groups over $\overline{\N}$, with $\Gamma_k$ sitting over $k\in\overline{\N}$. Set $F=\{\infty\}$ and $U=\N$.

\begin{Prop}\label{counter} Let $\Gamma_\infty$ be an infinite, discrete subgroup of $SL_n(\R)$. Assume that there exists a filtration $(N_k)_{k>0}$ such that $\Gamma_\infty$ has property $(\tau)$ with respect to it. Let $\Group$ be the groupoid construct above, associated with this filtration. The the assembly map for $\Group$ is not surjective.
\end{Prop}

\begin{proof}[Proof] We check the two assumptions of Lemma \ref{notonto}. First, $\Group_F=\Gamma_\infty$. As the assembly map $\mu_r$ is injective for every closed subgroup of any connected Lie group (e.g. $SL_n(\R)$), it is injective for $\Group_F$. It remains to see that the sequence (\ref{notexact}) is not exact in our case. For the representation $\pi=\oplus_{k>0}\lambda_{\Gamma_\infty/N_k}$ of $\Gamma_\infty$, denote by $C^*_\pi(\Gamma_\infty)$ the completion of $\C\Gamma_\infty$ defined by $\pi$. Because of property $(\tau)$ there exists a Kazhdan projection $e_\pi\in C^*_\pi(\Gamma_\infty)$ that projects on the $\Gamma_\infty$-invariant vectors\footnote{If $\Gamma_\infty$ has property (T), $e_\pi$ is the image in $C^*_\pi(\Gamma_\infty)$ of the Kazhdan projection $e_\Group\in C^*_{\rm max}(\Gamma_\infty)$ from Proposition \ref{Kazhdanproj}.} in every representation of $C^*_\pi(\Gamma_\infty)$.

Now $C^*_r(\Group)$ is the completion of $C_c(\Group^1)$ for the norm 
$$\|f\|=\sup_{k\in\overline{\N}}\|\lambda_{\Gamma_\infty/N_k}(f_k)\|,$$
where $f\in C_c(\Group^1)$ and $f_k=f|_{\{k\}\times\Gamma_k}$. 

Consider the homomorphism 
$\C\Gamma_\infty\rightarrow C_c(\Group^1)$ 
which to $g\in\Gamma_{\infty}$ associates the characteristic function of the set of $(k,h)\in \bar\N\times\Gamma_{\infty}$ such that $h=q_k(g)$. It 
extends to a homomorphism $\alpha:C^*_\pi(\Gamma_\infty)\rightarrow C^*_r(\Group)$, as is easily checked. The projection $\alpha(e_\pi)$ is in the kernel of the map $C^*_r(\Group)\rightarrow C^*_r(\Group_F)$: as $\Gamma_\infty$ is infinite, its regular representation has no non-zero invariant vector. Therefore the class $[\alpha(e_\pi)]\in K_0(C^*_r(\Group))$ is in the kernel of the map $K_0(C^*_r(\Group))\rightarrow K_0(C^*_r(\Group_F))$.

On the other hand $\Group_U=\coprod_{k>0}(\Gamma_\infty/N_k)$, so $C^*_r(\Group_U)=\oplus_{k>0}C^*(\Gamma_\infty/N_k)$ (a $C^*$-direct sum) and $K_0(C^*_r(\Group_U))=\oplus_{k>0} K_0(C^*(\Gamma_\infty/N_k))$ (an algebraic direct sum). Considering now the natural homomorphism $\lambda_{\Gamma_\infty/N_k}: C^*_r(\Group_U)\rightarrow C^*(\Gamma_\infty/N_k)$, we see in this way that $(\lambda_{\Gamma_\infty/N_k})_*(x)\neq 0$ for only finitely many $k$'s if $x$ lies in the image of $K_0(C^*_r(\Group_U))$ in $K_0(C^*_r(\Group))$, while $(\lambda_{\Gamma_\infty/N_k})_*[\alpha(e_\pi)]\neq 0$ for every $k\in\N$. This shows that $[\alpha(e_\pi)]$ is not in the image of $K_0(C^*_r(\Group_U))$.
\end{proof}

\begin{Ex}\label{T+tau} Explicit examples where Proposition \ref{counter} applies, are $SL_n(\Z)$ with $n\geq 3$ and any filtration (because of property (T)), and $SL_2(\Z)$ with a filtration by congruence subgroups (property $(\tau)$ is established in \cite{Lub}).
\end{Ex}

The paper \cite{HLS} by Higson-Lafforgue-Skandalis contains several other counter-examples to the Baum-Connes conjecture for groupoids:
\begin{itemize}
\item injectivity counter-examples for Hausdorff groupoids;
\item injectivity counter-examples for (non-Hausdorff) holonomy groupoids of foliations;
\item surjectivity counteramples for semi-direct product groupoids $Z\rtimes \Gamma$, where $Z$ is a suitable locally compact space carrying an action of a Gromov monster $\Gamma$ (see section \ref{coarseHilbert} below for more on Gromov monsters). In terms of $C^*$-algebras, since $C^*_r(Z\rtimes\Gamma)=C^*_r(\Gamma, C_0(Z))$, this is a counter-example for the Baum-Connes conjecture with coefficients (conjecture \ref{ConjBCcoeff}).
\end{itemize}

\section{The coarse Baum-Connes conjecture (CBC)}\label{Coarse}

\hfill{\it We dedicate this section to the memory of John Roe (1959-2018)}

\medskip
The idea behind coarse, or large scale-geometry is very simple: ignore the local, small-scale features of a geometric space and concentrate on its large-scale, or long term, structure. By doing so, trends or qualities may become apparent which are obscured by small-scale irregularities. For a metric space $X$, the {\it coarse Baum-Connes conjecture} postulates an isomorphism
$$\mu_X: KX_*(X)=\lim_{d\rightarrow\infty} K_*(P_d(X))\stackrel{\simeq}{\longrightarrow} K_*(C^*(X)), $$
where the actors only depend on large scale, or coarse structure of $X$. The right-hand side is the K-theory of a certain $C^*$-algebra, the {\it Roe algebra of $X$} - a non-commutative object; while the left-hand side is the limit of the K-homology groups of certain metric spaces (i.e. commutative objects), namely Rips complexes of $X$, see Definition \ref{Rips}; and the isomorphism should be given by a concrete map, the {\it coarse assembly map $\mu_X$}. This way the analogy with the classical Baum-Connes conjecture (Conjecture \ref{ConjBC}) becomes apparent: both are in the spirit of bridging non-commutative geometry with classical topology and geometry. CBC has several applications, e.g. the Novikov conjecture (Conjecture \ref{Novikov}) when $X=\Gamma$, a finitely generated group equipped with a word metric.

Let $(X,d_X), (Y,d_Y)$ be metric spaces, and $f:X\rightarrow Y$ a map (not necessarily continuous). We say that $f$ is {\it almost surjective} if there exists $C>0$ such that $Y$ is the $C$-neighborhood of $f(X)$. Recall that $f$ is a {\it quasi-isometric embedding} if there exists $A>0$ such that 
$$\frac{1}{A}d_X(x,x')-A\leq d_Y(f(x),f(x'))\leq Ad_X(x,x')+ A,$$
for every $x,x'\in X$, and that $f$ is a {\it quasi-isometry} if $f$ is a quasi-isometric embedding which is almost surjective. A weaker condition is provided by coarse embeddings, relevant for large-scale structure and corresponding to injections in the coarse category: $f$ is a {\it coarse embedding} if there exist functions $\rho_+,\rho_-:\R^+\rightarrow\R^+$ (called control functions) such that $\lim_{t\rightarrow\infty}\rho_\pm(t)=\infty$ and
$$\rho_-(d_X(x,x'))\leq d_Y(f(x),f(x'))\leq \rho_+(d_X(x,x'))$$
for every $x,x'\in X$. Finally, $f$ is a {\it coarse equivalence} if $f$ is a coarse embedding which is almost surjective; coarse equivalences are isomorphisms in the coarse category.

\subsection{Roe algebras}

\subsubsection{Locality conditions on operators}

Let $(X,d_X)$ be a proper metric space. A {\it standard module over $C_0(X)$} is a Hilbert space $\HH_X$ carrying a faithful representation of $C_0(X)$, whose image meets the compact operators only in 0. Fix a bounded operator $T$ on $\HH_X$. A point $(x,x')\in X\times X$ is in {\it the complement of the support of $T$} if there exists $f,f'\in C_c(X)$, with $f(x)\neq 0\neq f'(x')$ and $f'Tf=0$. 

Say that $T$ is {\it pseudo-local} if the commutator $[T,f]$ is compact for every $f\in C_0(X)$, that $T$ is {\it locally compact} if $Tf$ and $fT$ are compact operators for every $f,f'\in C_0(X)$. Say that $T$ has {\it finite propagation} if the support of $T$ is contained in a neighborhood of the diagonal in $X$ of the form $\{(x,x')\in X\times X: d_X(x,x')\leq R\}.$

\begin{Def}\label{Roealgebra} The Roe algebra $C^*(X)$ is the norm closure of the set of locally compact operators with finite propagation on $\HH_X$.
\end{Def}

It can be shown that $C^*(X)$ does not depend on the choice of the standard module $\HH_X$ over $C_0(X)$. The K-theory $K_*(C^*(X))$ will be the right-hand side of the CBC.

\begin{Ex} If $X$ is a uniformly discrete metric space (i.e. the distance between two distinct points is bounded below by some positive number), then we may take $\HH_X=\ell^2(X)\otimes\ell^2(\N)$, any operator $T\in\mathcal{B}(\HH_X)$ can be viewed as a matrix $T=(T_{xy})_{x,y\in X}$. Then $T$ is locally compact if and only if $T_{xy}$ is compact for every $x,y\in X$, and $T$ has finite propagation if and only if there is $R>0$ such that $T_{xy}=0$ for $d(x,y)>R$. In particular $\ell^\infty(X,\mathcal{K})$, acting diagonally on $\HH_X$, is contained in $C^*(X)$.
\end{Ex}

\begin{Ex}\label{Roeforgroups} Let $\Gamma$ be a finitely generated group, endowed with the word metric $d(x,y)=|x^{-1}y|_S$ associated with some finite generating set $S$ of $G$. Let $|\Gamma|$ denote the underlying metric space, which is clearly uniformly discrete. Let $\rho$ be the right regular representation of $G$ on $\ell^2(G)$; observe that, because  $d(xg,x)=|g|_S$, the operator $\rho(g)\otimes 1$ has finite propagation. Actually the Roe algebra in this case is $C^*(|\Gamma|)=\ell^\infty(\Gamma,\mathcal{K})\rtimes_r \Gamma$, where $\Gamma$ acts via $\rho$.
\end{Ex}

\subsubsection{Paschke duality and the index map}

Let $X$ be a proper metric space and $\HH_X$ a standard module over $C_0(X)$, as in the previous paragraph. Denote by $\Psi_0(X,\HH_X)$ the set of pseudo-local operators, and by $\Psi_{-1}(X,\HH_X)$ the set of locally compact operators. It follows from the definitions that $\Psi_0(X,\HH_X)$ is a $C^*$-algebra containing $\Psi_{-1}(X,\HH_X)$ as a closed 2-sided ideal. 

The K-homology of $X$ may be related to the K-theory of the quotient 
$$\Psi_0(X,\HH_X)/\Psi_{-1}(X,\HH_X).$$
For $i=0,1$ there are maps 
\begin{equation}\label{paschke}
K_i(\Psi_0(X,\HH_X)/\Psi_{-1}(X,\HH_X))\rightarrow K_{1-i}(X)
\end{equation}
defined as follows. For $i=0$, let $p$ be a projection in $\Psi_0(X,\HH_X)/\Psi_{-1}(X,\HH_X)$ (or in a matrix algebra over $\Psi_0(X,\HH_X)/\Psi_{-1}(X,\HH_X)$), form the self-adjoint involution $f=2p-1$, let $F$ be a self-adjoint lift of $f$ in $\Psi_0(X,\HH_X)$. Then the pair $(\HH_X,F)$ is an odd Fredholm module over $C_0(X)$, in the sense of Definition \ref{Fredmodule}, so it defines an element of the K-homology $K_1(X)$. For $i=1$, let $u$ be a unitary in $\Psi_0(X,\HH_X)/\Psi_{-1}(X,\HH_X)$ (or in a matrix algebra over it), let $U$ be a lift of $u$ in $\Psi_0(X,\HH_X)$, form the self-adjoint operator
$$F=\left(\begin{array}{cc}0 & U \\U & 0\end{array}\right)$$
on $\HH_X\oplus\HH_X$: then $(\HH_X\oplus\HH_X,F)$ is an even Fredholm module over $C_0(X)$, defining an element of the K-homology $K_0(X)$. Paschke \cite{Paschke} proved that, when $\HH_X$ is a standard module, the homomorphisms in \ref{paschke} are isomorphisms: this is Paschke duality.

Now define $D^*(X,\HH_X)$ as the norm closure of the pseudo-local, finite propagation operators. It is clear that $C^*(X)$ is a closed 2-sided ideal in $D^*(X,\HH_X)$. It was proved by Higson and Roe (see \cite{HigsonRoeOberwolfach}, lemma 6.2), that the inclusion $D^*(X,\HH_X)\subset\Psi_0(X,\HH_X)$ induces an isomorphism $D^*(X,\HH_X)/C^*(X)\simeq \Psi_0(X,\HH_X)/\Psi_{-1}(X,\HH_X)$ of quotient $C^*$-algebras. Now consider the 6-terms exact sequence in K-theory associated with the short exact sequence
$$0\rightarrow C^*(X)\rightarrow D^*(X,\HH_X)\rightarrow D^*(X,\HH_X)/C^*(X)\rightarrow 0;$$
the connecting maps $K_{1-i}(D^*(X,\HH_X)/C^*(X))\rightarrow K_i(C^*(X)\;(i=0,1)$ can be seen as maps $K_{1-i}(\Psi_0(X,\HH_X)/\Psi_{-1}(X,\HH_X))\rightarrow K_i(C^*(X))$. Applying Paschke duality, we get an {\it index map}
$$Ind_X:K_*(X)\rightarrow K_*(C^*(X)),$$
for every proper metric space $X$.

\begin{Ex} If $X$ is compact, then $C^*(X)$ is the $C^*$-algebra of compact operators, so $K_0(C^*(X))=\Z$ and the map $Ind_X: K_0(X)\rightarrow \Z$ is the usual index map that associates its Fredholm index to an even Fredholm module over $C(X)$.
\end{Ex}

\subsection{Coarse assembly map and Rips complex}

\subsubsection{The Rips complex and its K-homology}

We now define the left-hand side of the assembly map, in terms of Rips complexes. Recall from Definition \ref{Rips} that, for $X$ a locally finite metric space (i.e. every ball in $X$ is finite) and $d\geq 0$, the {\it Rips complex} $P_d(X)$ is the simplicial complex with vertex set $X$, such that a subset $F$ with $(n+1)$-elements spans a $n$-simplex if and only if $diam(F)\leq d$. We define a metric on $P_d(X)$ by taking the maximal metric that restricts to the spherical metric on every $n$-simplex - the latter being obtained by viewing the $n$-simplex as the intersection of the unit sphere $S^n$ with the positive octant in $\R^{n+1}$.

The {\it coarse K-homology} of $X$ is then defined as:
$$KX_*(X):=\lim_{d\rightarrow \infty}K_*(P_d(X));$$
this will be the left hand side of the CBC. Observe that, for every $d\geq 0$, the spaces $X$ and $P_d(X)$ are coarsely equivalent. Then, taking K-theory, we see that $\lim_{d\rightarrow\infty}K_*(C^*(P_d(X)))$ is isomorphic to $K_*(C^*(X)).$ 
\begin{Ex} If $\Gamma$ is a finitely generated group and $X=|\Gamma|$, then $KX_*(X)=\lim_Y K_*(Y)$ where $Y$ runs in the directed set of closed, $\Gamma$-compact subsets of the classifying space for proper actions $\underline{E\Gamma}$. This is to say that CBC can really be seen as a non-equivariant version of the Baum-Connes conjecture \ref{ConjBC}.
\end{Ex}

\subsubsection{Statement of the CBC}\label{StatementCBC}

The index map $Ind_{P_d(X)}$ is compatible with the maps $K_*(P_d(X))\rightarrow K_*(P_{d'}(X))$ and $K_*(C^*(P_d(X)))\rightarrow K_*(C^*(P_{d'}(X)))$ induced by the inclusion $P_d(X)\rightarrow P_{d'}(X)$ for $d<d'$. Passing to the limit for $d\rightarrow\infty$, we get the {\it coarse assembly map}
$$\mu_X:KX_*(X)\rightarrow K_*(C^*(X)).$$
Say that $X$ has {\it bounded geometry} if, for every $R>0$, the cardinality of balls of radius $R$ is uniformly bounded over $X$. Here is now the statement of the {\it coarse Baum-Connes conjecture}.

\begin{Conj}\label{CBC}(CBC) For every space $X$ with bounded geometry, the coarse assembly map $\mu_X$ is an isomorphism.
\end{Conj}

\subsubsection{Relation to the Baum-Connes conjecture for groupoids}\label{coarsegroupoid}

It is a result of G. Yu \cite{Yu1995} that, if $\Gamma$ is a finitely generated group, the CBC for the metric space $|\Gamma|$ is the usual Baum-Connes conjecture for $\Gamma$ with coefficients in the $C^*$-algebra $\ell^\infty(\Gamma,\mathcal{K})$ (compare with example \ref{Roeforgroups}). Skandalis-Tu-Yu \cite{Skandalis-Tu-Yu} generalize this by associating to every discrete metric space $X$ with bounded geometry, a groupoid $\mathcal{G}(X)$ such that the coarse assembly map for $X$ is equivalent to the Baum-Connes assembly map for $\mathcal{G}(X)$ with coefficients in the $C^*$-algebra $\ell^\infty(X,\mathcal{K})$. 

\medskip
Let us explain briefly the groupoid $\mathcal{G}(X)$. So let $X$ be a countable metric space with bounded geometry. A subset $E$ of $X\times X$ is called an \emph{entourage} if $d$ is bounded on $E$, i.e. if there exists $R>0$ such that $\forall (x,y)\in E,\quad d(x,y)\leq R.$\\
Let $$\Group(X)=\bigcup\limits_{E\,\text{entourage}}\overline{E}\subset \beta(X\times X),$$ where $\beta(X\times X)$ is the Stone-\v{C}ech compactification of $X\times X$ and $\overline{E}$ is the closure of $E$ in $\beta(X\times X)$. $\Group(X)$ is the spectrum of the abelian $C^*$-subalgebra of $\ell^{\infty}(X\times X)$ generated by the characteristic functions $\chi_E$ of entourages $E$. Skandalis, Tu and Yu proved that it can be endowed with a structure of groupoid extending the one on $X\times X$. Recall that $X\times X$ is endowed with a structure of groupoid where the source and range are defined by $s(x,y)=y$ and $r(x,y)=x$. These maps extend to maps from $\beta(X\times X)$ to $\beta X$, hence to maps from $\Group(X)$ to $\beta X$ so that $\Group(X)$ is a groupoid whose unit space is $\beta X$ and which is \'etale, locally compact, Hausdorff and principal (cf. \cite{Skandalis-Tu-Yu}[Proposition 3.2]).\\
In the case where $X$ is a finitely generated discrete group $\Gamma$ with a word metric, the groupoid $\Group(X)$ is $\beta \Gamma\rtimes \Gamma$. Skandalis, Tu and Yu proved the following result 

\begin{Thm}[\cite{Skandalis-Tu-Yu}]
Let $X$ be a discrete metric space with bounded geometry. Then $X$ has property $A$(in the sense of Definition \ref{Yu'sA} below) if and only if $\Group(X)$ is amenable. Moreover, $X$ is coarsely embedded into a Hilbert space if and only if $\Group(X)$ has Haagerup property. 
\end{Thm}

The coarse Baum-Connes conjecture can be put inside the framework of the conjecture for groupoids : let $C^*(X)$ be the Roe algebra associated to $(X,d)$, see definition \ref{Roealgebra}. Then $C^*(X)$ is isomorphic to the reduced crossed product
$C^*_r(\Group(X),\ell^{\infty}(X,\mathcal{K}))$ and the coarse assembly map identifies with the Baum-Connes assembly map for the groupoid $\Group(X)$ with coefficients in $\ell^{\infty}(X,\mathcal{K})$.

\subsubsection{The descent principle}

For a finitely generated group $\Gamma$, there is a ``descent principle'' saying that the CBC for $|\Gamma|$ implies the Novikov conjecture for $\Gamma$ (see Theorem 8.4 in \cite{Roe96})

\begin{Thm}\label{descent} Let $\Gamma$ be a finitely generated group. Assume that $\Gamma$ admits a finite complex as a model for its classifying space $B\Gamma$. If CBC holds for the underlying metric space $|\Gamma|$, then the assembly map $\mu_\Gamma$ is injective; in particular the Novikov conjecture (Conjecture \ref{Nov}) holds for $\Gamma$.
\end{Thm}

\subsection{Expanders}

Expanders are families of sparse graphs which are ubiquitous in mathematics, from theoretical computer science to dynamical systems, to coarse geometry.

Let $X=(V,E)$ be a finite, connected, $d$-regular graph. The {\it combinatorial Laplace operator}  on $X$ is the operator $\Delta$ on $\ell^2(V)$ defined by 
$$(\Delta f)(x)=d\cdot f(x)-\sum_{y\in V: y\sim x} f(y)$$
where $f\in\ell^2(V)$ and $\sim$ denotes the adjacency relation on $X$.

It is well known from algebraic graph theory (see e.g. \cite{Lub}, \cite{DavidoffSarnakValette}) that, if $X$ has $n$ vertices, the spectrum of $\Delta$ consists of $n$ eigenvalues (repeated according to multiplicity):
$$0=\lambda_0<\lambda_1 \leq\lambda_2\leq... \leq\lambda_{n-1}\in [0,2d].$$

On the other hand, the {\it Cheeger constant}, or {\it isoperimetric constant} of $X$, is defined as 
$$h(X)=\inf_{A\subset V}\frac{|\partial A|}{\min\{|A|,|V\backslash A|\}}$$
where $\partial A$ is the {\it boundary} of $A$, i.e. the set of edges connecting $A$ with $V\backslash A$. The Cheeger constant measures the difficulty of disconnecting $X$.

The {\it Cheeger-Buser inequality} says that $h(X)$ and $\lambda_1(X)$ essentially measure the same thing:
$$\frac{\lambda_1(X)}{2}\leq h(X)\leq \sqrt{2k\lambda_1(X)}.$$

Expanders are families of large graphs which are simultaneously sparse (i.e. they have few edges, a condition ensured by $d$-regularity, with $d$ fixed) and hard to disconnect (a condition ensured by $h(X)$ being bounded away form 0).

\begin{Def} A family $(X_k)_{k>0}$ of finite, connected, $d$-regular graphs is a family of expanders if $\lim_{k\rightarrow\infty}|V_k|=+\infty$ and there exists $\varepsilon>0$ such that $\lambda_1(X_k)\geq \varepsilon$ for all $k$ (equivalently: there exists $\varepsilon'>0$ such that $h(X_k)\geq\varepsilon'$ for every $k$).
\end{Def}

The tension between sparsity of $X$ and $h(X)$ being bounded away from 0, makes the mere existence of expanders non-trivial. The first explicit construction, using property (T), is due to Margulis:

\begin{Thm}\label{Margulisexp} Let $\Gamma$ be a discrete group with property (T), let $S=S^{-1}$ be a finite, symmetric, generating set of $\Gamma$. Assume that $\Gamma$ admits a sequence of finite index normal subgroups $N_k\triangleleft\Gamma$ with $\lim_{k\rightarrow\infty}[\Gamma:N_k]=+\infty$. Then the sequence of Cayley graphs $(Cay(\Gamma/N_k,S))_{k>0}$ is a family of expanders.
\end{Thm}

\begin{Ex} Take $\Gamma=SL_d(\Z)$, with $d\geq 3$, and $N_k=\Gamma(k)$ the congruence subgroup of level $k$, i.e. the kernel of the map of reduction modulo $k$:
$$\Gamma(k)=\ker(SL_d(\Z)\rightarrow SL_d(\Z/k\Z)).$$
\end{Ex}

Coarse geometry prompts us to view a family $(X_k)_{k>0}$ of finite connected graphs as a single metric space. This is achieved by the {\it coarse disjoint union}: on the disjoint union $\coprod_{k>0} X_k$, consider a metric $d$ such that the restriction of $d$ to each component $X_k$ is the graph metric, and $d(X_k,X_\ell)\geq diam(X_k)+diam(X_\ell)$ for $k\neq\ell$. Such a metric is unique up to coarse equivalence. 

A favorite source of examples comes from {\it box spaces}, that we now define. Let $\Gamma$ be a finitely generated, residually finite group, and let $(N_k)_{k>0}$ be a filtration in the sense of Definition \ref{resfin}. If $S$ is a finite, symmetric, generating set of $\Gamma$, we may form the Cayley graph $Cay(\Gamma/N_k,S)$, as in Theorem \ref{Margulisexp}. 

\begin{Def}\label{box} The coarse disjoint union $\coprod_{k>0} Cay(\Gamma/N_k,S)$ is the box space of $\Gamma$ associated with the filtration $(N_k)_{k>0}$. 
\end{Def}

It is clear that, up to coarse equivalence, it does not depend on the finite generating set $S$, so we simple write $\coprod_{k>0} \Gamma/N_k$. By Theorem \ref{Margulisexp}, any box space of a residually finite group with property (T) is an expander. More generally, it is a result by Lubotzky and Zimmer \cite{LuZi} that $\coprod_{k>0} Cay(\Gamma/N_k,S)$ is a family of expanders if and only if $\Gamma$ has property $(\tau)$ with respect to the filtration $(N_k)_{k>0}$, in the sense of Definition \ref{tau}.

For future reference (see subsection \ref{superexp}), we give one more characterization of expanders:

\begin{Prop}\label{expandersPoincare} Let $(X_k)_{k>0}$ be a sequence of finite, connected, $d$-regular graphs with $\lim_{k\rightarrow\infty}|V_k|=+\infty$. The family $(X_k)_{k>0}$ is a family of expanders if and only if there exists $C>0$ such that, for every map $f$ from $\coprod_{k>0}X_k$ to a Hilbert space $\HH$, the following Poincar\'e inequality holds for every $k>0$:
\begin{equation}\label{Poincare}
\frac{1}{|V_k|^2}\sum_{x,y\in V_k}\|f(x)-f(y)\|^2\leq \frac{C}{|V_k|}\sum_{x\sim y}\|f(x)-f(y)\|^2.
\end{equation}
\end{Prop}

\begin{proof}[Proof:] 
\begin{enumerate}
\item[1)] Let $X=(V,E)$ be a finite connected graph. We first re-interpret the first non-zero eigenvalue $\lambda_1$ of $\Delta$. Consider two quadratic forms on $\ell^2(V)$, both with kernel the constant functions: $\phi\mapsto \frac{1}{|V|^2}\sum_{x,y\in V}|\phi(x)-\phi(y)|^2$ and $\phi\mapsto\frac{1}{|V|}\sum_{x\sim y}|\phi(x)-\phi(y)|^2$. Then $\frac{1}{\lambda_1}$ is the smallest constant $K>0$ such that\footnote{The re-interpretation goes as follows: fix an auxiliary orientation on the edges of $E$, allowing one to define the {\it coboundary operator} $d:\ell^2(V)\rightarrow\ell^2(E):\phi\mapsto d\phi$, where $d\phi(e)=\phi(e^+)-\phi(e^-)$. Observe that $\Delta=d^*d$, so that $\langle\Delta\phi,\phi\rangle=\|d\phi\|^2=\frac{1}{2}\sum_{x\sim y}|\phi(x)-\phi(y)|^2$. By the Rayleigh quotient, $\frac{1}{\lambda_1}$ is the smallest constant $K>0$ such that $\|\phi\|^2\leq K\|d\phi\|^2$ for every $\phi\perp 1$. We leave the rest as an exercise.}
$$\frac{1}{|V|^2}\sum_{x,y\in V}|\phi(x)-\phi(y)|^2\leq \frac{K}{|V|}\sum_{x\sim y}|\phi(x)-\phi(y)|^2$$
for all $\phi\in\ell^2(V)$. 

\item[2)] By the first step, the sequence $(X_k)_{k>0}$ is an expander if and only if there exists a constant $C$ such that, for every function $\phi$ on $\coprod_{k>0} X_k$, we have:
$$\frac{1}{|V_k|^2}\sum_{x,y\in V_k}|\phi(x)-\phi(y)|^2\leq \frac{C}{|V_k|}\sum_{x\sim y}|\phi(x)-\phi(y)|^2.$$

\item[3)] Taking a map $f:\coprod_{k>0}X_k\rightarrow\HH$ and expanding in some orthonormal basis of $\HH$, we immediately deduce inequality (\ref{Poincare}) from the 2nd step.
\end{enumerate}
\end{proof}

\subsection{Overview of CBC}

\subsubsection{Positive results}\label{positiveCBC}

The CBC was formulated by J. Roe in 1993, see \cite{Roe93}. 

\begin{itemize}

\item G. Yu 2000: if a discrete metric space with bounded geometry that admits a coarse embedding into Hilbert space, then CBC holds for $X$, see \cite{YuA};
\item G. Kasparov and G. Yu 2006: if $X$ is a discrete metric space with bounded geometry that coarsely embeds into a super-reflexive Banach space, then the coarse Novikov conjecture (i.e. the injectivity of $\mu_X$) holds for $X$, see \cite{KaYu}.
\end{itemize}

\subsubsection{Negative results}

\begin{itemize}
\item G.Yu 1998: the coarse assembly map is not injective for the coarse disjoint union $\coprod_{n>0} n\cdot S^{2n}$, where $n\cdot S^{2n}$ denotes the sphere of radius $n$ in $(2n+1)$-Euclidean space, with induced metric, see \cite{Yu98}.
\item R. Willett and G. Yu 2012: the coarse assembly map is not surjective for expanders with large girth, see \cite{WiYu1}.
\item N. Higson, V. Lafforgue and G. Skandalis 2001: the coarse assembly map is not surjective for box spaces of residually finite groups $\Gamma$ which happen to be expanders, when $\Gamma$ moreover satisfies injectivity of the assembly map with coefficients, see \cite{HLS}. 

\end{itemize}

Let us describe those counter-examples of Higson-Lafforgue-Skandalis \cite{HLS} more precisely. We first observe (building on lemma \ref{notonto}) that any family of expanders provide a counter-example either to injectivity or to surjectivity of the Baum-Connes assembly map for suitable associated groupoids. To see this, let $(X_k)_{k>0}$ be a family of $d$-regular expanders, and let $X=\coprod_{k>0} X_k$ be their coarse disjoint union. Let $\mathcal{G}(X)$ be the groupoid associated to $X$, as in section \ref{StatementCBC}. Let $F=\beta(X)\backslash X$ be a saturated closed subset in the space of objects, and $U=X$ its complement. 

\begin{Prop}\label{inj/surj} Let $X$ be the coarse disjoint union of a family of $d$-regular expanders. Let $\mathcal{G}(X)$ be the associated groupoid, set $F=\beta(X)\backslash X$. Either the assembly map is not injective for the groupoid $\mathcal{G}(X)_F$ or the coarse assembly map is not surjective for the space $X$. The same holds true for the assembly map with coefficients in $\ell^\infty(X,\mathcal{K})$.
\end{Prop} 

\begin{proof}[Sketch of proof] In view of lemma \ref{notonto}, we must check that
$$K_0(C^*_r(\mathcal{G}(X)_U))\rightarrow K_0(C^*_r(\mathcal{G}(X)))\rightarrow K_0(C^*_r(\mathcal{G}(X)_F))$$
is NOT exact in the middle term. Set $\HH_X=\ell^2(X)\otimes\ell^2(\N)$, fix some rank 1 projection $e\in\mathcal{K}(\ell^2(\N))$ on some unit vector $\xi$, let $\Delta_k$ denote the combinatorial Laplacian on $X_k$, and set $\Delta_X=\oplus_{k>0}(\Delta_k\otimes e)$. Then $\Delta_X$ a locally compact operator with finite propagation on $\HH_X$, as such it defines an element of the Roe algebra $C^*(X)$. The fact that $(X_k)_{k>0}$ is a family of expanders exactly means that 0 is isolated in the spectrum of $\Delta_X$. By functional calculus, the spectral projector $p_X$ associated with $\{0\}$ is also in $C^*(X)$. Now the kernel of $\Delta_k$ on $\ell^2(X_k)$ is spanned by $u_k$, with $u_k=(1,1,...,1)$, so the restriction of $p_X$ to $\ell^2(X_k)\otimes\ell^2(\N)$ is $p_k\otimes(1-e)$, where $p_k$ is the $|V_k|\times |V_k|$-matrix with all entries equal to $\frac{1}{|V_k|}$. In particular entries $(p_X)_{x,y}$ of $p_X$, go to 0 when $d(x,y)\rightarrow\infty$, so $p_X$ is in the kernel of the map $C^*_r(\mathcal{G}(X)))\rightarrow C^*_r(\mathcal{G}(X)_F)$.

It remains to show that the class $[p_X]$ in $K_0(C^*_r(\mathcal{G}(X)))$ does not lie in the image of $K_0(C^*_r(\mathcal{G}(X)_U))$. To see this, first observe that $\mathcal{G}(X)_U$ is the groupoid with space of objects $X$ and exactly one morphism between every two objects. So $C^*_r(\mathcal{G}(X)_U)$ is nothing but $\mathcal{K}(\ell^2(X))$. To proceed, for an operator $T$ with finite propagation on $X$, denote by $T_k$ the restriction of $T$ to $X_k\times X_k$. If $S,T$ are operators with finite propagation then, for $k$ large enough, we have $(ST)_k=S_kT_k$: the reason is that, given $R>0$, for $k\gg 0$ an $R$-neighborhood in $X$ coincides with an $R$-neighborhood in $X_k$, as the $X_k$'s are further and further apart. As a consequence, there exists a homomorphism 
$$C^*(X)\rightarrow (\prod_{k>0} \mathcal{K}(\ell^2(X_k))\otimes\mathcal{K})/(\oplus_{k>0}\mathcal{K}(\ell^2(X_k))\otimes\mathcal{K}),$$
that factors through $C^*(X)/\mathcal{K}(\ell^2(X))$. To conclude, it is enough to show that the image of $[p_X]$ is non-zero in $K_0\big((\prod_{k>0} \mathcal{K}(\ell^2(X_k))\otimes\mathcal{K})/(\oplus_{k>0}\mathcal{K}(\ell^2(X_k))\otimes\mathcal{K})\big)$. For this observe that $p_X$ lifts to a projector $\tilde{p}_X\in \prod_{k>0} \mathcal{K}(\ell^2(X_k))\otimes\mathcal{K}$, and that projections on all factors define a homomorphism $$K_0(\prod_{k>0} \mathcal{K}(\ell^2(X_k))\otimes\mathcal{K})\rightarrow\Z^\N$$ that maps $[\tilde{p}_X]$ to $(1,1,1,...)\in\Z^\N$. Since that homomorphism also maps the group $K_0(\oplus_{k>0} \mathcal{K}(\ell^2(X_k))\otimes\mathcal{K})$ to $\Z^{(\N)}$, we have shown that $[\tilde{p}_X]$ is not in the image of $$K_0(\oplus_{k>0} \mathcal{K}(\ell^2(X_k))\otimes\mathcal{K})\rightarrow K_0(\prod_{k>0} \mathcal{K}(\ell^2(X_k))\otimes\mathcal{K}),$$ so $[p_X]\neq 0$ in $K_0((\prod_{k>0} \mathcal{K}(\ell^2(X_k))\otimes\mathcal{K})/(\oplus_{k>0}\mathcal{K}(\ell^2(X_k))\otimes\mathcal{K}))$. 
\end{proof}

By carefully choosing the family of expanders, we get actual counter-examples to surjectivity in the CBC. For this we need a group $\Gamma$ exactly as in Proposition \ref{counter} (with explicit examples provided by Example \ref{T+tau}), and a box space in the sense of Definition \ref{box}.

\begin{Thm} Let $\Gamma$ be an infinite, discrete subgroup of $SL_n(\R)$, endowed with a filtration $(N_k)_{k>0}$ such that $\Gamma$ has property $(\tau)$ with respect to it. Then the coarse assembly map for the box space $X$ associated with this filtration, is not surjective.
\end{Thm}

\begin{proof}[Proof] Because of property $(\tau)$, the space $X$ is the coarse disjoint union of a family of expanders, and Proposition \ref{inj/surj} will apply. Since by \cite{Skandalis-Tu-Yu} the coarse assemply map for $X$ is the Baum-Connes assembly map for the groupoid $\mathcal{G}(X)$ with coefficients in $\ell^\infty(X,\mathcal{K})$, by lemma \ref{notonto} it is enough to check that the assembly map for the groupoid $\mathcal{G}(X)_F$ is injective with coefficients in $\ell^\infty(X,\mathcal{K})$. Now, because $X$ is a box space, $\mathcal{G}(X)_F$ identifies with the semi-direct product groupoid $(\beta(X)\backslash X)\rtimes\Gamma$. Since $\Gamma$ is a discrete subgroup of $SL_n(\R)$, the assembly map $\mu_{A,r}$ is injective for any coefficient $C^*$-algebra $A$: this proves the desired injectivity, so the coarse assembly map for $X$ is not surjective by Proposition \ref{inj/surj}.
\end{proof}

\subsection{Warped cones}\label{warped}

Warped cones were introduced by J. Roe in 2005, see \cite{Roe05}; he had the intuition that they might lead to counter-examples to CBC. Let $(Y,d_Y)$ be a compact metric space. Let $\Gamma$ be a finitely generated group, with a fixed finite generating set $S$. Assume that $\Gamma$ acts on $Y$ by Lipschitz homeomorphisms, not necessarily preserving $d_Y$. The {\it warped metric} $d_S$ on $Y$ is the largest metric $d_S \leq d_Y$ such that, for every $x\in Y, s\in S$, $d_S(sx,x)\leq 1$. It is given by 
$$d_S(x,y)=\inf\{n+\sum_{i=0}^n d_Y(x_i,y_i): x_0=x, y_n=y, x_i=s_i(y_{i-1}), s_i\in S\cup S^{-1}, n\in\N\}.$$
Intuitively, we modify the metric $d_Y$ by introducing ``group shortcuts'', as two points $x,\gamma x$ will end end at distance $d_S(x,\gamma x)\leq |\gamma|_S$, where $|.|_S$ denotes word length on $\Gamma$.

Form the ``cone'' $Y\times ]1,+\infty[$, with the distance $d$ given by:
$$ d_{Cone}((y_1,t_1),(y_2,t_2))=:|t_1-t_2| + \min\{t_1,t_2\}\cdot d_Y(y_1,y_2).$$
Let $\Gamma$ act trivially on the second factor. The {\it warped cone} $\mathcal{O}_\Gamma Y$ is the cone $Y\times ]1,+\infty[$, with the warped metric obtained from $d_{Cone}$. To get an intuition of what the warped metric does on the level sets $Y \times \{t\}$: assume for a while that $Y$ is a closed Riemannian manifold, fix a $\frac{1}{t}$-net on $Y$, and consider the Voronoi tiling of $Y$ associated to this net (if $y$ is a point in the net, the tile around $y$ is the set of points of $Y$ closer to $y$ that to any other point in the net). Define a graph $X_t$ whose vertices are closed Voronoi tiles, and two tiles $T_1,T_2$ are adjacent if there exists $s\in S\cup S^{-1}\cup\{1\}$ such that $s(\overline{T_1})\cap\overline{T_2}\neq\emptyset$. Then the family of level sets $(Y\times\{t\})_{t>1}$ is uniformly quasi-isometric to the family of graphs $(X_t)_{t>1}$ (i.e. the quasi-isometry constants do not depend on $t$).

In 2015, C. Dru\c tu and P. Nowak \cite{DrNo} made Roe's intuition more precise with the following conjecture. Assume that, on top of the above assumptions, $Y$ carries a $\Gamma$-invariant probability measure $\nu$ such that the action $\Gamma\curvearrowright (Y,\nu)$ is ergodic. Assume that the measure $\nu$ is adapted to the metric $d_Y$ in the sense that $\lim_{r\rightarrow 0} \sup_{y\in Y}\nu(B(y,r))=0$. 

\begin{Conj}\label{DruNow} If the action of $\Gamma$ on $Y$ has a spectral gap (i.e the $\Gamma$-representation on $L^2_0(Y,\
nu)$ does not have almost invariant vectors), then $\mathcal{O}_\Gamma Y$ violates CBC.
\end{Conj}

At the time of writing, warped cones are a hot topic: 
\begin{itemize}
\item P. Nowak and D. Sawicki 2015: warped cones do not embed coarsely into a large class of Banach spaces (those with non-trivial type), containing in particular all $L^p$-spaces ($1\leq p<+\infty$), see \cite{NoSa}.
\item F. Vigolo 2016: relates warped cones and expanders, therefore getting new families of expanders \cite{Vig16}. 
\item D. Sawicki 2017: the level sets $Y\times \{t\}$ of warped cones provide new examples of super-expanders, i.e. expanders not embedding coarsely into any Banach space with non-trivial type, see \cite{Saw17}. 
\item T. de Laat and F. Vigolo 2017: those examples of super-expanders are different (i.e. not coarsely equivalent) to V. Lafforgue's super-expanders, see \cite{deLaVi}.
\item D. Fisher, T. Nguyen and W. Van Limbeek 2017: there is a continuum of coarsely pairwise inequivalent super-expanders obtained from warped cones, see \cite{FNVL}. See subsection \ref{superexp} for super-expanders.
\end{itemize}

In 2017, D. Sawicki \cite{Saw17b} confirmed Roe's intuition by proving the following form of Conjecture \ref{DruNow}.

\begin{Thm} Let $\Gamma$ having Yu's property A. Assume that $\Gamma$ acts on $Y$ by Lipschitz homeomorphisms, freely, and with a spectral gap. Set $A=\{2^n:n\in\N\}\subset ]1,+\infty[$, let $\mathcal{O}'_\Gamma Y$ be the subspace $Y\times A\subset\mathcal{O}_\Gamma Y$, equipped with the warped cone metric. Then $\mu_{CBC}$ is not surjective for $\mathcal{O}'_\Gamma Y$.
\hfill$\square$
\end{Thm}

By looking at actions on Cantor sets, Sawicki is even able to produce counter-examples to CBC which are NOT coarsely equivalent to any family of graphs.

\section{Outreach of the Baum-Connes conjecture}

The Baum-Connes conjecture and the coarse Baum-Connes conjecture prompted a surge of activity at the interface between operator algebras and other fields of mathematics, e.g. geometric group theory and metric geometry. Indeed results like the Higson-Kasparov theorem (see Theorem \ref{HigsonKasparov} above) are of the form "{\it groups (resp. spaces) in a given class satisfy the Baum-Connes (resp. coarse Baum-Connes) conjecture}". This leads naturally to trying to extend the class of groups (resp. spaces) in question, as a way of enlarging the the domain of validity of either conjecture. The study of a class of groups (resp. spaces) has two obvious counterparts: providing new examples, and studying permanence properties of the class. We sketch some of those developments below.

\subsection{The Haagerup property}

The 5-authors book \cite{CCJJV} was the first survey on the subject. Although motivated by Theorem \ref{HigsonKasparov}, it barely mentions the Baum-Connes conjecture and focuses on new examples and stability properties. It was updated in the paper \cite{Valette18}, which can serve as a guide to more recent literature. Here we mention some longstanding open questions on the Haagerup property, and partial results.

\begin{itemize}
\item Let $B_n$ denote the braid group on $n$ strands. Does $B_n$ have the Haagerup property? Yes trivially for $B_2\simeq\Z$, and yes easily for $B_3\simeq\F_2\rtimes\Z$. A recent result by T. Haettel \cite{Haettel} shows that, if the general answer is affirmative, it will not be for a very geometric reason: for $n\geq 4$, the group $B_n$ has no proper, cocompact isometric action on a CAT(0) cube complex\footnote{Recall that a group acting properly isometrically on a CAT(0) cube complex, has the Haagerup property, see e.g. Corollary 1 in \cite{Valette18}.}. Note that a fairly subtle proof of the Baum-Connes conjecture with coefficients for $B_n$, has been given by T. Schick \cite{Schick}.
\item Unlike amenability or property (T), the Haagerup property is {\it not} stable under extensions\footnote{Amenability (resp. property (T)) can be defined by a fixed point property: existence of a fixed point for affine actions on compact convex sets (resp. affine isometric actions on Hilbert spaces). This makes clear that it is preserved under extensions.}. The standard examples to see this are $\Z^2\rtimes SL_2(\Z)$ and $\R^2\rtimes SL_2(\R)$, where the relative property (T) with respect to the non-compact normal subgroup, is an obstruction to the Haagerup property. However the Haagerup property is preserved by some types of semi-direct products: e.g, Cornulier-Stalder-Valette \cite{CorStaVal} proved that, if $\Gamma,\Lambda$ are countable groups with the Haagerup property, then the wreath product $\Gamma\wr\Lambda=(\oplus_\Lambda \Gamma)\rtimes\Lambda $ has the Haagerup property. A probably difficult question: is $G,N$ are locally compact groups with the Haagerup property and $G$ acts continuously on $N$ by automorphisms, under which conditions on the action $G\curvearrowright N$ does the semi-direct product $N\rtimes G$ have the Haagerup property? When $G,N$ are $\sigma$-compact and $N$ is {\it abelian}, the answer was provided by Cornulier-Tessera (Theorem 4 in \cite{CorTes}): $N\rtimes G$ has the Haagerup property if and only if there exists a net $(\mu_i)_{i\in I}$ of Borel probability measures on the Pontryagin dual $\hat{N}$, such that there is a weak-* convergence $\mu_i\rightarrow\delta_1$, and $\mu_i\{1\}=0$ for every $i\in I$, and $\|g\cdot\mu_i-\mu_i\|\rightarrow 0$ uniformly on compact subsets of $G$, and finally the Fourier transform $\widehat{\mu_i}$ is a $C_0$ function on $N$ for every $i\in I$.
\item The behavior of the Haagerup property under central extensions is a widely open question. More precisely: if $Z$ is a closed central subgroup in the locally compact group $G$, is it true that $G$ has the Haagerup property if and only if $G/Z$ has it? Both implications are open. See Proposition 4.2.14 and Section 7.3.3 in \cite{CCJJV} for partial results on lifting the Haagerup property from $G/Z$ to $G$, in particular from $SU(n,1)$ to $\widetilde{SU(n,1)}$.
\item The Haagerup property for discrete groups is stable under free products or more generally amalgamated products over finite groups, by Proposition 6.2.3(1) of \cite{CCJJV}. In general, it is {\it not} true that, if $A,B$ have the Haagerup property and $C$ is a common subgroup, then $A\ast_C B$ has the Haagerup property: see section 4.3.3 in \cite{Valette18} for an example with $C=\Z^2$. An open question concerns the permanence of the Haagerup property for amalgamated products $A\ast_C B$ with $C$ virtually cyclic; a first positive result was obtained recently by M. Carette, D. Wise and D. Woodhouse \cite{CWW}: recall that if a group $G$ acts by isometries on a metric space $(X, d)$, the action of $G$ on $X$ is said to be semisimple if, for every $g\in G$, the infimum $\inf_{x\in X}d(gx,x)$ is actually a minimum. They proved that, if $A,B$ are groups acting properly and semisimply on some real hyperbolic space $\mathbb{H}^n(\R)$, and $C$ is a cyclic subgroup common to $A$ and $B$, then the amalgamated product $A\ast_C B$ has the Haagerup property. 
\end{itemize}

\subsection{Coarse embeddings into Hilbert spaces}\label{coarseHilbert}

In 2000, Guoliang Yu  \cite{YuA} opened a new direction in mathematics by uniting the fields of K-theory for $C^*$-algebras and of metric embeddings into Hilbert space. Indeed he proved that if a metric space $X$ with bounded geometry coarsely embeds into Hilbert space, then $X$ satisfies the CBC. Using the descent principle (Theorem \ref{descent}), this implies that if some Cayley graph $|\Gamma|$ of a finitely generated group $\Gamma$ coarsely embeds into Hilbert space, then the Baum-Connes assembly map for $\Gamma$ is injective\footnote{Under the assumption that $|\Gamma|$ coarsely embeds into Hilbert space, the assumption that $B\Gamma$ is a finite complex was removed by G. Skandalis, J.-L. Tu and G. Yu \cite{Skandalis-Tu-Yu}, using their groupoid approach to CBC.}, i.e. the assembly map $\mu$ embeds the K-homology of the classifying space $B\Gamma$ into the K-theory of the reduced $C^*$-algebra of $\Gamma$. This implies the Novikov conjecture on the homotopy invariance of the higher signatures for $\Gamma$. This was a stunning result, as a strong topological conclusion resulted from a weak metric assumption. 

Finitely generated groups with the Haagerup property coarsely embed into Hilbert space. Indeed if $\alpha$ is a proper isometric action of $\Gamma$ on $\HH$, then for every $x\in\HH$ the orbit map $g\mapsto\alpha(g)x$ is a coarse embedding. 

Using their groupoid approach, Skandalis, Tu and Yu (Theorem 6.1 in \cite{Skandalis-Tu-Yu}) proved the following:

\begin{Thm}\label{coarseimpliesNov} Let $\Gamma$ be a finitely generated group that admits a coarse embedding into Hilbert space. Then the assembly map $\mu_{A,r}$ is injective for every $\Gamma-C^*$-algebra $A$.
\end{Thm}

Lots of finitely generated groups embed coarsely into Hilbert space, as they satisfy the stronger property A (see section \ref{propA} below). Actually it is not even easy to find a {\it bounded geometry space} not embedding coarsely. The most famous example is due to Matousek \cite{Matousek}, and was popularized by Gromov \cite{Gro03}; we will give a proof of a stronger result in Proposition \ref{superexpdonotembed}:

\begin{Prop}\label{expandersdonotembed} Let $X$ be the coarse disjoint union of a family of expanders. Then $X$ does not coarsely embed into Hilbert space.
\end{Prop}

In \cite{Gro03}, Gromov sketched the construction of families of groups containing families of expanders coarsely embedded in their Cayley graphs, which therefore do not embed coarsely into Hilbert space. These are called Gromov's {\it random groups}, or {\it Gromov monsters}. Details of their construction were supplied by Arzhantseva-Delzant \cite{ArzDel}. It was shown by Higson-Lafforgue-Skandalis \cite{HLS} that those groups provide counter-examples to the Baum-Connes conjecture with coefficients (Conjecture \ref{ConjBCcoeff}).

\begin{Thm} Let $\Gamma$ be a Gromov monster. Consider the commutative $C^*$-algebra $A=\ell^\infty(\N,c_0(\Gamma))$, with the natural $\Gamma$-action. Then the Baum-Connes conjecture with coefficients fails for $\Gamma$ and $A$, in the sense that $\mu_{A,r}$ is not onto. 
\end{Thm}

We will come back on those groups in section \ref{propA}, and explain what exactly is needed to get counter-examples to Conjecture \ref{ConjBCcoeff}.

\subsection{Yu's property A: a polymorphous property}

One of the crucial new invariants of metric spaces introduced by G. Yu \cite{YuA} is property A, a non-equivariant form of amenability. Like standard amenability, it has several equivalent definitions. In particular we will see that three concepts from different areas (property A for discrete spaces, boundary amenability from topological dynamics, and exactness from $C^*$-algebra theory) provide one and the same concept when applied to finitely generated groups.

\subsubsection{Property A}\label{propA}

\begin{Def}\label{Yu'sA} Let $(X,d)$ be a discrete metric space. The space $X$ has property A if there exists a sequence $\Phi_n:X\times X\rightarrow \C$ of normalized, positive definite kernels on $X$ such that $\Phi_n$ is supported in some entourage\footnote{Recall from subsection \ref{coarsegroupoid} that an entourage is a subset of $X\times X$ on which $d(.,.)$ is bounded.}, and $(\Phi_n)_{n>0}$ converges to 1 uniformly on entourages for $n\rightarrow\infty$. 
\end{Def}

This is inspired by the following characterization of amenability for a countable group $\Gamma$: the group $\Gamma$ is amenable if and only if there exists a sequence $\varphi_n:\Gamma\rightarrow\C$ of normalized, finitely supported, positive definite functions on $\Gamma$ such that $(\varphi_n)_{n>0}$ converges to 1 for $n\rightarrow\infty$. If this happens and if $\Gamma$ is finitely generated, then $\Phi_n(s,t)=\varphi_n(s^{-1}t)$ witnesses that $|\Gamma|$ has property A. However there are many more examples of finitely generated groups with property A. Other natural examples are provided by linear groups, i.e., subgroups of the group $GL_n(F)$ for some field $F$, this is a result by Guentner, Higson and Weinberger \cite{GuentnerHigsonWeinberger}; this class includes many groups with property (T). The list of classes of groups that satisfy property A also includes one-relator groups, Coxeter groups, groups acting on finite-dimensional CAT(0) cube complexes, and many more.

\begin{Thm}\label{Aimpliescoarse}(see Theorem 2.2 in \cite{YuA}) A discrete metric space with property A admits a coarse embedding into Hilbert space . 
\end{Thm} 

The converse is false: endow $\{0,1\}^n$ with the Hamming distance; then the coarse disjoint union $\coprod_n\{0,1\}^n$ coarsely embeds into Hilbert space but does not have property A, as proved by P. Nowak \cite{Nowak2007}; however this space does not have bounded geometry. For a while, an unfortunate situation was that the only way of disproving property A for a space $X$, was to prove that $X$ has no coarse embedding into Hilbert space (see section \ref{coarseHilbert}). The situation began to evolve with a paper of R. Willett \cite{Willett11} containing a nice result addressing property A directly: the coarse disjoint union of a sequence of finite regular graphs with girth tending to infinity (i.e. graphs looking more and more like trees), does not have property A. On the other hand some of them can be coarsely embedded into Hilbert space, as was shown by Arzhantseva, Guentner and \v Spakula \cite{AGS} using box spaces of the free group. For every group $G$, denote by $G^{(2)}$ the normal subgroup generated by squares in $G$, and define a decreasing sequence of subgroups in $G$ by $G_0=G$ and $G_n=G_{n-1}^{(2)}$. The main result of \cite{AGS} is:

\begin{Thm}\label{AGS} For the free group $\F_k$ of rank $k\geq 2$, with $(\F_k)_{n}$ defined as above, the box space $\coprod_{n>0}\F_k/(\F_k)_n$ does not have property A but is coarsely embeddable into Hilbert space.
\end{Thm}

To summarize the above discussion, we have a square of implications, for finitely generated groups (where CEH stands for {\it coarse embeddability into Hilbert space}):

$$\left. \begin{array}{ccc}
\mbox{amenable} & \Longrightarrow & \mbox{property A} \\ \Downarrow &  & \Downarrow \\  \mbox{Haagerup property}& \Longrightarrow & \mbox{CEH}
\end{array}\right.$$

Let us observe:
\begin{itemize}
\item The top horizontal and the left vertical implications cannot be reversed: indeed a non-abelian free group enjoys both property A and the Haagerup property, but is not amenable.
\item The bottom horizontal implication cannot be reversed: $SL_3(\Z)$ has CEH but, because of property (T), it does not have the Haagerup property. The same example shows that property A does {\it not} imply the Haagerup property.
\end{itemize}

This leaves possibly open the implications ``{\it CEH $\Rightarrow$ property A}'' (which was known to be false for spaces, by Theorem \ref{AGS}), and the weaker implication ``{\it Haagerup property $\Rightarrow$ property A}. The latter was disproved by D. Osajda \cite{Osajda}: he managed, using techniques of graphical small cancellation, to embed sequences of graphs isometrically into Cayley graphs of suitably constructed groups. This way he could prove:

\begin{Thm} There exists a finitely generated group not having property A, but admitting a proper isometric action on a CAT(0) cube complex (and therefore having the Haagerup property).
\end{Thm}

We refer to \cite{KhukhroBourbaki} for a nice survey of that work.

\subsubsection{Boundary amenability}\label{boundary}

Let $\Gamma$ be a countable group; we denote by $Prob(\Gamma)$ the set of probability measures on $\Gamma$, endowed with the topology of pointwise convergence. 

\begin{Def}\begin{enumerate}
\item Let $X$ be a compact space on which $\Gamma$ acts by homeomorphisms. We say that the action $\Gamma\curvearrowright X$ is topologically amenable if there exists a sequence of continuous maps $\mu_n:X\rightarrow Prob(\Gamma)$ which are almost $\Gamma$-equivariant, i.e.
$$\lim_{n\rightarrow\infty}\sup_{x\in X}\|\mu_n(gx)-g\mu_n(x)\|_1=0.$$
\item The group $\Gamma$ is boundary amenable if $\Gamma$ admits a topologically amenable on some compact space.
\end{enumerate}
\end{Def}

For example, the action of $\Gamma$ on a point is topologically amenable if and only if $\Gamma$ is amenable, so boundary amenability is indeed a generalization of amenability. We will see in Theorem \ref{pleindegens} below that, for finitely generated group, boundary amenability is equivalent to property A. Boundary amenability attracted the attention of low-dimensional topologists, so that the following groups were shown to verify it:
\begin{itemize}
\item Mapping class groups, see \cite{Hamenstadt, Kida};
\item $Out(\F_n)$, the outer automorphism group of the free group, see \cite{BestvinaHorbezGuirardel}.
\end{itemize}

\subsubsection{Exactness}\label{exact}

For $C^*$-algebras $A,B$, denote by $A\otimes_{\rm min}B$ (resp. $A\otimes_{\rm max} B$) the minimal (resp. maximal) tensor product. Recall that $A$ is {\it nuclear} if the canonical map $A\otimes_{\rm max} B\rightarrow A\otimes_{\rm min} B$ is an isomorphism for every $C^*$-algebra $B$, and that $A$ is {\it exact} if the minimal tensor product with $A$ preserves short exact sequences of $C^*$-algebras. As the maximal tensor product preserves short exact sequences, nuclear implies exact. 

A classical result of Lance says that, for discrete groups, a group $\Gamma$ is amenable if and only if $C^*_r(\Gamma)$ is nuclear. It turns out that, for exactness we have an analogous result merging this section with subsections \ref{propA} and \ref{boundary}; it is a combination of results by Anantharaman-Delaroche and Renault \cite{AnanDelRen}, Guentner and Kaminker \cite{GK02}, Higson and Roe \cite{Higson-Roe}, and Ozawa \cite{Oz00}.

\begin{Thm}\label{pleindegens} For a finitely generated group $\Gamma$, the following are equivalent:
\begin{enumerate}
\item $\Gamma$ has property A;
\item $\Gamma$ is boundary amenable;
\item $C^*_r(\Gamma)$ is exact.
\end{enumerate}
\end{Thm}

Combining with Theorems \ref{Aimpliescoarse} and \ref{coarseimpliesNov}, we get immediately:

\begin{Cor} If $\Gamma$ is a finitely generated group with property A, then for every $\Gamma-C^*$-algebra $A$ the assembly map $\mu_{A,r}$ is injective.
\end{Cor}

As a consequence of Theorem \ref{pleindegens}, for a finitely generated group $\Gamma$, nuclearity and exactness of $C^*_r(\Gamma)$ are quasi-isometry invariants (which is by no means obvious on the analytical definitions). An interesting research question is: which other properties of $C^*_r(\Gamma)$ are quasi-isometry invariants of $\Gamma$?

We now explain how the lack of exactness of $C^*_r(\Gamma)$, when detected at the level of K-theory, leads to counterexamples to Conjecture \ref{ConjBCcoeff}.

\begin{Def} A $C^*$-algebra $C$ is {\it half-K-exact} if for any short exact sequence $0\rightarrow J\rightarrow A\rightarrow B\rightarrow 0$ of $C^*$-algebras, the sequence $$K_*(J\otimes_{\rm min}C)\rightarrow K_*(A\otimes_{\rm min}C)\rightarrow K_*(B\otimes_{\rm min} C)$$ is exact in the middle term.
\end{Def}

The following statement is an unpublished result by N. Ozawa (see however Theorem 5.2 in \cite{ThesisOzawa}).

\begin{Thm}\label{monsters} Gromov monsters are not half-K-exact.
\end{Thm}

\begin{proof}[Proof] Let $\Gamma$ be a Gromov monster. So there is a family $(X_k)_{k>0}$ of $d$-regular expanders which coarsely embeds in $\Gamma$, i.e. there exists a family of maps $f_k:X_k\rightarrow \Gamma$ such that, for $x_k,y_k\in X_k$, we have $d_{X_k}(x_k,y_k)\rightarrow+\infty \Longleftrightarrow d_\Gamma(f_k(x_k),f_k(y_k))\rightarrow+\infty$. We will need below a consequence of this fact: there exists a constant $K>0$ such that the fiber $f_k^{-1}(g)$ has cardinality at most $K$, for every $k>0$ and every $g\in\Gamma$. (Indeed, first observe that, as a consequence of the coarse embedding, there exists $R>0$ such that, for every $k$ and $g$, we have $d_{X_k}(x,y)\leq R$ for every $x,y\in f_k^{-1}(g)$; then use the bounded geometry of the family $(X_k)_{k>0}$: we may for example take for $K$ the cardinality of a ball of radius $R$ in the $d$-regular tree.)

We now start the proof really. Denote by $n_k$ the number of vertices of $X_k$, and form the product of matrix algebras $M=\prod_{k>0} M_{n_k}(\C)$ together with its ideal $J=\oplus_{k>0} M_{n_k}(\C)$. We are going to show that the sequence
$$K_0(J\otimes_{\rm min}C^*_r(\Gamma))\rightarrow K_0(M\otimes_{\rm min} C^*_r(\Gamma))\rightarrow K_0((M/J)\otimes_{\rm min} C^*_r(\Gamma))$$
is not exact at its middle term. Let us identify $M_{n_k}(\C)$ with $End(\ell^2(X_k)$ via the canonical basis. 

We first define an injective homomorphism $\iota_k:M_{n_k}(\C)\rightarrow M_{n_k}(\C)\otimes C^*_r(\Gamma)$ by $\iota_k(E_{xy})=E_{xy}\otimes f_k(x)^{-1}f_k(y)$, where $E_{xy}$ is the standard set of matrix units in $End(\ell^2(X_k))$. We then use an idea similar to the one in the proof of Proposition \ref{inj/surj}. Let $\Delta_k$ be the combinatorial Laplace operator on $X_k$, let $p_k$ be the projection on its 1-dimensional kernel: recall that $(p_k)_{xy}=\frac{1}{n_k}$ for every $x,y\in X_k$. Then $\Delta:=(\iota_k(\Delta_k))_{k>0}\in M\otimes_{\rm min}C^*_r(\Gamma)$ has 0 as an isolated point in its spectrum, as the $X_k$'s are a family of expanders. The spectral projection associated with 0 is $q=(\iota_k(p_k))_{k>0}$. The class $[q]\in K_0(M\otimes_{\rm min}C^*_r(\Gamma))$ will witness the desired non-exactness. 

Let $\pi:M\rightarrow M/J$ denote the quotient map. To show that $q$ is in the kernel of $\pi\otimes_{\rm min}Id$, consider the conditional expectation $\mathbb{E}_M=Id_M\otimes\tau: M\otimes_{\rm min}C^*_r(\Gamma)\rightarrow M$, where $\tau$ denotes the canonical trace on $C^*_r(\Gamma)$. We have 
$$(Id_{n_k}\otimes \tau)(\iota_k(p_k))_{xy}=\left\{\begin{array}{ccc}\frac{1}{n_k} & if  & f_k(x)=f_k(y) \\  0& if & f_k(x)\neq f_k(y)\end{array}\right.$$
So the operator norm of $(Id_{n_k}\otimes \tau)(\iota_k(p_k))$ satisfies:
$$\|(Id_{n_k}\otimes \tau)(\iota_k(p_k))\|\leq \frac{1}{n_k}\cdot\max_{x\in X_k}|f_k^{-1}(f_k(x))|\leq \frac{K}{n_k},$$
where $K$ is the constant introduced at the beginning of the proof. As a consequence $\mathbb{E}_M(q)$ belongs to $J$ and 
$${\scriptstyle 0=\pi(\mathbb{E}_M(q))=\pi((Id_M\otimes\tau)(q))=(Id_{M/J}\otimes \tau)((\pi\otimes_{\rm min}Id)(q))=\mathbb{E}_{M/J}((\pi\otimes_{\rm min}Id)(q));}$$
by faithfulness of $\mathbb{E}_{M/J}$ we get $(\pi\otimes_{\rm min}Id)(q)=0$.

It remains to show that $[q]$ is not in the image of $K_0(J\otimes_{\rm min}C^*_r(\Gamma))$ in $K_0(M\otimes_{\rm min}C^*_r(\Gamma)$. For this, denote by $\sigma_k:M\otimes_{\rm min}C^*_r(\Gamma)\rightarrow M_{n_k}(\C)\otimes C^*_r(\Gamma)$ the projection on the $k$-th factor. Because $K_0(J\otimes_{\rm min}C^*_r(\Gamma))=\oplus_{k>0} K_0(M_{n_k}(\C)\otimes C^*_r(\Gamma))$, for every $x\in K_0(J\otimes_{\rm min}C^*_r(\Gamma))$ we have
$(\sigma_k\otimes\tau)(x)=0$ for $k$ large enough. On the other hand $(\sigma_k\otimes\tau)(q)>0$ for every $k>0$.
\end{proof}

The following result may be extracted from \cite{HLS}, where it is not stated explicitly.

\begin{Thm}\label{BCnotonto} Let $\Gamma$ be a countable group. If $\Gamma$ is not half-K-exact, then there is a $C^*$-algebra $C$ with {\it trivial} $\Gamma$-action such that the assembly map $$\mu_{C,r}: K^{top}_*(\Gamma,C)\rightarrow K_*(C^*_r(\Gamma,C))$$ is NOT onto.
\end{Thm}

The proof will be given below. Combining with Theorem \ref{monsters} and its proof, we immediately get the following:

\begin{Cor} If $\Gamma$ is a Gromov monster, there exists a non-commutative $C^*$-algebra $C$ with {\it trivial} $\Gamma$-action such that the assembly map $$\mu_{C,r}: K^{top}_*(\Gamma,C)\rightarrow K_*(C^*_r(\Gamma,C))$$ is NOT onto.
\hfill$\square$
\end{Cor}

It seems this is as close as one can get to a counter-example to the Baum-Connes conjecture without coefficients (conjecture \ref{ConjBC}).

To prove Theorem \ref{BCnotonto}, we start by some recalls about mapping cones.

\begin{Def} Let $\beta:A\rightarrow B$ be a homomorphism of $C*$-algebras. The {\it mapping cone} of $\beta$ is the $C^*$-algebra $C(\beta) =\{(a,f)\in A\oplus C([0,1],B): f(0)=\beta(a), f(1)=0\}$.
\end{Def}

Consider now the following situation, with three $C^*$-algebras $J,A,B$ and homomorphisms:
\begin{itemize}
\item $\alpha:J\rightarrow A$, injective;
\item $\beta:A\rightarrow B$, surjective, such that $\beta\circ\alpha=0$.
\end{itemize}
We then have an inclusion $\gamma:J\rightarrow C(\beta):j\mapsto(\alpha(j),0)$.

\begin{Lem}\label{mappingcone}\begin{enumerate}
\item If $Im(\alpha)=\ker(\beta)$, i.e. the sequence $0\rightarrow J\rightarrow A\rightarrow B\rightarrow 0$ is exact, then $\gamma_*:K_*(J)\rightarrow K_*(C(\beta))$ is an isomorphism.
\item If $\gamma_*$ is an isomorphism, then the sequence $K_*(J)\stackrel{\alpha_*}{\rightarrow}K_*(A)\stackrel{\beta_*}{\rightarrow}K_*(B)$ is exact.
\item $\gamma_*$ is an isomorphism if and only if $K_*(C(\gamma))=0$.
\end{enumerate}
\end{Lem}

\begin{proof}[Proof of lemma \ref{mappingcone}] \begin{enumerate}
\item See Exercise 6.N in \cite{WeggeOlsen}.
\item Set $I=\ker(\beta)$ and $\tilde{\gamma}:I\rightarrow C(\beta):x\mapsto (x,0)$, so that $\gamma=\tilde{\gamma}\circ\alpha$. Since $\tilde{\gamma}_*$ is an isomorphism by the previous point, and $\gamma_*$ is an isomorphism by assumption, we get that $\alpha_*: K_*(J)\rightarrow K_*(I)$ is an isomorphism. Since the sequence $\xymatrix{K_*(I)\ar[r]& K_*(A)\ar[r]^{\beta_*}&K_*(B)}$ is exact, so is the sequence $$\xymatrix{K_*(J)\ar[r]^{\alpha_*} &K_*(A)\ar[r]^{\beta_*}&K_*(B).}$$
\item Since $\gamma$ is injective, we may identify the mapping cone $C(\gamma)$ with $\{f\in C([0,1],C(\beta)):f(0)\in\gamma(J),f(1)=0\}$. By evaluation at 0, we get a short exact sequence 
$$\xymatrix{0\ar[r]& C_0]0,1[\otimes C(\beta)\ar[r]& C(\gamma)\ar[r]&J\ar[r] &0.}$$
In the associated 6-terms exact sequence in K-theory, the use of Bott periodicity to identify $K_*(C_0]0,1[\otimes C(\beta))$ with $K_*(C(\beta))$ allows to identify the connecting maps with $\gamma_*$, so the result follows.
\end{enumerate}
\end{proof}

\begin{proof}[Proof of Theorem \ref{BCnotonto}] Since $\Gamma$ is not half-K-exact, we find a short exact sequence $\xymatrix{0\ar[r] &J\ar[r]^{\alpha}&A\ar[r]^{\beta}& B\ar[r] &0}$ such that 
\begin{equation}\label{nonKexact}
\xymatrix{
K_*(J\otimes_{\rm min}C^*_r(\Gamma))\ar[r]^{\scriptscriptstyle{(\alpha\otimes_{\rm min}Id)_*}} &K_*(A\otimes_{\rm min} C^*_r(\Gamma)) \ar[r]^{\scriptscriptstyle(\beta\otimes_{\rm min}Id)_*}&K_*(B\otimes_{\rm min}C^*_r(\Gamma))}
\end{equation}
is not exact in the middle term. As above, define the mapping cone $C(\beta)$ and the inclusion $\gamma:J\rightarrow C(\beta)$. Set $C=C(\gamma)$, with trivial $\Gamma$-action. We prove in three steps that the assembly map $\mu_{C,r}$ with coefficients in $C$, is not onto.
\begin{itemize}
\item $K_*(C\otimes_{\rm min}C^*_r(\Gamma))=K_*(C(\gamma\otimes_{\rm min}Id))$ is non-zero: this follows from non-exactness of the sequence (\ref{nonKexact}) together with the two last statements of lemma \ref{mappingcone}.
\item $K_*(C\otimes_{\rm max}C^*_{\rm max}(\Gamma))=K_*(C(\gamma\otimes_{\rm max}Id))$ is zero: this follows from exactness of
$$0\rightarrow J\otimes_{\rm max}C^*_{\rm max}(\Gamma)\rightarrow A\otimes_{\rm max}C^*_{\rm max}(\Gamma)\rightarrow B\otimes_{\rm max}C^*_{\rm max}(\Gamma)\rightarrow 0$$
together with the first and last statements of lemma \ref{mappingcone}.
\item The assembly map $\mu_{C,r}: K^{top}_* (\Gamma,C)\rightarrow K_*(C^*_r(\Gamma,C))=K_*(C\otimes_{\rm min}C^*_r(\Gamma))$ is zero, and therefore is not onto: this is because, as explained in the beginning of section \ref{KvsH}, $\mu_{C,r}$ factors through $$\mu_{C,{\rm max}}: K^{top}_*(\Gamma,C)\rightarrow K_*(C^*_{\rm max}(\Gamma,C))=K_*(C\otimes_{\rm max}C^*_{\rm max}(\Gamma)),$$ and this is the zero map.
\end{itemize}
\end{proof}

\subsection{Applications of strong property (T)}

\subsubsection{Super-expanders}\label{superexp}

A Banach space is {\it super-reflexive} if it admits an equivalent norm making it uniformly convex. As mentioned in subsection \ref{positiveCBC} Kasparov and Yu proved in \cite{KaYu} that if a discrete metric space with bounded geometry coarsely embeds into some super-reflexive space, then the coarse assembly map $\mu_X$ is injective. Since families of expanders do not embed coarsely into Hilbert space, by Proposition \ref{expandersdonotembed}, it is natural to ask: {\it is there a family of expanders that admits a coarse embedding into some super-reflexive Banach space?} This is a very interesting open question. However, certain families of expanders are known {\it not} to embed coarsely into any super-reflexive Banach space, and we wish to explain the link with strong property (T) from section \ref{strong(T)}.

Let $(X_k=(V_k,E_k))_{k>0}$ be a family of finite, connected, $d$-regular graphs with $\lim_{k\rightarrow\infty}|V_k|=+\infty$, and let $B$ be a Banach space. We say that $(X_k)_{k>0}$ {\it satisfies a Poincar\'e inequality with respect to $B$} if there exists $C=C(B)>0$ such that for every map $f:\coprod_{k>0}X_k\rightarrow B$ we have:
\begin{equation}\label{PoincareB}
\frac{1}{|V_k|^2}\sum_{x,y\in V_k}\|f(x)-f(y)\|_B^2\leq \frac{C}{|V_k|}\sum_{x\sim y}\|f(x)-f(y)\|_B^2.
\end{equation}
Compare with inequality (\ref{Poincare}), which is the Poincar\'e inequality with respect to Hilbert spaces. In view of Proposition \ref{expandersPoincare}, the following result implies Proposition \ref{expandersdonotembed}.

\begin{Prop}\label{superexpdonotembed} Assume that the family $(X_k)_{k>0}$ satisfies a Poincar\'e inequality with respect to the Banach space $B$. Then the coarse disjoint union $X$ of the $X_k$'s, admits no coarse embedding into $B$.
\end{Prop}

\begin{proof}[Proof] Suppose by contradiction that there exists a coarse embedding $f:X\rightarrow B$, with control functions $\rho_\pm$. Then, using $\|f(x)-f(y)\|_B\leq \rho_+(1)$ for $x\sim y$ in any $X_k$, we get for every $k>0$:
$$\frac{1}{|V_k|^2}\sum_{x,y\in V_k}\rho_-(d(x,y))^2\leq\frac{1}{|V_k|^2}\sum_{x,y\in V_k}\|f(x)-f(y)\|^2_B\leq\frac{C}{|V_k|}\sum_{x\sim y}\|f(x)-f(y)\|^2_B$$
$$\leq \frac{2C|E_k|\rho_+(1)^2}{|V_k|}=dC\rho_+(1)^2,$$
where the second inequality is the Poincar\'e inequality and the final equality is $|E_k|=\frac{d|V_k|}{2}$. Set $M=dC\rho_+(1)^2$; since the mean of the quantities $\rho_-(d(x,y))^2$ is at most $M$, this means that for at least half of the pairs $(x,y)\in V_k\times V_k$, we have $\rho_-(d(x,y))^2\leq 2M$, for every $k>0$. Since $\lim_{t\rightarrow\infty}\rho_-(t)=+\infty$, we find a constant $N>0$ such that, for every $k>0$ and at least half of the pairs $(x,y)\in V_k\times V_k$, we have $d(x,y)\leq N$. But as $X_k$ is $d$-regular, the cardinality of a ball of radius $N$ is at most $(d+1)^N$, so the cardinality of the set of pairs $(x,y)\in V_k\times V_k$ with $d(x,y)\leq N$, is at most $|V_k|(d+1)^N$. For $k\gg 0$, this is smaller than $\frac{|V_k|^2}{2}$, and we have reached a contradiction.
\end{proof}

\begin{Def} A sequence $(X_k)_{k>0}$ of finite, connected, $d$-regular graphs with $\lim_{k\rightarrow\infty}|X_k|=+\infty$, is a family of super-expanders if, for any super-reflexive Banach space $B$, the sequence $(X_k)_{k>0}$ satisfies the Poincar\'e inequality (\ref{PoincareB}) with respect to $B$.
\end{Def}

It follows from Proposition \ref{expandersPoincare} that, assuming they do exist, super-expanders are expanders, and from Proposition \ref{superexpdonotembed} that super-expanders do not admit a coarse embedding into {\it any} super-reflexive Banach space. Lafforgue's construction of super-expanders in \cite{LafforgueT,LafforgueFourier}, following a suggestion by A. Naor, answered a question from \cite{KaYu}:

\begin{Thm} Let $F$ be a non-Archimedean local field, let $G$ be a simple algebraic group of higher rank defined over $F$, and let $G(F)$ be the group of $F$-rational points of $G$. Let $\Gamma$ be a lattice in $G(F)$, fix any filtration $(N_k)_{k>0}$ of $\Gamma$. Then the box space $\coprod_{k>0} Cay(\Gamma/N_k,S)$ (see Definition \ref{box}) is a family of super-expanders.
\end{Thm}

\begin{proof}[Proof] Write $X_k=:Cay(\Gamma/N_k,S)$. Let $B$ be a super-reflexive Banach space. The goal is to show that the Poincar\'e inequality \ref{PoincareB} is satisfied.
 \begin{enumerate}
\item Let $B_k$ be the space of functions $X_k\rightarrow B$, with norm $\|f\|_{B_k}^2=\frac{1}{|X_k|}\sum_{x\in X_k}\|f(x)\|_B^2$. For $f\in B_k$, set $m_f=\frac{1}{|X_k}\sum_{x\in X_k}f(x)\in B$. Then\footnote{Note typos regarding inequality \ref{average} in Proposition 5.2 of \cite{LafforgueT} and in Proposition 5.5 of \cite{LafforgueFourier}: $\leq \frac{4}{|X_k|}$ is erroneously written as $=\frac{2}{|X_k|}$.}
\begin{equation}\label{average}
\frac{1}{|X_k|^2}\sum_{x,y\in X_k}\|f(x)-f(y)\|^2_B\leq \frac{4}{|X_k|}\sum_{x\in X_k}\|f(x)-m_f\|^2_B.
\end{equation}
To see this: by translation we may assume $m_f=0$. Then by the triangle inequality:
$$\|f(x)-f(y)\|_B^2\leq(\|f(x)\|_B+\|f(y)\|_B)^2\leq 2(\|f(x)\|_B^2+\|f(y)\|_B^2),$$
and inequality \ref{average} follows by averaging over $X_k\times X_k$.

\item Let $\pi_k$ be the natural isometric representation of $\Gamma$ on $B_k$. As $\Gamma$ acts transitively on $X-K$, the fixed point space of $\Gamma$ in $B_k$ is the space of constant functions. Now strong property (T) for representations in a Banach space is defined by analogy with Definition \ref{StrongT}, by replacing Hilbert space by a suitable class of Banach spaces: it posits the existence of a Kazhdan projection projecting onto the fixed point space, for any representation in a suitable class. It turns out that the lattice $\Gamma$ has strong property (T) for isometric representations in super-reflexive Banach spaces: this is due to Lafforgue \cite{LafforgueT,LafforgueFourier} when $G(F)$ contains $SL_3(F)$, and to Liao \cite{Benben14} in general. So, denoting by $\mathcal{C}_{0,1}(\Gamma)$ the Banach algebra completion of $\C\Gamma$ with respect to isometric $\Gamma$-representations in the spaces $(B_k)_{k>0}$, there exists an idempotent $p\in  \mathcal{C}_{0,1}(\Gamma)$ such that in particular $\pi_k(p)f=m_f$ for every $f\in B_k$. Inequality \ref{average} is then reformulated
\begin{equation}\label{refaverage}
\frac{1}{|X_k|^2}\sum_{x,y\in X_k}\|f(x)-f(y)\|^2_B\leq 4\|f-\pi_k(p)f\|_{B_k}^2.
\end{equation}

\item Let $q\in \C\Gamma$ be an element such that $\|p-q\|_{\mathcal{C}_{0,1}(\Gamma)}<\frac{1}{2}$ and $\sum_\gamma q(\gamma)=1$. Then
\begin{align*}
\|\pi_k(p)f-\pi_k(q)f\|_{B_k}&=\|(\pi_k(p)-\pi_k(q))(f-m_f)\|_{B_k}\\
&\leq\frac{1}{2}\|f-m_f\|_{B_k}=\frac{1}{2}\|f-\pi_k(p)f\|_{B_k};
\end{align*}
but $\|f-\pi_k(p)f\|_{B_k}\leq\|f-\pi_k(q)f\|_{B_k}+\|\pi_k(q)f-\pi_k(p)f\|_{B_k}$ by the triangle inequality, so
$$\|f-\pi_k(p)f\|_{B_k}\leq 2\|f-\pi_k(q)f\|_{B_k}, $$
that we plug in (\ref{refaverage}).

\item Finally it is easy to see that there exists a constant $C_1>0$, only depending on $q$, such that for every $k>0$:
$$\|f-\pi_k(q)f\|_{B_k}^2\leq \frac{C_1}{|X_k|}\sum_{x\sim y}\|f(x)-f(y)\|_B^2.$$
\end{enumerate}
\end{proof}

Later on, other constructions of super-expanders were provided:
\begin{itemize}
\item by M. Mendel and A. Naor \cite{MendelNaor}, using zig-zag products;
\item independently by D. Sawicki \cite{SawickiSuper} and by T. de Laat and F. Vigolo \cite{deLaatVigolo}, using warped cones, as defined in section \ref{warped}: the constructions appeal to actions on manifolds of groups with strong property (T).
\end{itemize}

\subsubsection{Zimmer's conjecture}

A striking, unexpected application of Lafforgue's strong property (T) from section \ref{strong(T)} is the recent solution of Zimmer's conjecture on actions of higher rank lattices on manifolds. Roughly speaking, Zimmer's conjecture claims that a lattice $\Gamma$ in a higher rank simple Lie group $G$, has only finite actions on manifolds of dimension small enough (relative to data only associated with $G$). Somewhat more precisely, in this section:

\begin{itemize}
\item {\it higher rank} means that the real rank of $G$ is at least 2 (think of $G=SL_n(\R)$, for $n\geq 3$; or $G=Sp_{2n}(\R)$, for $n\geq 2$);
\item {\it manifold} means a smooth closed manifold $M$;
\item {\it action of $\Gamma$ on $M$} means an action by diffeomorphisms of class at least $C^2$;
\item {\it a finite action of $\Gamma$} is one that factors through a finite quotient of $\Gamma$.
\end{itemize}
It remains to explain ``{\it dimension small enough}'' and for this we will restrict to $G=SL_n(\R), n\geq 3$. For the general case, we refer to Conjecture 1.2 in \cite{BFH}. For the original statement by R.J. Zimmer, see \cite{Zimmer1987}.

If $\Gamma$ is a lattice in $SL_n(\R)$, we may let it act linearly on $\R^n$. So we get an infinite action of $\Gamma$ on the $(n-1)$-dimensional projective space $P^{n-1}(\R)$; we observe that this action has no invariant volume form. On the other hand, $\Gamma=SL_n(\Z)$ has an infinite action on the $n$-dimensional torus $\mathbf{T}^n=\R^n/\Z^n$, this one clearly preserving a volume form. Zimmer's conjecture basically claims that those examples are of minimal dimension among non-finite actions. Precisely, Zimmer's conjecture for cocompact lattices in $SL_n(\R)$, is now the following result by A. Brown, D. Fisher and S. Hurtado (Theorem 1.1 in \cite{BFH}):

\begin{Thm}\label{Zimmer} Let $\Gamma$ be a cocompact lattice in $SL_n(\R), n\geq 3$. 
\begin{enumerate}
\item If $\dim M<n-1$, any action of $\Gamma$ on $M$ is finite.
\item If $\dim M<n$, any volume-preserving action of $\Gamma$ on $M$ is finite.
\end{enumerate}
\end{Thm}

Let us give a rough sketch, in 3 steps, of the proof of the first statement in Theorem \ref{Zimmer}. So we consider $\alpha:\Gamma\rightarrow Diff(M)$, with $\dim M<n-1$, we must show that $\alpha$ is finite.
\begin{itemize}
\item Let $\alpha:\Gamma\rightarrow Diff^\infty(M)$ be a homomorphism (for simplicity we assume that $\Gamma$ acts by $C^\infty$ diffeomorphisms). Fix any Riemannian structure on $M$. For $x\in M,\gamma\in\Gamma$, denote by $D_x\alpha(\gamma)$ the differential of $\alpha(\gamma)$ at $x$. Then $\alpha$ has {\it uniform subexponential growth of derivatives}, i.e. for every $\varepsilon>0$, there exists $C\geq 1$ such that for every $\gamma\in\Gamma$:
\begin{equation}\label{subexponential}
\sup_{x\in M} \|D_x\alpha(\gamma)\|\leq Ce^{\varepsilon\ell(\gamma)},
\end{equation}
where $\ell$ denotes the word length with respect to a fixed finite generating set of $\Gamma$. Morally, this means that generators of $\Gamma$ are close to being isometries of $M$.
\item A Riemannian structure of class $C^k$ on $M$ is a $C^k$ section of the symmetric square $S^2(TM)$ of the tangent bundle $TM$ of $M$. Via $\alpha$, the group $\Gamma$ acts on $C^k$ Riemannian structures on $M$ and this defines a homomorphism $\alpha_\sharp$ from $\Gamma$ to the group of invertibles in the algebra $\mathcal{B}(C^k(S^2(TM)))$ of bounded operators on $C^k(S^2(TM))$. At this point we introduce the Hilbert space $\HH^k$ which is the Sobolev space of sections of $S^2(TM)$ with weak $k$-th derivative being $L^2$. By the Sobolev embedding theorem, we have $\HH^k\subset C^\ell(S^2(TM))$ for $k\gg\ell$. If $\alpha$ satisfies \ref{subexponential}, then $\alpha_\sharp$ has slow exponential growth: for all $\varepsilon>0$, there exists  $C\geq 1$ such that for all $g\in G$:
$$\|\alpha_\sharp(g)\|_{\HH^k\rightarrow\HH^k}\leq Ce^{\varepsilon\ell(g)}.$$

It is here that strong property (T) enters the game; it is however needed in a form both stronger and more precise than in Definition \ref{StrongT}, namely: there exists a constant $\delta>0$ and a sequence $\mu_n$ of probability measures supported in the balls $B_\ell(n)$ of radius $n$ in $\Gamma$, such that for all $C>0$ and any representation $\pi$ on a Hilbert space with $\|\pi(g)\|\leq Ce^{\delta\ell(g)}$, the operators $(\pi(\mu_n))_{n>0}$ converge exponentially quickly to a projection $P_\infty$ onto the space of invariant vectors. That is, there exists $K>0$ and $0<\lambda<1$, independent of $\pi$, such that $\|\pi(\mu_i)-P_\infty\|<K\cdot\lambda^i$. Theorem 6.3 in \cite{BFH} explains how to deduce the extra desired features (exponentially fast convergence and approximation by positive measures rather than signed measures) from the {\it proofs} of Theorem \ref{StrongThigher} by Lafforgue, de Laat and de la Salle \cite{LafforgueT, delaSalleb, dLdlS}\footnote{The subtlety here is that, as lucidly explained in \cite{dlS16}, Definition \ref{StrongT} for an arbitrary finitely generated group is equivalent to the existence of a sequence of {\it signed} probability measures as above.}.

Coming back to our sketch of proof of Theorem \ref{Zimmer}:

\begin{Prop}\label{invmetric} $\alpha(\Gamma)$ preserves some $C^\ell$ Riemannian structure on $M$.
\end{Prop}
\begin{proof}[Proof] We will apply the above form of strong property (T) to the representation $\alpha_\sharp$. Let $(\mu_n)_{n>0}$ be the sequence of probability measures as above, set $P_n=\alpha_\sharp(\mu_n)$, so that $\|P_i-P_\infty\|_{\HH^k\rightarrow\HH^k}<K\cdot\lambda^i$.

We start with any smooth Riemannian metric $g$ on $M$, view it as an element in $\HH^k$, and apply the averaging operators $P_i$: then $g_i=:P_i(g)$. We set $g_\infty=\lim_{i\rightarrow\infty} g_i$, so that $g_\infty$ is $\alpha_\sharp(\Gamma)$-invariant in $\HH^k$, hence also in $C^\ell(S^2(TM))$. We have $g_\infty(v,v)\geq 0$ for every $v\in TM$, as $g_\infty$ is a limit of positive-definite forms, but we must show that $g_\infty$ is positive-definite, i.e. $g_\infty(v,v)>0$ for every unit vector $v\in TM$. By the previous point (subexponential growth of derivatives), taking $e^\varepsilon=\lambda^{-\frac{1}{3}}$, we have for every $\gamma\in\Gamma$:
\begin{align*}
C^2\lambda^{-\frac{2\ell(\gamma)}{3}}\geq \|D\alpha(\gamma^{-1})\|^2&=\sup_{u\in TM}\frac{g(u,u)}{g(D\alpha(\gamma)(u),D\alpha(\gamma)(u))}\\
&\geq\frac{1}{g(D\alpha(\gamma)(v),D\alpha(\gamma)(v))}
\end{align*}
hence, if $\ell(\gamma)\leq i$:
$$g(D\alpha(\gamma)(v),D\alpha(\gamma)(v))\geq\frac{1}{C^2} \cdot\lambda^{\frac{2\ell(\gamma)}{3}}\geq\frac{1}{C^2}\cdot\lambda^{\frac{2i}{3}}$$
Since $\mu_i$ is supported in the ball of radius $i$ of $\Gamma$, we have 
$$g_i(v,v)\geq\frac{1}{C^2}\cdot\lambda^{\frac{2i}{3}}.$$
On the other hand $|g_\infty(v,v)-g_i(v,v)|\leq K\cdot\lambda^i$, hence
$$g_\infty(v,v)\geq g_i(v,v)-K\cdot\lambda^i\geq \frac{1}{C^2}\cdot\lambda^{\frac{2i}{3}}- K\cdot\lambda^i,$$
which is positive for $i\gg0$.
\end{proof}

\item Set $m=\dim M$. Let $g$ be an $\alpha(\Gamma)$-invariant $C^\ell$ metric on $M$, so that $\alpha(\Gamma)$ is a subgroup of the isometry group $K=:Isom(M,g)$. Now $K$ is a compact Lie group, of dimension at most $\frac{m(m+1)}{2}$. Assuming by contradiction that $\alpha(\Gamma)$ is infinite, a suitable version of Margulis' super-rigidity says that the Lie algebra $\mathfrak{su}_n$, which is the compact real form of $\mathfrak{sl}_n(\R)$, must embed into the Lie algebra of $K$. Counting dimensions we get
$$n^2-1=\dim\mathfrak{su}_n\leq\dim K\leq \frac{m(m+1)}{2},$$
contradicting the assumption $m<n-1$. So $\alpha$ is finite.
\end{itemize}

More recently in \cite{BFH2}, Brown, Fisher and Hurtado verified Zimmer's conjecture for $SL_3(\Z)$. For this they had to appeal to de la Salle's result \cite{delaSalleb} that strong property (T) holds for arbitrary lattices in higher rank simple Lie groups.

It is expected that in 2019, A. Brown, D. Fisher and S. Hurtado, with the help of D. Witte-Morris, will complete a proof of Zimmer's conjecture for any lattice in any higher rank simple Lie group.

\addcontentsline{toc}{section}{References}

\bibliographystyle{alpha}
\bibliography{bibliographie-fusion}

\noindent
Authors addresses:

\medskip
\noindent
Maria Paula Gomez Aparicio\\
Institut de Math\'ematiques B\^atiment 307\\
Facult\'e des Sciences d'Orsay\\
Universit\'e Paris-Sud\\
F-91405 Orsay Cedex - France\\
maria.gomez@math.u-psud.fr

\bigskip
Pierre Julg\\
Institut Denis Poisson\\
Universit\'e d'Orl\'eans\\
Collegium Sciences et Techniques\\
B\^atiment de math\'ematiques\\
Rue de Chartres B.P. 6759\\
F-45067 Orl\'eans Cedex 2 - France\\
pierre.julg@univ-orleans.fr

\bigskip
\noindent
Alain Valette\\
Institut de Math\'ematiques\\
Universit\'e de Neuch\^atel\\
11 Rue Emile Argand\\
CH-2000 Neuch\^atel - Switzerland\\
alain.valette@unine.ch

\end{document}